\documentclass[10pt]{amsart}
\usepackage{amsmath, amsfonts, amssymb,  mathrsfs, color, amsthm, epsfig, color, graphics, oubraces}
  \usepackage[all]{xy}
 \usepackage[latin1]{inputenc}
\usepackage{hyperref}

\newtheorem{definition}{\bf Definition}[section]
\newtheorem{prop}[definition]{{\bf Proposition}}
\newtheorem{rem}[definition]{{\bf Remark}}
\newtheorem{lemma}[definition]{{\bf Lemma}}
\newtheorem{cor}[definition]{{\bf Corollary}}

\newtheorem{theorem}[definition]{{\bf Theorem}}

\newenvironment{my_enumerate}{
\begin{enumerate}
  \setlength{\itemsep}{3pt}
}{\end{enumerate}}

\DeclareMathOperator{\sumperp}{\stackrel{\perp}{\oplus}}

\DeclareMathOperator*{\argmin}{argmin} 

\makeatletter
\renewcommand{\theequation}{\ifnum \c@section>\z@ \thesection.\fi
	\@arabic\c@equation}
\@addtoreset{equation}{section}
\makeatother

\newcommand{\F}{\mathcal F}
\newcommand{\dx}{\,{\rm d}x}
\newcommand{\ds}{\,{\rm d}s}
\newcommand{\dt}{\,{\rm d}t}
\newcommand{\forallt}{\qquad\text{for all }}

\renewcommand{\Re}{{\rm Re}\,}
\renewcommand{\Im}{{\rm Im}\,}
\title[Vorticity and stream function formulations for the 2D NS equations in a bounded domain]{Vorticity and stream function formulations for the 2D Navier-Stokes equations in a bounded domain}
\author{Julien Lequeurre}
\address{Institut \'Elie Cartan de Lorraine, Universit\'e de Lorraine, Metz - \'Equipe-projet SPHINX Inria Nancy-Grand Est}
\email{julien.lequeurre@univ-lorraine.fr}
\author{Alexandre Munnier}
\address{Institut \'Elie Cartan de Lorraine, Universit\'e de Lorraine, Nancy - \'Equipe-projet SPHINX Inria Nancy-Grand Est}
\email{alexandre.munnier@univ-lorraine.fr}

\begin{document}
\date{\today}
\begin{abstract}
The main purpose of this work is to provide a Hilbertian functional framework  for the analysis of the planar Navier-Stokes (NS) equations either in vorticity 
or in stream function
formulation. The fluid is assumed to occupy a bounded possibly multiply connected domain. The velocity field satisfies either homogeneous (no-slip boundary 
conditions) or prescribed 
Dirichlet boundary conditions. We prove that the analysis of the 2D Navier-Stokes equations can be carried out in terms of the so-called nonprimitive 
variables only (vorticity field and stream function) without resorting to the classical NS theory (stated in primitive variables, i.e. velocity and pressure fields).
Both approaches (in primitive and nonprimitive variables) are shown to be equivalent for weak (Leray) and strong (Kato) solutions. Explicit, 
Bernoulli-like formulas 
are derived and allow recovering the pressure field from the vorticity fields or the stream function. 
In the last section, the functional framework described earlier leads to a simplified rephrasing  
of the vorticity dynamics, as introduced by Maekawa in  \cite{Maekawa:2013aa}. At this level of regularity, the 
vorticity equation splits into a coupling between a parabolic and an elliptic equation corresponding respectively to the non-harmonic and 
harmonic parts of the vorticity equation. By exploiting this structure it is possible to prove 
new existence and uniqueness results, as well as the exponential decay of the palinstrophy  (that is, loosely speaking, the $H^1$ norm 
of the vorticity) for large time, an estimate which was not known so far.
\end{abstract}
\maketitle

\tableofcontents
\printindex
\listoffigures
\section{Introduction}
The NS equations stated in primitive variables (velocity and pressure) have been received much attention since  the pioneering work of Leray \cite{Leray:1933aa, Leray:1934aa, Leray:1934ab}.  Strong solutions were shown to exist in 2D by Lions 
and Prodi \cite{Lions:1959aa} and Lions \cite{Lions:1960aa}.  Henceforth, 
we will refer for instance to the books of Lions \cite[Chap.~1, Section 6]{Lions_book1969}, Ladyzhenskaya \cite{Ladyzhenskaya:1969aa} and Temam \cite{Temam:1977aa} 
for the main results that we shall need on this topic.
\par
In 2D, the vorticity equation provides an attractive alternative model to the classical NS equations for describing the dynamics of a viscous, incompressible fluid. 
Thus it exhibits many advantages: It is a nice advection-diffusion scalar equation while the classical NS system, although parabolic as well, is a coupling 
between an unknown vector  field (the velocity) and an unknown scalar field (the pressure). However the lack of natural and simple 
boundary conditions for 
the vorticity field makes the analysis of the vortex dynamics troublesome and explains why the problem has been addressed  mainly so far  in the case 
where the fluid occupies the whole space. In this configuration, 
a proof of existence and uniqueness for the corresponding Cauchy problem assuming the initial data to be integrable and twice continuously 
differentiable  was first provided by McGrath  \cite{McGrath:1968aa}. Existence results were extended independently by Cottet \cite{Cottet:1986aa} and 
Giga {\it et al.}  \cite{Giga:1988aa}   to the case where the initial data is  a finite measure. These authors proved that uniqueness also
holds when the atomic part of the initial vorticity is sufficiently small; see also \cite{Kato:1994aa}.  
For initial data in $L^1(\mathbb R^2)$, the Cauchy 
problem was proved to be well posed by Ben-Artzi   \cite{Ben-Artzi:1994aa}, and Br\'ezis \cite{Brezis:1994aa}. 
Then, Gallay and Wayne  
\cite{Gallay:2005aa} and Gallagher {\it et al.} \cite{Gallagher:2005aa} proved the uniqueness of the solution for an initial vorticity that is a large Dirac mass.  Finally, Gallagher and Gallay \cite{Gallagher:2005ab} succeeded  in removing the smallness assumption on the atomic part of the initial measure and 
shown that the Cauchy problem is globally  well-posed for any initial data in $\mathcal M(\mathbb R^2)$.
\par
As explained in \cite[Chap. 11, \S 2.7]{Giga:2018aa}, the vorticity equation (still set in the whole space) provides an interesting line of attack to study the large time behavior of the NS equations. 
This idea was exploited for instance by Giga and Kambe   \cite{Giga:1988ab}, Carpio   \cite{Carpio:1994aa}, Gallay {\it et al.} \cite{Gallay:2002aa, Gallay:2002ab, Gallay:2005aa, Gallay:2008aa} and Kukavica and Reis
\cite{Kukavica:2011aa}.
\par
Among the quoted authors above, some of them, such as McGrath \cite{McGrath:1968aa} and Ben-Artzi   \cite{Ben-Artzi:1994aa} were actually 
interested in studying 
the convergence of solutions to the NS equations 
towards solutions of the Euler equations when the viscosity vanishes. This is a very challenging problem, well understood in the absence of 
solid walls (that is, when the fluid fills the whole space) and for which the vorticity equation plays a role of paramount importance. 
In the introduction of the chapter ``Boundary Layer Theory'' in the book \cite{Rosenhead:1988aa}, Lighthill argues that
{\it 
To explain convincingly the existence of boundary layers, and, also to show what consequences of flow separation (including matters of such practical importance as the effect of trailing vortex wakes) may be expected, arguments concerning vorticity are needed.}
More recently, Chemin in \cite{Chemin:1996aa} claims {\it The key quantity for understanding 2D incompressible fluids is the
vorticity}. There exists a burgeoning literature treating the problem of vanishing viscosity limit and we refer to the recently-released book  \cite[Chap. 15]{Giga:2018aa} for a comprehensive 
list of references. When the fluid is partially of totally confined, the analysis of the vanishing viscosity limits turns into a more involved problem due to 
the formation of a boundary layer. In this case, the vorticity equation still plays a crucial role: In \cite{Kato:1984aa}, Kato gives a necessary and sufficient
condition for the vanishing viscosity limit to hold and this condition is shown by Kelliher \cite{Kelliher:2007aa,Kelliher:2008aa,Kelliher:2017aa} to 
be equivalent to the formation of a vortex sheet   on the boundary of the fluid domain.
\par
In the presence of walls, the derivation of suitable boundary conditions for the vorticity was also of prime importance for the design of numerical schemes. 
A review of these conditions (and more generally on stream-vorticity based numeral schemes), can be found in \cite{Gatski:1991aa}, \cite{Gresho:1992aa}, \cite{E:1996aa} and \cite{Napolitano:1999aa}. However, it has been actually well known since the  work of Quatarpelle and co-workers \cite{Quartapelle:1981aa, Guermond:1994aa, Guermond:1997aa,
Ern:1999aa, Auteri:2002aa, Biava:2002aa},  that the vorticity does not satisfy pointwise conditions on the boundary but rather a {\it non local} or integral 
condition which reads:
\begin{equation}
\label{damdam}
\forallt h\in\mathfrak H,\qquad\int_\F \omega h\dx=0,
\end{equation}
where $\F$ is the domain of the fluid and $\mathfrak H$ the closed subspace of the harmonic functions in $L^2(\F)$ 
(see also \cite[Lemma 1.2]{Benachour:2001aa}). Anderson \cite{Anderson:1989aa} 
and more recently 
Maekawa \cite{Maekawa:2013aa} propose nonlinear boundary conditions that will be shown to be equivalent  (see Section~\ref{SEC:more_regular}) to:
\begin{equation}
\label{cond_mae_1}
\forallt h\in\mathfrak H,\qquad\int_\F \big(-\nu\Delta\omega +u \cdot\nabla\omega \big) h\dx=0,
\end{equation}
where $\nu>0$ is the kinematic viscosity and $u$ the velocity field deduced from $\omega$ via the Biot-Savart law. Providing that $\omega$ is a solution 
to the classical vorticity equation, Equality \eqref{cond_mae_1} is nothing but the time derivative of \eqref{damdam}.
\par
Starting from \eqref{damdam}, the aim of this paper is to provide a Hilbertian functional framework allowing the analysis of the 2D vorticity equation in a bounded 
multiply connected domain. The analysis is wished to be self-contained, without recourse to classical results on the NS equations in primitive variables.  
We shall prove  that the analysis can equivalently be carried out at the level of the stream function. Homogeneous and nonhomogeneous boundary 
conditions for the velocity field will be considered and explicit formulas for the pressure will be derived. In the last section, new estimates (in particular 
for the palinstrophy) will be established.
\par
%
\section{General settings}
%
\label{SEC:geometry}
The planar domain  $\F$ occupied by the fluid  is assumed to be open, bounded and path-connected. We assume furthermore that its boundary  $\Sigma$
can be decomposed into a disjoint union of ${\mathcal C}^{1,1}$ Jordan curves:
\begin{equation}
\Sigma=\Big(\bigcup_{k=1}^N \Sigma^-_k\Big)\cup\Sigma^+.
\end{equation}
The curves $\Sigma^-_k$ for $k\in\{1,\ldots,N\}$ are the inner boundaries of $\F$ while $\Sigma^+$ is the outer boundary. 
On $\Sigma$ we denote by $n$   the unit normal vector  directed toward the exterior of the fluid and by $\tau$  the unit tangent vector
oriented in such a way that $\tau^\perp=n$ (see Fig.~\ref{domain_boundary}). Here and subsequently in the paper, for every $x=(x_1,x_2)\in\mathbb R^2$, the notation $x^\perp$ is used to represent the vector $(-x_2,x_1)$.
%
\par
%
\begin{figure}[h]
\centerline{\input{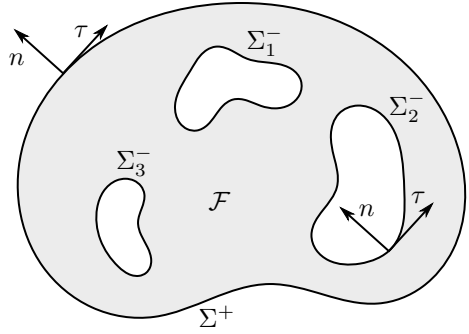}}
\caption[Domain of the fluid]{\label{domain_boundary}The domain of the fluid: an open,  $\mathcal C^{1,1}$ and $N-$connected open set. }
\end{figure}
%
Let now $T$ be a positive real number and define the space-time cylinder $\F_T=(0,T)\times\F$, whose lateral boundary is $\Sigma_T=(0,T)\times\Sigma$. 
The velocity of the fluid is supposed to be prescribed,  equal on $\Sigma_T$ to some vector field $b$ satisfying the compatibility condition:
\begin{equation}
\int_{\Sigma}b\cdot n\ds=0\qquad \text{on }(0,T).
\end{equation} 
The density and the dynamic viscosity of the fluid, denoted respectively by $\varrho$ and $\mu$, are both positive  constants. 
The flow is governed by the Navier-Stokes equations. Introducing 
$u$ the Eulerian velocity field and  $\pi$  the (static) pressure field, the equations read:
\begin{subequations}
\label{eq:main:1}
\begin{alignat}{3}
\label{eq:main:u_1a}
\partial_t u+\omega u^\perp-\nu \Delta u+\nabla\bigg( p+\frac{1}{2}| u|^2\bigg)&= f&\quad&\text{in }\F_T\\ 
\nabla\cdot u&=0&&\text{in }\F_T\\ 
\label{eq_stokes:bound}
u&=b&&\text{on }\Sigma_T\\ 
\label{eq:initi}
u(0)&=u^{\rm i}&&\text{in }\F.
\end{alignat}
\end{subequations}
In this system $\nu=\mu/\varrho$ is  the kinematic viscosity, ${\frac12}\varrho|u|^2$ is the dynamic pressure, $p= {\pi}/{\varrho}$, $f$ is a body force, $u^{\rm i}$ is a given initial condition and
$\omega$  the vorticity field defined as the curl of $u$, namely:
\begin{equation}
\omega=\nabla^\perp\cdot u \qquad\text{in }\F_T.
\end{equation}
%
%
%
\subsection{The NS system in nonprimitive variables}
The Helmholtz-Weyl  decomposition of the velocity field  (see \cite[Theorem 3.2]{Girault:1986aa}) leads to
the existence, at every moment $t$, of a potential function $\varphi(t,\cdot)$ and a stream function $\psi(t,\cdot)$ such that:
\begin{equation}
\label{eq:helmhotlz}
u(t,\cdot)=\nabla\varphi(t,\cdot)+\nabla^\perp\psi(t,\cdot)\qquad\text{in } \F.
\end{equation}
The potential function (also referred to as Kirchhoff potential) depends only on the boundary conditions satisfied 
by the velocity field of the fluid. 
It is defined at every moment $t$ as the  solution (unique up to an additive constant) of the Neumann problem:
\begin{equation}
\label{eq:neumann_phi}
\Delta \varphi(t,\cdot)=0\quad\text{in }\F\qquad\text{and}\qquad \frac{\partial\varphi}{\partial n}(t,\cdot)=b(t,\cdot)\cdot n\quad\text{on }\Sigma.
\end{equation}
The stream function $\psi$ in \eqref{eq:helmhotlz} vanishes on  $\Sigma^+$ and is constant on every connected component $\Sigma^-_j$ 
$(j=1,\ldots,N$) of the inner boundary $\Sigma^-$. Moreover, it satisfies:
\begin{equation}
\label{eq:edp_psi}
\Delta \psi(t,\cdot)=\omega(t,\cdot)\quad\text{in }\F\qquad\text{and}\qquad \frac{\partial\psi}{\partial n}(t,\cdot)=
-\big[b(t,\cdot)-\nabla\varphi(t,\cdot)\big]\cdot \tau\quad\text{on }\Sigma\forallt t>0.
\end{equation}
Forming, at any moment,  the  scalar product in $\mathbf L^2(\F)$ (the bold font notation $\mathbf L^2(\F)$ stands for $L^2(\F;\mathbb R^2)$) of \eqref{eq:main:u_1a} with $\nabla^\perp\theta$ where $\theta$ is 
a test function  that vanishes on $\Sigma^+$ and is constant 
on every $\Sigma^-_j$, we obtain (up to an integration by parts): 
\begin{equation}
\label{NS-to_vorticitiy}
\bigg(\int_{\F}\nabla\partial_t\psi\cdot\nabla\theta\dx+\int_{\F}\omega u\cdot\nabla \theta{\rm d}x\bigg)-\nu\int_{\F}\nabla\omega\cdot\nabla \theta\dx=
\int_{\F} \nabla\psi_f\cdot\nabla\theta\dx\qquad
\text{on }(0,T).
\end{equation}
In this equality, the force field $f(t,\cdot)$ has been decomposed according to the Helmholtz-Weyl  theorem:
$$f(t,\cdot)=\nabla\varphi_f(t,\cdot)+\nabla^\perp\psi_f(t,\cdot)\forallt t>0.$$
Integrating by parts again the terms in \eqref{NS-to_vorticitiy}, we end up with the system:
\begin{subequations}
\label{eq:main_vorti:1}
\begin{alignat}{3}
\label{eq:main:vorti}
\partial_t\omega+u\cdot\nabla\omega-\nu\Delta\omega&=f_V&\quad&\text{in }\F_T\\ 
\label{eq:main:flux}
-\frac{\rm d}{{\rm d}t}\left(\int_{\Sigma^-_k}b\cdot\tau\ds\right)+\int_{\Sigma^-_k}\omega (b\cdot n)\ds-\nu\int_{\Sigma_k^-}\frac{\partial\omega}{\partial n}\ds&=\int_{\Sigma^-_k}
\frac{\partial \psi_f}{\partial n}\ds&&\text{on }(0,T),\quad k=1,\ldots,N,\\
\omega(0)&=\omega^{\rm i}&&\text{in }\F,
\end{alignat}
\end{subequations}
where $f_V=\Delta \psi_f$ and the initial condition $\omega^{\rm i}$ is the curl of $u^{\rm i}$ in \eqref{eq:initi}. To be closed, System \eqref{eq:main_vorti:1} has to be supplemented with the identities \eqref{eq:helmhotlz}, \eqref{eq:neumann_phi} and \eqref{eq:edp_psi}.
\begin{rem}
The $N$ equations \eqref{eq:main:flux} (that will be termed ``Lamb's fluxes conditions'' in the sequel)
cannot be derived from \eqref{eq:main:vorti} (this is well explained in \cite[Remark 3.2]{Guermond:1997aa}). They control the mean amount of vorticity produced on the inner boundaries. 
Such  relations can be traced back to Lamb in \cite[Art. 328a]{Lamb:1993aa} (see also \cite{Wu:1998aa} for more recent references), where in a two-dimensional viscous flow the change of circulation along any curve is given by:
$$\frac{{\rm D}\varGamma}{{\rm D}t}=\nu\oint\frac{\partial\omega}{\partial n}\ds.$$
\end{rem}
At this point, 
the lack of boundary conditions for $\omega$ might indicate that System~\eqref{eq:main_vorti:1} is unlikely to be solved. Indeed, seeking for an {\it a priori} enstrophy estimate (enstrophy is the square of the $L^2(\F)$ norm of the vorticity), we multiply \eqref{eq:main:vorti} by $\omega$ an integrate over $\F$, but shortly get stuck with the term:
\begin{equation}
\label{eq:stuck}
\int_{\F}\Delta\omega\omega\dx,
\end{equation}
that cannot be integrated by parts.
The other sticking point  is that the boundary value problem \eqref{eq:edp_psi} permitting the reconstruction of the stream function from the vorticity is overdetermined since the stream function $\psi$ has to satisfy both Dirichlet and Neumann boundary conditions 
on $\Sigma$. All  these observations are well known.
%
%
\subsection{Some leading ideas}
Before going into details, we wish to give some  insights  on how the aforementioned difficulties can be  circumvented. To simplify, we shall focus for the time being on the case of homogeneous 
boundary conditions (i.e. $b=0$) and of a simply connected fluid domain (i.e. $\Sigma=\Sigma^+$). The latter assumption leads to the disappearance of the equations \eqref{eq:main:flux} in the system.
\par
The first elementary observation, that can be traced back to Quartapelle and Valz-Gris in \cite{Quartapelle:1981aa}, is that a function $\omega$ defined 
in $\F$ is the Laplacian of    some   function $\psi$ 
if and only if the following equality holds for every harmonic  function $h$:
$$
\int_\F \omega h\dx=\int_{\Sigma}\left(\frac{\partial\psi}{\partial n}\Big|_\Sigma-\Lambda_{DN}\psi|_{\Sigma}\right)h|_{\Sigma}\ds,
$$
where the notation $\Lambda_{DN}$  stands to the Dirichlet-to-Neumann operator. Introducing $\mathfrak H$, the closed subspace of the harmonic functions in $L^2(\F)$, we deduce from this assertion that:
\begin{equation}
\label{LEM:ortho_L2}
\Delta H^2_0(\F)=\mathfrak H^\perp\quad\text{in}\quad L^2(\F).
\end{equation}
%
We denote by $V_0$ the closed space $\mathfrak H^\perp$ and decompose the space $L^2(\F)$ 
 into the orthogonal sum 
\begin{equation}
\label{split_L2}
L^2(\F)=V_0\sumperp \mathfrak H.
\end{equation}
%
%
This orthogonality condition satisfied by the vorticity plays the role of boundary conditions classically  expected when dealing with a parabolic 
type equation like \eqref{eq:main:vorti}.  The authors in  \cite{Quartapelle:1981aa} and in \cite{Guermond:1994aa} do not elaborate on 
this idea and instead of deriving an autonomous functional framework for the analysis of the vorticity equation \eqref{eq:main:vorti}, System \eqref{eq:main_vorti:1} 
is supplemented with the identity:
$$
\omega(t,\cdot)=\Delta\psi(t,\cdot)\qquad\text{in }\F\forallt t\in (0,T),
$$
and some function spaces  for the stream function are introduced. However, as it will be explained later on,  the dynamics of the flow can be dealt with with any one 
of the nonprimitive variable alone (vorticity or stream function) by introducing the appropriate functional framework.
\par
Let us go back to the splitting \eqref{split_L2}. The orthogonal projection onto $\mathfrak H$ in $L^2(\F)$ is  usually referred to 
as the harmonic Bergman projection and has been received much attention so far. 
The Bergman projection, as well as the orthogonal projection onto $V_0$, denoted by $\mathsf P$ in the sequel, enjoys 
some useful  properties (see for instance \cite{Bell:1982aa}, \cite{Straube:1986aa} and references therein). 
In particular, $\mathsf P$ maps continuously   $H^k(\F)$ onto $H^k(\F)$ 
for every nonnegative integer $k$, providing that $\Sigma$ is of class ${\mathcal C}^{k+1,1}$. 
This leads us to define the spaces $V_1=\mathsf P  H^1_0(\F)$, which is therefore 
a subspace of $H^1(\F)$. We denote by ${\mathsf P}_1$ the restriction to $H^1_0(\F)$ of the projection $\mathsf P$. A quite surprising result is that ${\mathsf P}_1:H^1_0(\F)\to V_1$ 
is invertible and we denote by ${\mathsf Q}_1$ its inverse. The operator ${\mathsf Q}_1$ will be proved to be the orthogonal projector onto $H^1_0(\F)$ in $H^1(\F)$ for  
the semi-norm $\|\nabla\cdot\|_{\mathbf L^2(\F)}$. The space $V_1$ is next equipped with the scalar product 
$$(\omega_1,\omega_2)_{V_1}=(\nabla {\mathsf Q}_1\omega_1,\nabla {\mathsf Q}_1\omega_2)_{\mathbf L^2(\F)},\qquad\omega_1,\omega_2\in V_1,$$
and the corresponding norm is shown to be equivalent to the usual norm of $H^1(\F)$. Since the inclusion $H^1_0(\F)\subset L^2(\F)$ is continuous, dense and compact, 
we can draw the same conclusion for the inclusion $V_1\subset V_0$. Identifying $V_0$ with its dual space by means of 
Riesz Theorem and denoting by $V_{-1}$ the dual space of $V_1$, we end up 
with a so-called Gelfand triple of Hilbert spaces (see for instance \cite[Chap. 14]{Berezansky:1996aa}):
$$V_1\subset V_0\subset V_{-1},$$
where $V_0$ is the pivot space. With these settings, it is classical to introduce first the isometric operator ${\mathsf A}^V_1:V_1\to V_{-1}$ defined by the relation: 
$$\langle {\mathsf A}^V_1 \omega_1,\omega_2\rangle_{V_{-1},V_1}=(\omega_1,\omega_2)_{V_1}\forallt \omega_1,\omega_2\in V_1,$$
and next the space $V_2$ as the preimage of $V_0$ by ${\mathsf A}^V_1$. The space $V_2$ is a Hilbert space as well, once equipped with the scalar product 
$$(\omega_1,\omega_2)_{V_2}=({\mathsf A}^V_1\omega_1,{\mathsf A}^V_1\omega_2)_{V_0}\forallt \omega_1,\omega_2\in V_2,$$
 and the inclusion $V_2\subset V_1$ is continuous dense and compact. We denote by ${\mathsf A}^V_2$ the restriction of ${\mathsf A}^V_1$ to $V_2$ and 
classical results on Gelfand triples assert that the operator 
${\mathsf A}^V_2$ is  an isometry from $V_2$ onto $V_0$. The crucial observation for our purpose is that, providing that $\Sigma$ is of class ${\mathcal C}^{3,1}$: 
$$V_2=\bigg\{\omega\in H^2(\F)\cap V_1\,:\, \frac{\partial\omega}{\partial n}\Big|_\Sigma=\Lambda_{DN}\omega|_\Sigma\bigg\} 
\qquad\text{and}\qquad {\mathsf A}^V_2\omega=-\Delta \omega\quad\text{ for every }\omega\text{ in }V_2.$$
In particular, every vorticity in $V_2$ has zero mean flux through the boundary $\Sigma$.
Denoting by $V_{k+2}$ the preimage of $V_k$ by $\mathsf A^V_2$ for every integer $k\geqslant 1$, we define by induction a chain of embedded Hilbert spaces $V_k$  whose dual spaces are denoted by $V_{-k}$. Each one of the following inclusion is continuous dense and compact:
$$\ldots \subset V_{k+1}\subset V_{k}\subset {V_{k-1}}\subset \ldots \subset V_{1}
\subset V_0\subset V_{-1}\subset \ldots \subset {V_{-k+1}}\subset V_{-k}\subset V_{-k-1}\subset \ldots$$
We define as well isometries ${\mathsf A}_k^V:V_k\to V_{k-2}$ for all the integers $k$. This construction is made precise in Appendix~\ref{gelf_triple}. It supplies a suitable functional framework to deal with the 
linearized vorticity equation. 
Thus, we shall prove in the sequel that for every $T>0$, every integer $k$, every $f_V\in L^2(0,T; V_{k-1})$ and every 
$\omega^{\rm i}$ in $V_k$ there exists a unique solution 
\begin{subequations}
\label{pb:test1}
\begin{equation}
\omega\in H^1(0,T;V_{k-1})\cap {\mathcal C}([0,T];V_k)\cap L^2(0,T;V_{k+1}),
\end{equation}
to the Cauchy problem:
\begin{alignat}{3}
\partial_t\omega + \nu {\mathsf A}_{k+1}^V\omega &=f_V&\quad &\text{in }\F_T\\
\omega(0)&=\omega^{\rm i}&&\text{in }\F.
\end{alignat}
\end{subequations}
Let us go back  to the problem of enstrophy estimate where we got stuck with the term \eqref{eq:stuck}.
At the level of regularity corresponding to $k=0$ in \eqref{pb:test1} for instance, we obtain:
\begin{equation}
\label{eq:enstro}
\frac{1}{2}\frac{\rm d}{{\rm d}t}\|\omega\|_{V_0}^2+\nu\|\omega\|_{V_1}^2=\langle f_V,\omega\rangle_{V_1,V_{-1}}\qquad
\text{on }(0,T).
\end{equation}
By definition $\|\omega\|_{V_0}^2=\|\omega\|_{L^2(\F)}^2$, which is the expected quantity but the second term in 
the left hand side is $\|\omega\|_{V_1}^2 =\|\nabla {\mathsf Q}_1 \omega\|_{\mathbf L^2(\F)}^2$, whereas one would naively expect $\|\nabla\omega\|_{\mathbf L^2(\F)}^2$. We recall that ${\mathsf Q}_1$ is the orthogonal projection 
onto $H^1_0(\F)$. So now, instead of multiplying \eqref{eq:main:vorti} by $\omega$, let multiply this equation by ${\mathsf Q}_1\omega$, whose trace vanishes 
on $\Sigma$, and 
integrate over $\F$.
The term \eqref{eq:stuck} is replaced by a quantity that can now be integrated by parts. Thus:
$$\int_\F\Delta \omega {\mathsf Q}_1\omega\dx=-(\nabla \omega,\nabla {\mathsf Q}_1\omega)_{\mathbf L^2(\F)}=
-\|\nabla {\mathsf Q}_1\omega\|_{\mathbf L^2(\F)}^2=-\|\omega\|_{V_1}^2.$$
On the other hand, regarding the first term in \eqref{eq:main:vorti}, we still have (at least formally):
$$\int_\F\partial_t\omega {\mathsf Q}_1\omega\dx=\int_\F\partial_t\omega  \omega\dx=\frac{1}{2}\frac{\rm d}{{\rm d}t}\int_\F|\omega|^2\dx,$$
because $\omega$ is orthogonal in $L^2(\F)$  to the harmonic functions and ${\mathsf Q}_1\omega$ and $\omega$ differ only up to 
an harmonic function. To sum up, in the enstrophy estimate, the natural dissipative term is not $\|\nabla\omega\|_{\mathbf L^2(\F)}^2$ but 
$\|\nabla {\mathsf Q}_1\omega\|_{\mathbf L^2(\F)}^2$. Notice that, since ${\mathsf Q}_1$ is the orthogonal projector onto $H^1_0(\F)$:
$$\|\nabla {\mathsf Q}_1\omega\|_{\mathbf L^2(\F)}^2\leqslant \|\nabla \omega\|_{\mathbf L^2(\F)}^2\
\forallt \omega\in H^1(\F).$$
Defining the lowest eigenvalue of ${\mathsf A}^V_1$ by means of a Rayleigh quotient:
\begin{equation}
\label{def_lambdaF}
{\lambda_\F}=\min_{\omega\in V_1\atop \omega\neq 0}\frac{\|\omega\|^2_{V_1}}{\|\omega\|^2_{V_0}}=\min_{\omega\in V_1\atop \omega\neq 0}\frac{\|\nabla {\mathsf Q}_1\omega\|^2_{\mathbf L^2(\F)}}{\|\omega\|^2_{L^2(\F)}},
\end{equation}
the following Poincaré-type estimate holds true:
$$\lambda_\F \|\omega\|_{V_0}^2\leqslant \|\omega\|_{V_1}^2\quad\text{for all }\omega\in V_1,$$
and classically leads with \eqref{eq:enstro} (assuming that $f_V=0$ to simplify) and Gr\"onwall's inequality to the estimate:
$$\|\omega(t)\|_{V_0} \leqslant \|\omega^{\rm i}\|_{V_0}  e^{-\nu\lambda_\F t},\qquad t\geqslant 0,$$
where the constant $\lambda_\F$ is optimal. This constant governing the exponential decay of the solution 
is actually the same at any level of regularity. Thus, the solution to \eqref{pb:test1} (with $\beta=0$) satisfies for every integer $k$:
$$ \|\omega(t)\|_{V_k} \leqslant \|\omega^{\rm i}\|_{V_{k}}  e^{-\nu\lambda_\F t},\qquad t\geqslant 0.$$
\begin{rem}
\label{first_kato:rem}
Kato's criteria for the existence of the vanishing viscosity limit in \cite{Kato:1984aa} and rephrased in  terms of the vorticity by Kelliher in \cite{Kelliher:2008aa} will 
be shown to be equivalent to the convergence of $\omega^\nu$ toward $\omega$ in the space $V_{-1}$ ($\omega^\nu$ 
stands for the vorticity of   NS equations with 
vorticity $\nu$ and $\omega$ is the vorticity of  Euler equations). Some care should be taken with the space $V_{-1}$ because it is not 
a distribution space, what may result in some mistakes or misunderstandings (we refer here to the very instructive paper of Simon \cite{Simon:2010aa}).
\end{rem}
\par
As   mentioned earlier, the analysis of the dynamics of the flow can as well be carried out in terms of the sole stream function. It suffices to 
introduce the function spaces $S_0=H^1_0(\F)$ and $S_1=H^2_0(\F)$. The inclusion $S_1\subset S_0$ being continuous dense and compact, the configuration
$S_1\subset S_0\subset S_{-1}$ (with $S_{-1}$ the dual space of $S_1$) is a Gelfand triple where $S_0$ is the pivot space. We proceed as for the 
vorticity spaces and define  a chain of embedded Hilbert 
spaces $S_k$ and related isometries ${\mathsf A}^S_k:S_k\to S_{k-2}$ for every integer $k$ (we refer again to Appendix~\ref{gelf_triple} for the details). In particular, providing that $\Sigma$ is of class 
${\mathcal C}^{2,1}$, we will verify that:
$$S_2=  H^3(\F)\cap H^2_0(\F)\qquad\text{and}\qquad {\mathsf A}_2^S\psi=-{\mathsf Q}_1\Delta \psi\forallt\psi\in S_2.$$
The counterpart of the Cauchy problem \eqref{pb:test1}, restated in terms of the stream function is:
\begin{subequations}
\label{pb:test2}
\begin{alignat}{3}
\partial_t\psi + \nu {\mathsf A}_{k+1}^S\psi &=f_S&\quad &\text{in }\F_T\\
\psi(0)&=\psi^{\rm i}&&\text{in }\F.
\end{alignat}
For every $T>0$, every integer $k$, every $f_S\in L^2(0,T; S_{k-1})$ and every 
$\psi^{\rm i}$ in $S_k$, this problem admits a unique solution:
\begin{equation}
\label{def_psi_intro}
\psi\in H^1(0,T;S_{k-1})\cap {\mathcal C}([0,T];S_k)\cap L^2(0,T;S_{k+1}),
\end{equation}
\end{subequations}
which satisfies in addition the exponential decay estimate (assuming that $f_S=0$ to simplify):
$$\|\psi(t)\|_{S_k}\leqslant \|\psi^{\rm i}\|_{S_k}e^{-\nu\lambda_\F t}\forallt t\geqslant 0.$$
The constant $\lambda_\F$ is defined in \eqref{def_lambdaF} and is therefore the same as the one governing the exponential decay of 
the enstrophy.
\par
The solution to problem \eqref{pb:test1} can easily be deduced from the solution to problem \eqref{pb:test2} and  vice versa. Indeed, for every integer $k$, 
the operator:
$$\Delta_k:\psi\in S_{k+1}\longmapsto \Delta \psi\in V_k,$$
can be shown to be an isometry. 
Thus, let be given $T>0$ and consider 
\begin{itemize}
\item $\omega\in H^1(0,T;V_{k-1})\cap {\mathcal C}([0,T];V_{k})\cap L^2(0,T;V_{k+1})$ the unique solution to Problem \eqref{pb:test1} for some integer $k$, 
some initial condition $\omega^{\rm i}\in V_k$ and some source term $f_V\in L^2(0,T;V_{k-1})$;
\item $\psi\in H^1(0,T;S_{k'-1})\cap {\mathcal C}([0,T];S_{k'})\cap L^2(0,T;S_{k'+1})$ the unique solution to Problem \eqref{pb:test2} for some integer $k'$, 
some initial condition $\psi^{\rm i}\in S_{k'}$ and some source term $f_S\in L^2(0,T;S_{k'-1})$. 
\end{itemize}
Providing that $k'=k+1$, we claim that both following assertions are equivalent:
\begin{my_enumerate}
\item $\omega=\Delta_k \psi$;
\item $\omega^{\rm i}=\Delta_k\psi^{\rm i}$ and $f_V=\Delta_{k-1} f_S$.
\end{my_enumerate}
If we take for granted that the  operators ${\mathsf P}_k:S_{k-1}\to V_k$ and ${\mathsf Q}_k:V_k\to S_{k-1}$ can be defined at any level of regularity in such 
a way that 
${\mathsf P}_k$ extend ${\mathsf P}_{k'}$ if $k\leqslant k'$ and ${\mathsf Q}_k={\mathsf P}_k^{-1}$, we can show that the diagram 
in Fig.~\ref{diag_2intro} commutes and all the operators are isometries. 
\par
\begin{figure}[ht]
 $$
  \xymatrix @!0 @R=25mm @C=35mm {
 V_{k+2}     \ar[r]^{{\mathsf A}^{V}_{k+2}} 
 	 \ar@/^/[d]^{{\mathsf Q}_{k+2}}  
 %
 & V_{k}  
  \ar@/^/[d]^{{\mathsf Q}_k}  \\
    S_{k+1}\ar@/^/[u]^-{{\mathsf P}_{k+2}}
    \ar[r]^-{{\mathsf A}^{S}_{k+1}} 
    \ar[ru]^{\Delta_k}
    &S_{k-1}\ar@/^/[u]^{{\mathsf P}_k}
  }
  $$
  \caption[The spaces $V_k$, $S_k$ and associated operators]{\label{diag_2intro}The top row contains the spaces $V_k$ for the vorticity fields while 
  the bottom row contains the stream function spaces $S_k$. The operators ${\mathsf A}^V_k$ and ${\mathsf A}^S_k$ appears in the Cauchy problems \eqref{pb:test1} and \eqref{pb:test2} respectively. The operators $\Delta_k$ 
  link isometrically the stream functions to the corresponding vorticity fields.}
  \end{figure}
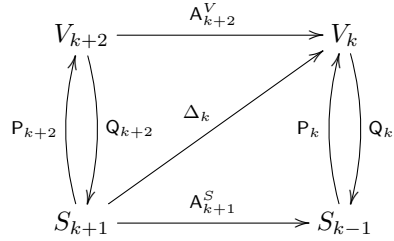
 %
 To accurately  state the equivalence result between Problems \eqref{pb:test1} (Stokes problem in vorticity variable), \eqref{pb:test2} (Stokes problem 
 in stream function variable) and the evolution homogeneous Stokes equations in primitive variables, it is worth recalling 
 the functional framework for the Stokes equations by introducing the spaces:
 \begin{subequations}
 \label{def:Jk}
 \begin{align}
\mathbf J_0&=\big\{u\in \mathbf L^2(\F)\,:\,\nabla\cdot u=0\text{ in }\F\text{ and }u|_\Sigma\cdot n=0 \big\},\\
 \mathbf J_1&=\big\{u\in \mathbf H^1(\F)\,:\,\nabla\cdot u=0\text{ in }\F\text{ and }u|_\Sigma =0\big\},
 \end{align}
 whose scalar products are respectively:
  \begin{align}
(u_1,u_2)_{\mathbf J_0}&=\int_\F u_1\cdot u_2\dx\forallt u_1,u_2\in \mathbf J_0,\\
(u_1,u_2)_{\mathbf J_1}&=\int_\F \nabla u_1:\nabla u_2\dx\forallt u_1,u_2\in \mathbf J_1.
\end{align}
\end{subequations}
The inclusion $\mathbf J_1\subset \mathbf J_0$ being continuous dense and compact, from the Gelfand triple $\mathbf J_1\subset \mathbf J_0\subset \mathbf J_{-1}$ 
we can define a chain of embedded Hilbert spaces $\mathbf J_k$ and 
isometries ${\mathsf A}^{\mathbf J}_k:\mathbf J_k\to \mathbf J_{k-2}$ for every integer $k$. Providing that $\Sigma$ is of class ${\mathcal C}^{1,1}$, it can be shown in particular that:
$$\mathbf J_2=\mathbf J_1\cap \mathbf H^2(\F)\qquad\text{and}\qquad {\mathsf A}_2^{\mathbf J}=-\Pi_0\Delta,$$
where $\Pi_0:\mathbf L^2(\F)\to \mathbf J_0$ is the Leray projector.
For every $T>0$, every integer $k$, every $f_{\mathbf J}\in L^2(0,T; \mathbf J_{k-1})$ and every 
$u^{\rm i}$ in $\mathbf J_k$, it is well known that there exists a unique solution 
$$
u\in H^1(0,T;\mathbf J_{k-1})\cap {\mathcal C}([0,T];\mathbf J_k)\cap L^2(0,T;\mathbf J_{k+1}),
$$
\begin{subequations}
\label{pb:test3}
to the Cauchy problem:
\begin{alignat}{3}
\partial_t u + \nu {\mathsf A}_{k+1}^{\mathbf J} u &=f_{\mathbf J}&\quad &\text{in }\F_T\\
u(0)&=u^{\rm i}&&\text{in }\F.
\end{alignat}
\end{subequations}
The operator:
$$\nabla_k^\perp:\psi\in S_k\longmapsto \nabla^\perp\psi\in \mathbf J_{k-1},$$
will be proved to be an isometry for every integer $k$. It allows us to link Problem \eqref{pb:test3} to the equivalent problems \eqref{pb:test1} and \eqref{pb:test2}. 
More precisely,  let be given $T>0$ and consider 
\begin{itemize}
\item $u\in H^1(0,T;\mathbf J_{k-1})\cap {\mathcal C}([0,T];\mathbf J_{k})\cap L^2(0,T;\mathbf J_{k+1})$ the unique solution to Problem \eqref{pb:test1} for some integer $k$, 
some initial condition $u^{\rm i}\in \mathbf J_k$ and some source term $f_{\mathbf J}\in L^2(0,T;\mathbf J_{k-1})$;
\item $\psi\in H^1(0,T;S_{k'-1})\cap {\mathcal C}([0,T];S_{k'})\cap L^2(0,T;S_{k'+1})$ the unique solution to Problem \eqref{pb:test2} for some integer $k'$, 
some initial condition $\psi^{\rm i}\in S_{k'}$ and some source term $f_S\in L^2(0,T;S_{k'-1})$.
\end{itemize}
Providing that $k'=k$, we claim that both following assertions are equivalent:
\begin{my_enumerate}
\item $u=\nabla_k^\perp \psi$;
\item $u^{\rm i}=\nabla_k^\perp\psi^{\rm i}$ and $f_{\mathbf J}=\nabla^\perp_{k-1} f_S$.
\end{my_enumerate}
\par
To conclude this short presentation of the main ideas that will be further elaborated in this paper, it is worth noticing that, contrary to what 
happens with primitive variables, the case where $\F$ is multiply connected
 is notably more involved than the simply connected case. The same observation could still be came across 
 in the articles of  Glowinski and   Pironneau \cite{Glowinski:1979aa} and  Guermond and  Quartapelle  \cite{Guermond:1994aa}.
%
\subsection{Organization of the paper}
The next section is devoted to the study of the Stokes operator in nonprimitive variables (namely the operators $\mathsf A^V_k$ and $\mathsf A^S_k$ mentioned 
in the preceding section). The expression of the Biot-Savart law is also provided. 
Then, in Section~\ref{SEC:lift_oper}, lifting operators (for both the vorticity field and the stream function) are defined. 
They are required in Section~\ref{SEC:evol_stokes} for the analysis of the evolution Stokes problem (in nonprimitive variables) 
with nonhomogeneous boundary conditions. The NS equations in nonprimitive variables  is dealt with in Section~\ref{SEC:Navier-Stokes} where 
weak and strong solutions are addressed. Explicit formulas 
to recover the pressure from the vorticity or the stream function 
are supplied 
in Section~\ref{SEC:pressure}. The existence and uniqueness of more regular vorticity solutions is examined in Section~\ref{SEC:more_regular}. In this section we also prove the exponential decay of the palinstrophy (i.e. of the 
quantity $\|\nabla\omega\|_{\mathbf L^2(\F)}$) when time growths. In 
Section~\ref{SEC:conclude} we conclude with providing some  insights on upcoming generalization results for 
coupled fluid-structure systems.
%
\section{Stokes operator}
%

%
\subsection{Function spaces}
\label{SEC:main_spaces}
Let $\Sigma_0$ stands for  either $\Sigma^+$ or $\Sigma^-_j$ for some $j\in\{1,\ldots,N\}$. Providing that $\Sigma_0$ is of class ${\mathcal C}^{k,1}$ ($k$ being a nonnegative integer), it makes sense to consider 
the boundary Sobolev space $H^{k+{\frac12}}(\Sigma_0)$ and its dual space $H^{-k-{\frac12}}(\Sigma_0)$. 
Using $L^2(\Sigma_0)$ 
as pivot space, we shall use a boundary integral notation in place of the duality pairing all along this paper. More precisely, 
we adopt the following convention of notation:
\begin{equation}
\label{rem:brackets}
\langle g_1,g_2\rangle_{H^{-k-{\frac12}}(\Sigma_0),H^{k+{\frac12}}(\Sigma_0)}=\int_{\Sigma_0} g_1 g_2\ds\forallt g_1\in H^{-k-{\frac12}}(\Sigma_0)\text{ and }g_2\in H^{k+{\frac12}}(\Sigma_0).
\end{equation}
In particular, following this rule:
$$\langle g,1\rangle_{H^{-k-{\frac12}}(\Sigma_0),H^{k+{\frac12}}(\Sigma_0)}=\int_{\Sigma_0} g\ds\forallt \in H^{-k-{\frac12}}(\Sigma_0).$$
\subsubsection*{Fundamental function spaces}
For every nonnegative integer $k$, we denote by $H^k(\F)$ the classical Sobolev spaces 
of index $k$ and we define the Hilbert spaces:
\begin{subequations}
\begin{align}
S_0&=\{\psi\in  H^1(\F)\,:\,\, \psi|_{\Sigma^+}=0\quad\text{and}\quad \psi|_{\Sigma^-_j}=c_j,\quad c_j\in\mathbb R,\quad j=1,\ldots,N\},\\
S_1&=\bigg\{\psi\in S_0\cap H^2(\F)\,:\,\frac{\partial \psi}{\partial n}\Big|_\Sigma=0\bigg\},
\end{align}
%
provided with the scalar products:
\begin{align}
(\psi_1,\psi_2)_{S_0}&=(\nabla \psi_1,\nabla\psi_2)_{\mathbf L^2(\F)}\forallt \psi_1,\psi_2\in  S_0,\\
(\psi_1,\psi_2)_{S_1}&=(\Delta\psi_1,\Delta\psi_2)_{L^2(\F)}\forallt \psi_1,\psi_2\in  S_1.
\end{align}
\end{subequations}
The norm $\|\cdot\|_{S_0}$ is equivalent in $S_0$ to the usual norm of $H^1(\F)$. For every $j=1,\ldots,N$, we 
define the continuous linear form $\mathsf{Tr}_j:\psi\in S_0\mapsto \psi|_{\Sigma^-_j}\in \mathbb R$ and the function $\xi_j$ as the unique solution in $S_0$ to the variational problem: 
\begin{subequations}
\begin{equation}
(\xi_j,\theta)_{S_0}+\mathsf{Tr}_j\theta=0\forallt\theta\in S_0.
\end{equation}
The functions $\xi_j$ are harmonic in $\F$ and obey the mean fluxes conditions:
\begin{equation}
\label{flux_cond_xi}
\int_{\Sigma_k^-}\frac{\partial \xi_j}{\partial n}\ds=-\delta_j^k\qquad\text{for }k=1,\ldots,N,
\end{equation}
\end{subequations}
where $\delta_j^k$ is the Kronecker symbol.
We denote by $\mathbb F_S$ the $N$ dimensional subspace of $S_0$ spanned by the functions $\xi_j$ ($j=1,\ldots,N$) that will
account for the fluxes of the stream functions through the inner boundaries. Notice that the Gram matrix 
$\big((\xi_j,\xi_k)_{S_0}\big)_{1\leqslant j,k\leqslant N}$ is invertible and equal to the matrix of 
the traces $\big(-\mathsf{Tr}_k\xi_j\big)_{1\leqslant j,k\leqslant N}$. Therefore, by means of a Gram-Schmidt process, we can derive from the free 
family $\{\xi_j,\,j=1,\ldots,N\}$, 
an orthonormal family in $S_0$, denoted by $\{\hat\xi_j,\,j=1,\ldots,N\}$. The space $S_0$ admits the following orthogonal decomposition:
\begin{equation}
\label{decom_S0_1}
S_0=H^1_0(\F)\sumperp \mathbb F_S.
\end{equation}
In $S_1$, the norm $\|\cdot\|_{S_1}$ is equivalent to the usual norm of $H^2(\F)$. For every $j=1,\ldots,N$,  
we denote by $\chi_j$ the unique solution in $S_1$ such that:
%
\begin{equation}
\label{def_chi}
\Delta^2\chi_j=0\quad\text{ in }\F\qquad\text{and}\qquad\int_{\Sigma_k^-}\frac{\partial \Delta\chi_j}{\partial n}\ds=-\delta_j^k\qquad\text{for }k=1,\ldots,N,
\end{equation}
where the normal derivative of $\Delta\chi_j$ is in $H^{-{\frac32}}(\Sigma_k^-)$  (see the convention of notation \eqref{rem:brackets}).
We denote by ${\mathbb B}_S$ the $N$ dimensional subspace of $S_1$ spanned by the functions $\chi_j$. The Gram matrix $\big((\chi_j,\chi_k)_{S_1}\big)_{1\leqslant j,k\leqslant N}$ being invertible and equal to the matrix of traces 
$\big({\mathsf {Tr}}_j\chi_k\big)_{1\leqslant j,k\leqslant N}$, we infer that:
\begin{equation}
\label{split_S1}
S_1=H^2_0(\F)\sumperp {\mathbb B}_S.
\end{equation}
In $L^2(\F)$, we denote by $\mathfrak H$   the closed subspace of the harmonic functions with zero mean flux through every connected 
part $\Sigma_j^-$ of the inner boundary ($j=1,\ldots,N$), namely:
\begin{equation}
\mathfrak H= 
\Big\{h\in L^2(\F)\,:\,\Delta h=0\text{ in }\mathcal D(\F)\text{ and }(h,\Delta\chi_j)_{L^2(\F)}=0,\quad j=1,\ldots,N\Big\}.
\end{equation}
\index{H0@$\mathfrak H$: harmonic functions in $L^2(\F)$ wish zero flux}
%
%
The space $L^2(\F)$ admits the orthogonal decomposition:
\begin{equation}
\label{eq:decomp_L2}
L^2(\F)=\mathfrak H\sumperp V_0 \quad
\text{ where }
\quad V_0 =\mathfrak H^\perp.
\end{equation}
\index{V0@$V_0$: vorticity space in $L^2(\F)$}
It results from the following lemma that the space $V_0$ is the natural function space for the vorticity field.
%
\begin{lemma}
\label{first_iso}
The operator $\Delta_0:\psi\in S_1\mapsto \Delta\psi\in V_0$
%
is an isometry. 
\end{lemma}
%
\begin{proof}
Let $\omega$ be in $V_0$ and denote by $\psi_0$ the unique function in $H^1_0(\F)\cap H^2(\F)$ satisfying $\Delta\psi_0=\omega$.
On the other hand, using the rule of notation \eqref{rem:brackets},  
the function
$$h_0=h-\sum_{j=1}^N \left(\int_{\Sigma}\frac{\partial \hat\xi_j}{\partial n}h\ds\right) \hat\xi_j,$$
is in the space $\mathfrak H^1=\mathfrak H\cap H^1(\F)$ providing that $h$ is a harmonic function in $H^1(\F)$. It follows that:
$$
(\omega,h_0)_{L^2(\F)}=(\Delta\psi_0,h_0)_{L^2(\F)}=\int_{\Sigma}  \frac{\partial\psi_0}{\partial n}h_0 \ds
=\int_{\Sigma} \bigg[\frac{\partial\psi_0}{\partial n}-\sum_{j=1}^N 
 \left(\int_{\Sigma} \frac{\partial \psi_0}{\partial n}\hat\xi_j \ds\right) \frac{\partial \hat\xi_j}{\partial n}\bigg]h \ds=0.
$$
Since every element in $H^{{\frac12}}(\Sigma)$ can be achieved as the trace of a harmonic function in $H^1(\F)$, the equality above entails that:
$$
\frac{\partial\psi_0}{\partial n}=\sum_{j=1}^N 
 \left(\int_{\Sigma} \frac{\partial \psi_0}{\partial n}\hat\xi_j  \ds\right) \frac{\partial \hat\xi_j}{\partial n}\qquad\text{in }H^{-{\frac12}}(\Sigma).
$$
We are done by noticing now that the function:
$$\psi=\psi_0-\sum_{j=1}^N 
 \left(\int_{\Sigma} \frac{\partial \psi_0}{\partial n}\hat\xi_j  \ds\right)  \hat\xi_j,$$
is in $S_1$ and solves $\Delta \psi=\omega$. Uniqueness being straightforward, the proof is then complete.
\end{proof}
\subsubsection*{The Bergman projection and its inverse}
Considering the orthogonal splitting \eqref{eq:decomp_L2} of $L^2(\F)$, 
we denote by $\mathsf P$ the orthogonal projection from $L^2(\F)$ onto $V_0$ while the notation $\mathsf P^\perp$ will stand for
the orthogonal projection onto $\mathfrak H$. When the domain $\F$ is simply connected, the operator $\mathsf P^\perp$ 
is referred to as the harmonic Bergman projection and has been extensively studied (see for instance \cite{Bell:1982aa}, \cite{Straube:1986aa} and references therein). 
The projector $\mathsf P$ (and also $\mathsf P^\perp$ which we are less interested in) enjoys the following property:
%
\begin{lemma}
\label{regul_P0}
Assume that $\Sigma$ is of class ${\mathcal C}^{k+1,1}$ for some nonnegative integer $k$, then $\mathsf P$ (and $\mathsf P^\perp$) maps $H^k(\F)$ into $H^k(\F)$
and $\mathsf P$, seen as an operator from $H^k(\F)$ into $H^k(\F)$, is bounded.
\end{lemma}
%
\begin{proof}
Let $u$ be in $L^2({\F})$. The proof consists in verifying that 
$$\mathsf P u=\Delta w_0+\sum_{k=1}^N (\Delta\chi_k,u)_{L^2(\F)}\Delta\chi_k,$$
 where the functions $w_0$ belongs to 
$H^2_0(\F)$ and satisfies the variational formulation:
\begin{equation}
\label{varia_w0}
(\Delta w_0,\Delta\theta_0)_{L^2(\F)}=(u,\Delta \theta_0)_{L^2(\F)},\forallt \theta_0\in H^2_0(\F).
\end{equation}
The conclusion of the Lemma will  follow according to elliptic regularity results for the biharmonic operator stated for instance in \cite[Theorem 1.11]{Girault:1986aa}.
By definition:
$$\mathsf P u={\rm argmin}\bigg\{\frac{1}{2}\int_{\F} |v-u|^2\dx\,:\, v\in V_0\bigg\}.$$
According to Lemma~\ref{first_iso}, there exists a unique $w\in S_1$ such that $\mathsf P u=\Delta  w$ and:
$$ w={\rm argmin}\bigg\{\frac{1}{2}\int_{\F} |\Delta \theta-u|^2\dx\,:\, \theta\in S_1\bigg\}.$$
Owning to the orthogonal decomposition \eqref{split_S1}, the function $w$ can be decomposed as:
$$w=w_0+\sum_{k=1}^N \alpha_k \chi_k,$$ 
where $w_0\in H^2_0(\F)$ and $(\alpha_1,\ldots,\alpha_N)\in\mathbb R^N$ are such that:
$$(w_0,\alpha_1,\ldots,\alpha_N)={\rm argmin}\bigg\{\frac{1}{2}\int_{\F} \Big|\Delta \theta_0+\sum_{k=1}^N{\beta_k}\Delta\chi_k-u\Big|^2\dx\,:\, (\theta_0,\beta_1,\ldots,\beta_N)\in H^2_0({\F})\times\mathbb R^N \bigg\}.$$
It follows that $w_0$ solves indeed the variational problem \eqref{varia_w0} and $\alpha_k=(\Delta\chi_k,u)_{L^2(\F)}$ for every $k=1,\ldots,N$. 
\end{proof}
%
\begin{rem}
\label{rem_fluxA}
The following observations are in order:
\begin{enumerate}
\item
The harmonic Bergman projection is quite demanding in terms of boundary regularity, and one may wonder if the assumption on the regularity 
of $\Sigma$ is optimal in the statement of Lemma~\ref{regul_P0}. 
Focusing on the case $k=0$, the definition of the space $\mathfrak H$ requires defining the flux of harmonic functions 
through the connected parts of $\Sigma^-$.  The normal derivative of harmonic functions in $L^2(\F)$ can be defined as   elements
of $H^{-{\frac32}}(\Sigma)$. However, it requires the boundary to be ${\mathcal C}^{1,1}$ (see \cite[page 54]{Grisvard:1985aa}), which is the default level of regularity 
assumed for the domain $\F$ throughout this article. 
%
\item
\label{rem_flux}
For every $u\in L^2(\F)$, the function $\mathsf P^\perp u$ belongs to $\mathfrak H$ and therefore admits a normal trace 
on every $\Sigma_j^-$ ($j=1,\ldots,N$) in $H^{-{\frac32}}(\Sigma_j^-)$. We deduce that, when $u$ belongs to $H^2(\F)$, the fluxes of $u$ 
across  the parts $\Sigma^-_j$ ($j=1,\ldots,N$) of the boundary
are conserved by the projection $\mathsf P$, namely:
$$\int_{\Sigma^-_j}\frac{\partial \mathsf Pu}{\partial n}\ds=\int_{\Sigma^-_j}\frac{\partial u}{\partial n}\ds\forallt j=1,\ldots,N,$$
where ${\partial \mathsf Pu}/{\partial n}$ belongs to $H^{-{\frac32}+k}(\Sigma_j^-)$ providing that $\Sigma_j^-$ is of class ${\mathcal C}^{k+1,1}$ for $k=0,1,2$.
\end{enumerate}
\end{rem}
%
%
%
Let us define now the operator:
\begin{equation}
\label{def_Q}
\mathsf Q:u\in H^1(\F)\longmapsto \mathsf Q u = \argmin\bigg\{ \int_\F |\nabla \theta-\nabla u|^2\dx\,:\, \theta\in S_0\bigg\}\in S_0.
\end{equation}
The variational formulation  corresponding to the minimization problem reads:
\begin{equation}
\label{first_S0}
(\nabla \mathsf Qu,\nabla \theta)_{\mathbf L^2(\F)}=(\nabla  u,\nabla \theta)_{\mathbf L^2(\F)}\qquad \text{for all }\theta\in S_0,
\end{equation}
what means that $\mathsf Qu$ is  the unique function in $S_0$ that satisfies   $\Delta \mathsf Qu=\Delta u$ in $\F$. Denoting $\mathsf Q^\perp={\rm Id}-\mathsf Q$, this entails that $\mathsf Q^\perp u$ is harmonic and
choosing $\chi_k$ ($k=1,\ldots,N$) as test function in \eqref{first_S0} and integrating by parts, we obtain:
$$(\mathsf Q^\perp u,\Delta\chi_k)_{L^2(\F)}=0,\qquad k=1,\ldots,N,$$
whence we deduce that $\mathsf Q^\perp u$ lies in $\mathfrak H^1$. Besides, there exists real coefficients $\alpha_j$ such that the function:
$$u_0=\mathsf Qu-\sum_{j=1}^N\alpha_j\chi_j,$$
belongs to $H^1_0(\F)$ because  the Gram matrix $\big((\chi_j,\chi_k)_{S_1}\big)_{1\leqslant j,k\leqslant N}$ is invertible and equal to the matrix of traces 
$\big({\mathsf T}_j\chi_k\big)_{1\leqslant j,k\leqslant N}$.
It follows that for every $h\in \mathfrak H^1$:
\begin{equation}
\label{second_S0}
(\nabla \mathsf Qu,\nabla h)_{\mathbf L^2(\F)}=(\nabla u_0,\nabla h)_{\mathbf L^2(\F)}-\sum_{j=1}^N\alpha_j(h,\Delta\chi_j)_{L^2(\F)}=0.
\end{equation}
So the operators $\mathsf P$ and $\mathsf Q$  are both orthogonal projections whose kernels are harmonic functions ($\mathfrak H$ for $\mathsf P$ 
and $\mathfrak H^1$ for $\mathsf Q$) but for different scalar products. They are tightly related, as expressed in the next lemma, the statement of which 
requires introducing a new function space. Thus, we define $V_1$ as the image of $S_0$ by $\mathsf P$ 
and we denote by ${\mathsf P}_1$ the restriction of $\mathsf P$ to $S_0$. It is elementary to verify that 
$${\mathsf P}_1:S_0\longrightarrow V_1$$ 
is one-to-one. We 
denote by ${\mathsf Q}_1$ the inverse of ${\mathsf P}_1$. 
The space $V_1$
 is then provided with the image topology, namely with the scalar product:
 $$(\omega_1,\omega_2)_{V_1}=({\mathsf Q}_1\omega_1,{\mathsf Q}_1\omega_2)_{S_0}=(\nabla {\mathsf Q}_1\omega_1,\nabla {\mathsf Q}_1\omega_2)_{\mathbf L^2(\F)}\forallt \omega_1,\omega_2\in V_1.$$
 Observe that since $H^1(\F)=S_0\oplus \mathfrak H^1$, we have also $V_1=\mathsf PH^1(\F)$.
%
\begin{lemma}
\label{PQ_inverses}
If $\Sigma$ is of class ${\mathcal C}^{2,1}$, $V_1$ is a subspace of $H^1(\F)$, ${\mathsf Q}_1$ is the restriction of $\mathsf Q$ to $V_1$ and the topology of $V_1$ is equivalent 
to the topology of $H^1(\F)$.
\end{lemma}
%
\begin{proof}
According to Lemma~\ref{regul_P0}, if $\Sigma$ is of class ${\mathcal C}^{2,1}$, the space 
$V_1$ is a subspace of $H^1(\F)$ and $\mathsf P$ is bounded 
from $S_0$ onto $V_1$ (seen as subspace of $H^1(\F)$).
The operator ${\mathsf P}_1:S_0\to V_1$ being bounded and invertible, it is an isomorphism according to the bounded inverse theorem. 
Moreover, for every $\psi$ be in $S_0$: 
$$\mathsf Q {\mathsf P}_1 \psi =\mathsf Q (\psi+({\mathsf P}_1\psi-\psi))=\mathsf Q \psi+\mathsf Q ({\mathsf P}_1\psi-\psi)=\mathsf Q \psi=\psi,$$
since the function ${\mathsf P}_1\psi-\psi$ is in $\mathfrak H^1$.  The proof is now complete. 
\end{proof}
\begin{rem}
When $\Sigma$ is only of class $\mathcal C^{1,1}$, $V_1$ is a subspace of $S_0+\mathfrak H$. In particular, every element of $V_1$ has 
a trace on $\Sigma$ in  the space $H^{-\frac12}(\Sigma)$. This trace is in $H^{\frac12}(\Sigma)$ when $\Sigma$ is of class ${\mathcal C}^{2,1}$.
\end{rem}
%
%
\par
%

%
\subsubsection*{Further scalar products}
For every nonnegative integer $k$, we define $\mathfrak H^k=\mathfrak H\cap H^k(\F)$. Assuming that $\Sigma$ is of class $\mathcal C^{k-1,1}$, the space 
$\mathfrak H^k$ is provided with the scalar product:
$$(h_1,h_2)_{\mathfrak H^k}=(h_1|_\Sigma,h_2|_\Sigma)_{H^{k-\frac12}(\Sigma)}\forallt h_1,h_2\in\mathfrak H^k.$$
It would sometimes come in handy to provide $H^1(\F)$ with a scalar product that turns $\mathsf Q$ into an orthogonal projection. To do that,  it suffices to define:
\begin{subequations}
\begin{equation}
\label{def_H1}
(u_1,u_2)_{H^1}^S=(\mathsf Qu_1,\mathsf Qu_2)_{S_0}+(\mathsf Q^\perp u_1,\mathsf Q^\perp u_2)_{\mathfrak H^1}\qquad\text{for all }u_1,u_2\in H^1(\F).
\end{equation}
Similarly, 
the scalar product:
\begin{equation}
\label{def_H1V1}
(u_1,u_2)_{H^1}^V=(\mathsf Pu_1,\mathsf Pu_2)_{V_1}+(\mathsf P^\perp u_1,\mathsf P^\perp u_2)_{\mathfrak H^1}\qquad\text{for all }u_1,u_2\in H^1(\F),
\end{equation}
\end{subequations}
turns the direct sum $H^1(\F)=V_1\oplus \mathfrak H^1$ into an orthogonal sum. 
%
\subsection{Stokes operator in nonprimitive variables}
\label{SEc:Stokes_SV}
The inclusion $S_1\subset S_0$ is clearly continuous dense and compact. Identifying the Hilbert space $S_0$ with its dual 
and denoting by $S_{-1}$ the dual space of $S_1$, we obtain the Gelfand triple:
\begin{equation}
S_1\subset S_0\subset S_{-1}.
\end{equation}
Following the lines of Appendix~\ref{gelf_triple},  we can
 define (with obvious notation) a family of embedded 
Hilbert spaces $\{S_k,\,k\in\mathbb Z\}$, a family of isometries  $\{{\mathsf A}_k^S:S_k\to S_{k-2},\,k\in\mathbb Z\}$ and a positive constant:
\begin{equation}
\label{def_poincare_const_S}
\lambda_\F^S=\min_{\psi\in S_1\atop \psi\neq 0}\frac{\|\psi\|^2_{S_1}}{\|\psi\|^2_{S_0}}.
\end{equation}
%
\begin{lemma}
\label{lem:sapce_S2}
The space $S_2$ is equal to $H^3(\F)\cap S_1$ providing that $\Sigma$ is of class ${\mathcal C}^{2,1}$.
For $k\geqslant 2$, the expressions of the operator  ${\mathsf A}^{S}_k$ is:
\begin{equation}
\label{expressAS2}
{\mathsf A}^{S}_k:\psi\in S_k\longmapsto -\mathsf Q_1\Delta \psi\in S_{k-2}.
\end{equation}
If $\Sigma$ is of class ${\mathcal C}^{k,1}$ then $S_k$ is a subspace of $H^{k+1}(\F)$ and the norm in $S_k$ is equivalent to the classical norm 
of $H^{k+1}(\F)$.
\end{lemma}
\begin{proof}
We recall that ${\mathsf A}^S_1$ is the operator $\psi\in S_1\longmapsto (\psi,\cdot)_{S_1}\in S_{-1}$.
The space $S_2$ is defined as the preimage of $S_0$ by ${\mathsf A}^S_1$, namely:
$$S_2=\{\psi\in S_1\,:\,(\psi,\cdot)_{S_1}=(f,\cdot)_{S_0}\text{ in }S_{-1}\text{ for some }f\text{ in }S_0\}.$$
Upon an integration by parts and according to Lemma~\ref{first_iso}, one easily obtains that the identity $(\psi,\cdot)_{S_1}=(f,\cdot)_{S_0}$ in $S_{-1}$ is equivalent 
to the equality $-\mathsf P\Delta \psi=\mathsf P f$ in $V_0$. Invoking Lemma~\ref{first_iso} again, we deduce, on the one hand, that $\mathsf P\Delta\psi=\Delta\psi$. 
Under the assumption on the regularity of the boundary $\Sigma$, the equality $-\Delta\psi=\mathsf Pf$ where $f$ and hence also $\mathsf Pf$ is in $H^1(\F)$ entails that $\psi$ belongs to $H^3(\F)$.
On the other hand, since $f$ belongs 
to $S_0$, $\mathsf Pf={\mathsf P}_1f$. Applying then the operator ${\mathsf Q}_1$ to both sides of the identity $-\Delta\psi={\mathsf P}_1 f$, 
we end up with the equality $-{\mathsf Q}_1\Delta\psi=f$ and \eqref{expressAS2} is proven for $k=2$. The expressions for $k>2$ follow 
from the general settings of Appendix~\ref{gelf_triple}. Then, by induction on $k$, invoking classical 
elliptic regularity results, one proves the inclusion $S_k\subset H^{k+1}(\F)$ and the equivalence of the norms.
\end{proof}
%
We straightforwardly deduce that, by definition of the space $V_1$, the inclusion $V_1\subset V_0$ enjoys the same properties as the inclusion $S_1\subset S_0$, namely 
it is continuous dense and compact. 
%
%
We consider then the Gelfand triple:
\begin{equation}
\label{gelfand_vorti}
V_1\subset V_0\subset V_{-1},
\end{equation}
in which $V_0$ is the pivot space and $V_{-1}$ is the dual space of $V_1$. As beforehand,
we define a family of embedded
Hilbert spaces $\{V_k,\,k\in\mathbb Z\}$, the corresponding family of isometries  $\{{\mathsf A}_k^V:V_k\to V_{k-2},\,k\in\mathbb Z\}$ and the positive constant:
\begin{equation}
\label{def_poincare_const_V}
\lambda_\F^V=\min_{\omega\in V_1\atop \omega\neq 0}\frac{\|\omega\|^2_{V_1}}{\|\omega\|^2_{V_0}}.
\end{equation}
%
\begin{rem}
\begin{my_enumerate}
\item
As already mentioned earlier, the space $V_{-1}$ is clearly not a distributions space.
\item
The guiding principle that the vorticity should be $L^2$-orthogonal to   harmonic functions is somehow still verified in a weak sense in $V_{-1}$. Indeed, 
$\mathsf A_1^V$ being an isometry, every element $\omega$ of $V_{-1}$ is equal to some $\mathsf A^V_1\omega'$ with $\omega'\in V_1$ and:
$$\langle \omega,\cdot\rangle_{V_{-1},V_1}=\langle\mathsf A_1^V \omega',\cdot\rangle_{V_{-1},V_1}=(\nabla\mathsf Q_1\omega',\nabla\mathsf Q_1\cdot)_{\mathbf L^2(\F)}.$$
Identifying the duality pairing  with the  scalar product in $L^2(\F)$ (i.e. the scalar product of the pivot space $V_0$), we obtain that formally $``(\omega,h)_{L^2(\F)}=0$'' for every $h\in\mathfrak H^1$. 
\end{my_enumerate}
\end{rem}
%
%
For the analysis of the spaces $V_k$ and their relations with the spaces $S_k$, it is worth introducing at this point
 an additional Gelfand triple, that will come in handy later on.
Thus, denote by ${ Z}_0$ the space $L^2(\F)$ (equipped with the usual scalar product) and by ${ Z}_1$ the space $S_0$. 
The configuration:
$${ Z}_1\subset { Z}_0\subset { Z}_{-1},$$
is obviously a Gelfand triple in which ${ Z}_0$ is the pivot space and ${ Z}_{-1}$ the dual space of ${ Z}_1$. As usual, following Appendix~\ref{gelf_triple}, we   define a family of 
embedded Hilbert spaces $\{{ Z}_k,\,k\in\mathbb Z\}$, and a family of isometries  $\{{\mathsf A}_k^{ Z}:{ Z}_k\to { Z}_{k-2},\,k\in\mathbb Z\}$.
Focusing on the case $k=2$, a simple integration by parts leads to:
\begin{lemma}
\label{LEM:S2A2}
The expressions of the space $ Z_2$ and of the operator ${\mathsf A}^{{ Z}}_2$ are respectively:
\begin{equation}
\label{def_S2_bold}
 Z_2=\bigg\{\psi\in H^2(\F)\cap S_0\,:\,\int_{\Sigma_j^-}\frac{\partial \psi}{\partial n}\ds=0,\quad j=1,\ldots,N\bigg\}
\quad\text{and}\quad
{\mathsf A}^{{ Z}}_2:\psi\in { Z}_2\longmapsto -\Delta\psi\in { Z}_0.
\end{equation}
\end{lemma}
The space of biharmonic functions in $L^2(\F)$ with zero mean flux  through the inner boundaries is denoted by $\mathfrak B$, namely:
\begin{equation}
\label{def_BV}
\mathfrak B=\Big\{\theta\in L^2(\F)\,:\,\Delta \theta\in\mathfrak H\text{ and }\int_{\Sigma^-_j}\frac{\partial \theta}{\partial n}\ds=0,\quad j=1,\ldots,N\Big\}.
\end{equation}
Since $\Delta \theta$ belongs to $L^2(\F)$, the trace of $\theta$ on $\Sigma$ is well defined and belongs to $H^{-{\frac12}}(\Sigma)$ and the trace of the normal derivative  is in $H^{-{\frac32}}(\Sigma)$. 
On the other hand, since $\Delta \theta$ belongs to $\mathfrak H$, its trace on $\Sigma$ is in $H^{-{\frac12}}(\Sigma)$ while its normal trace is in 
$H^{-{\frac32}}(\Sigma)$. 
\par
The space $S_1$ being a closed subspace of $ Z_2$  it admits an orthogonal complement denoted by $\mathfrak B_S$:
\begin{equation}
\label{decomp_S2b}
{ Z}_2=S_1\sumperp  {\mathfrak B_S}.
\end{equation}
An integration by parts and classical elliptic regularity results allow to deduce that:
\begin{equation}
\label{def_Scurl}
\mathfrak B_S=S_0\cap\mathfrak B,
\end{equation}
and that $\mathfrak B_S\subset H^2(\F)$.

\begin{lemma}
\label{decomp_S2}
The operator ${\mathsf A}^{ Z}_2$ is an isometry from $S_1$ onto $V_0$ and also an isometry from $\mathfrak B_S$ onto $\mathfrak H$, i.e. 
the operator ${\mathsf A}^{ Z}_2$ is block-diagonal with respect to the following decompositions   of the spaces:
$${\mathsf A}^{ Z}_2:S_1\sumperp\mathfrak B_S\longrightarrow V_0\sumperp \mathfrak H.$$
The operators ${\mathsf A}^{V}_1$ and ${\mathsf A}^{{ Z}}_1$ and the operators ${\mathsf A}^{V}_2$ and ${\mathsf A}^{{ Z}}_2$ 
are connected via the identities:
\begin{equation}
\label{eq:nice_relations}
{\mathsf A}^{V}_1={\mathsf Q}_1^\ast {\mathsf A}^{{ Z}}_1 {\mathsf Q}_1\qquad\text{and}\qquad {\mathsf A}^{V}_2={\mathsf A}^{{ Z}}_2 {\mathsf Q}_1\quad\text{in }V_2,
\end{equation}
where the operator ${\mathsf Q}_1^\ast$ is the adjoint of $\mathsf Q_1$.
\end{lemma}
%
%
\begin{proof}
The first claim of the lemma is a direct consequence of Lemma~\ref{first_iso}. 
\par
By definition, for every $\omega\in V_1$:
$${\mathsf A}^{V}_1\omega=(\nabla {\mathsf Q}_1 \omega,\nabla {\mathsf Q}_1\cdot)_{\mathbf L^2(\F)}={\mathsf Q}_1^\ast {\mathsf A}^{{ Z}}_1 {\mathsf Q}_1\omega,$$
and the first identity in \eqref{eq:nice_relations} is proved. Addressing the latter, notice that for every $\omega\in V_2$, 
the function $w={\mathsf A}^{V}_2 \omega$ is the unique element in $V_0$ such that:
$$(w,v)_{V_0}=(\omega,v)_{V_1},\forallt v\in V_1,$$
which can be rewritten as:
$$
(w,v)_{L^2(\F)}=(\nabla {\mathsf Q}_1 \omega,\nabla {\mathsf Q}_1 v)_{\mathbf L^2(\F)},\forallt v\in V_1.
$$
But the functions $v$ and ${\mathsf Q}_1v$ differ only up to an element of $\mathfrak H$ and since $w$ belongs to $V_0=\mathfrak H^\perp$, it follows 
that:
$$(w,v)_{L^2(\F)}=(w,{\mathsf Q}_1v)_{L^2(\F)},\forallt v\in V_1.$$
Finally, since $S_0= Z_1$ and ${\mathsf Q}_1:V_1\to S_0$ is an isometry, the function $w$ satisfies:
$$(w,z)_{L^2(\F)}=(\nabla {\mathsf Q}_1 \omega,\nabla z)_{\mathbf L^2(\F)},\forallt z\in { Z}_1,$$
which means that $w={\mathsf A}^{{ Z}}_2 {\mathsf Q}_1 \omega$ and completes the proof. 
\end{proof}
%
We can now go back to the study of the vorticity spaces $V_k$ and the related operators ${\mathsf A}^V_k$.
Starting with the case $k=2$, we claim:
\begin{lemma}
The space $V_2$ is equal to ${\mathsf P}_1 S_1$, or equivalently:
\begin{subequations}
\begin{equation}
V_2=\bigg\{\omega\in \mathsf P  H^2(\F)\,:\,\frac{\partial {\mathsf Q}_1\omega}{\partial n}\Big|_\Sigma=0\bigg\}.
\end{equation}
Moreover, the expression of the operator ${\mathsf A}_2^V$ is:
\begin{equation}
\label{expressA2V}
{\mathsf A}_2^V:\omega\in V_2\longmapsto -\Delta\omega\in V_0.
\end{equation}
\end{subequations}
\end{lemma}
\begin{proof}
The second formula in \eqref{eq:nice_relations} yields the following identity between function spaces: 
$${\mathsf A}_2^VV_2= {\mathsf A}_2^{ Z}{\mathsf Q}_1V_2$$
and then, since ${\mathsf A}_2^V V_2=V_0$:
$$({\mathsf A}_2^{ Z})^{-1}V_0={\mathsf Q}_1V_2.$$
Invoking the first point of Lemma~\ref{decomp_S2} we deduce first that ${\mathsf Q}_1V_2=S_1$ and then, applying the operator ${\mathsf P}_1$ to both sides of the identity, that $V_2={\mathsf P}_1S_1$. 
Using again the second formula in \eqref{eq:nice_relations} together with the expression of ${\mathsf A}_2^{ Z}$ given in \eqref{def_S2_bold}, we obtain the expression \eqref{expressA2V} of 
the operator ${\mathsf A}^V_2$.
\end{proof}
\begin{rem}
\label{regul_V2}
According to Lemma~\ref{regul_P0}, if $\Sigma$ is of class ${\mathcal C}^{3,1}$ then $V_2$ is a subspace of $H^2(\F)$. If $\Sigma$ is of class ${\mathcal C}^{2,1}$, $V_2$ is a subspace 
of $H^1(\F)$ and the functions in $V_2$ can be given a trace in $H^{{\frac12}}(\Sigma)$ and a normal trace in $H^{-{\frac32}}(\Sigma)$. Finally, if $\Sigma$ is only of class ${\mathcal C}^{1,1}$, 
the trace still exists in $H^{-{\frac12}}(\Sigma)$ and the normal trace in $H^{-{\frac32}}(\Sigma)$. We use the fact that every function in $V_2$ is by definition the sum of a function in 
$H^2(\F)$ with a harmonic function in $L^2(\F)$.
\end{rem}
%
For the ease of the reader, we can still state the following lemma which is a straightforward consequence of \eqref{expressA2V} and the general settings of 
Appendix~\ref{gelf_triple}:
%
\begin{lemma}
\label{express_AVk}
For every positive integer $k$, the expression of the operators ${\mathsf A}_k^V$ are:
$${\mathsf A}_1^V:u\in V_1\longmapsto ( u,\cdot)_{V_1}\in V_{-1}\qquad\text{ and }\quad {\mathsf A}_k^V:u\in V_k\longmapsto (-\Delta) u\in V_{k-2}
\quad\text{ for }k\geqslant 2.$$
For nonnegative integers $k$, $V_k$ is a subspace of $H^{k}(\F)$ providing that 
$\Sigma$ is of class ${\mathcal C}^{k+1,1}$ and the norm in $V_k$ is equivalent to the classical norm 
of $H^{k}(\F)$.
For nonpositive indices, the operators are defined by duality as follows:
$${\mathsf A}_{-k}^V:u\in V_{-k}\longmapsto\langle u,(-\Delta)\cdot\rangle_{V_{-k},V_k}\in V_{-k-2},\qquad(k\geqslant 0).$$

\end{lemma}
%
The next result states that the chain of embedded spaces for the stream function $\{S_k,\,k\in\mathbb Z\}$ 
is globally isometric to the chain of embedded spaces for the vorticity $\{V_k,\, k\in\mathbb Z\}$, the isometries being, loosely speaking, 
the operators $\mathsf P$ and $\mathsf Q$. 
So far, we have proven that ${\mathsf P}_1S_1=V_2$ and ${\mathsf P}_1S_0=V_1$. To generalized these relations to every integer $k$, we need to extend 
 the operators ${\mathsf P}_1$ and ${\mathsf Q}_1$.
%
\begin{lemma}
\label{lem:bef_theo}
For every positive integer $k$, the following inclusions hold:
$${\mathsf P}_1S_{k-1}\subset V_k\qquad\text{and}\qquad{\mathsf Q}_1V_k\subset S_{k-1}.$$
\end{lemma}
%
Considering this lemma as granted, it makes sense to define for every positive integer $k$ the operators:
\begin{equation}
\label{defPkQk}
{\mathsf P}_k:u\in S_{k-1}\longmapsto {\mathsf P}_1u\in V_k\qquad\text{and}\qquad {\mathsf Q}_k:u\in V_k\longmapsto {\mathsf Q}_1 u\in S_{k-1}.
\end{equation}
Then, we define also by induction, for every $k\geqslant 0$:
\begin{subequations}
\label{extend-kneg}
\begin{align}
{\mathsf P}_{-k}&={\mathsf A}^{V}_{-k+2}{\mathsf P}_{-k+2}({\mathsf A}^{S}_{-k+1})^{-1}:S_{-k-1}\to V_{-k},
\intertext{and}
\label{def_Qkn}
{\mathsf Q}_{-k}&={\mathsf A}^{S}_{-k+1} {\mathsf Q}_{-k+2} ({\mathsf A}^{V}_{-k+2})^{-1}:V_{-k}\to S_{-k-1}.
\end{align}
\end{subequations}
%
\begin{theorem}
\label{LEM:P_kQ_k}
%
For every integer $k$, the operators ${\mathsf P}_k$ and ${\mathsf Q}_k$ defined in \eqref{defPkQk} and  \eqref{extend-kneg} are   inverse isometries 
(i.e. $\mathsf P_k\mathsf Q_k={\rm Id}$ and $\mathsf Q_k\mathsf P_k={\rm Id}$). 
Moreover, formulas  \eqref{extend-kneg}  can be generalized to every integer $k$:
\begin{equation}
\label{extend-kneg_2}
 {\mathsf A}_k^V  {\mathsf P}_{k}={\mathsf P}_{k-2}{\mathsf A}^S_{k-1}\qquad\text{and}\qquad  {\mathsf A}^S_{k-1} {\mathsf Q}_k={\mathsf Q}_{k-2} {\mathsf A}^V_k,
\end{equation}
and for every pair of indices $k,k'$ such that $k'\leqslant k$:
\begin{equation}
\label{eq:Pkextend}
{\mathsf P}_{k'}={\mathsf P}_{k}\quad\text{in }S_{k-1}\qquad\text{and}\qquad {\mathsf Q}_{k'}={\mathsf Q}_{k}\quad\text{in }V_k.
\end{equation}
\end{theorem}
%
\begin{rem}
\label{rem:omega_no_vorti}
By definition, $\mathsf P$ is a projection in $L^2(\F)$ and $\mathsf Q$ a projection in $H^1(\F)$. The theorem tells us that formulas \eqref{extend-kneg} allow extending these projectors to larger spaces.
\par
The theorem ensures also that to every stream function $\psi$ in some space $S_{k-1}$, it can be associated a vorticity field 
$\omega={\mathsf P}_k\psi$ in $V_k$. 
The vorticity $\omega$ has the same regularity as $\psi$ and is obviously not the Laplacian of $\psi$.
\end{rem}
%
The lemma and the theorem are proved at once:
\begin{proof}[Proof of Lemma~\ref{lem:bef_theo} and Theorem~\ref{LEM:P_kQ_k}]
Denoting by $\mathsf Q_2$ the restriction of $\mathsf Q_1$ to $V_2$, 
the second formula in \eqref{eq:nice_relations} can be rewritten as ${\mathsf A}_2^V={\mathsf A}_2^{ Z}{\mathsf Q}_2$. Since the 
operators ${\mathsf A}_2^V$ and ${\mathsf A}_2^{ Z}$ are both isometries, this property is also shared by ${\mathsf Q}_2$ and its inverse, which
is denoted by ${\mathsf P}_2$. 
We have now at our disposal two pairs of isometries $({\mathsf P}_1,{\mathsf Q}_1)$ and $({\mathsf P}_2,{\mathsf Q}_2)$ corresponding to two successive indices in the chain 
of embedded spaces. Furthermore, ${\mathsf P}_2$ is the restriction of ${\mathsf P}_1$ to $V_2$. This fits within the framework of Subsection~\ref{isometric_chain}. 
We define first ${\mathsf P}_k$ and ${\mathsf Q}_k$ (for the indices $k\neq 1,2$) by induction with formulas \eqref{extend-kneg_2} and we apply Lemma~\ref{p_kp_k}.  
We obtain that the operators ${\mathsf P}_k$ and ${\mathsf Q}_k$ are indeed isometries from $S_{k-1}$ onto $V_k$ and from $V_k$ onto $S_{k-1}$ respectively. 
Lemma~\ref{p_kp_k} also ensures that ${\mathsf P}_k={\mathsf P}_{k'}$ in $V_k$ and ${\mathsf Q}_k={\mathsf Q}_{k'}$ in $S_{k-1}$ for indices $k'\leqslant k$, whence we deduce that ${\mathsf P}_k$ and 
${\mathsf Q}_k$ for $k\geqslant 1$ can be equivalently defined by \eqref{defPkQk}. Next, since ${\mathsf P}_1$ and ${\mathsf Q}_1$ are reciprocal isometries and ${\mathsf P}_k$ and ${\mathsf Q}_k$ 
are just restrictions of ${\mathsf P}_1$ and ${\mathsf Q}_1$, then ${\mathsf P}_k$ and ${\mathsf Q}_k$ are reciprocal isometries as well. We draw the same conclusion for nonpositive 
indices using formulas  \eqref{extend-kneg_2} and   complete the proof. \end{proof}
%
\begin{cor}
\label{COR:lambda}
The constant $\lambda_\F^S$ defined in \eqref{def_poincare_const_S} and the constant $\lambda_\F^V$ defined in \eqref{def_poincare_const_V} are equal. 
We denote simply by $\lambda_\F$ their common value. 
\end{cor}
Notice also that since $S_1\subset Z_2$ and $S_0= Z_1$, we have: 
\begin{equation}
\label{estim_eigen}
\lambda_\F^{ Z}=\min_{\psi\in  Z_2\atop \psi\neq 0}\frac{\|\psi\|^2_{ Z_2}}{\|\psi\|^2_{ Z_1}}
\leqslant \min_{\psi\in S_1\atop \psi\neq 0}\frac{\|\psi\|^2_{S_1}}{\|\psi\|^2_{S_0}}=\lambda_\F.
\end{equation}
Remark~\ref{rem:omega_no_vorti} points out that, for a given stream function $\psi$ in some space $S_{k-1}$, the function 
$\omega={\mathsf P}_k\psi$ is not 
the (physical) vorticity corresponding to the  velocity field $\nabla^\perp\psi$. We shall now define the operators $\Delta_k$ that associates the stream function to 
its corresponding vorticity field.
\begin{definition}
For every integer $k$, the operator $\Delta_k:S_{k+1}\to V_{k}$ is defined equivalently (according to \eqref{extend-kneg_2}) by:
\begin{equation}
\label{def_Bk}
\text{Either}\quad \Delta_k = -{\mathsf A}^{V}_{k+2}{\mathsf P}_{k+2}\qquad\text{or}\qquad \Delta_k=-{\mathsf P}_k{\mathsf A}^S_{k+1}.
\end{equation}
%
%
%
\end{definition}
%
%
The main properties of the operators $\Delta_k$ are summarized in the following lemma:
\begin{lemma}
\label{LEM:prop_delta}
The following assertions hold:
\begin{my_enumerate}
\item
For every integer $k$, the operator $\Delta_k$ is an isometry. 
\item
For every pair of integers $k,k'$ such that $k'\leqslant k$, $\Delta_k=\Delta_{k'}$ in $S_{k+1}.$
\item For every nonnegative integer $k$, the operator $\Delta_k$ is the classical Laplacian operator. 
%
%
\end{my_enumerate}
\end{lemma}
%
\begin{proof}
We recall that the operators ${\mathsf P}_k$ and ${\mathsf A}_k^V$ are isometries, what yields the first point of the lemma. The second point is a consequence of \eqref{eq:Pkextend} 
and the similar general property \eqref{Akexpand} satisfied by the operators ${\mathsf A}^V_k$. The third point is a direct consequence of Lemma~\ref{express_AVk}.
\end{proof}
 For any integer $k$, the operators ${\mathsf Q}_k$ and $-\Delta_{-k}$ can be shown to be somehow adjoint:
%
\begin{lemma}
	\label{LEM:prop:delta_negatif}
For negative indices, the operators $\Delta_k$ and ${\mathsf Q}_k$ satisfy the adjointness relations below:
\begin{subequations}
\begin{alignat}{3}
\label{alternQk}
{\mathsf Q}_{-k}=-\Delta_k^\ast&:\omega\in V_{-k}\longmapsto-\langle \omega,\Delta_k\cdot\rangle_{V_{-k},V_k}\in S_{-k-1}&\quad&(k\geqslant 0),\\
\label{alternQk_2}
\Delta_{-k}=-{\mathsf Q}_k^\ast&:\psi\in S_{-k+1}\longmapsto-\langle \psi,{\mathsf Q}_k\cdot\rangle_{S_{-k+1},S_{k-1}}\in V_{-k}&\quad&(k\geqslant 1).
\end{alignat}
\end{subequations}
%
\end{lemma}
\begin{proof} 
We prove \eqref{alternQk} by induction, the proof of \eqref{alternQk_2} being similar.  
According to \eqref{extend-kneg_2}, ${\mathsf Q}_{0}{\mathsf A}^V_2={\mathsf A}^S_1{\mathsf Q}_2$ what means, recalling \eqref{eq:defA1} that:
$${\mathsf Q}_{0}{\mathsf A}^V_2\omega=({\mathsf Q}_2\omega, \cdot)_{S_1}=(\omega,{\mathsf P}_2\cdot)_{V_2}=({\mathsf A}_2^V\omega,{\mathsf A}_2^V{\mathsf P}_2\cdot)_{V_0}=({\mathsf A}_2^V\omega,(-\Delta_0)\cdot)_{V_0},$$
where we have used the fact that the operators ${\mathsf P}_2$ and ${\mathsf A}^V_2$ are isometries. This proves \eqref{alternQk}  at the step $k=0$.
\par
According to the definition \eqref{def_Qkn} with $k=1$, ${\mathsf Q}_{-1}={\mathsf A}_0^S{\mathsf Q}_1({\mathsf A}_1^V)^{-1}$. Equivalently stated, for every $\omega\in V_1$:
$${\mathsf Q}_{-1}{\mathsf A}_1^V\omega= {\mathsf A}_0^S{\mathsf Q}_1\omega=({\mathsf Q}_1\omega,{\mathsf A}_2^S\cdot)_{S_0}\forallt \omega\in V_1,$$
where we have used the general relation \eqref{def_A0}. The operator ${\mathsf P}_1$ being an isometry:
$$({\mathsf Q}_1\omega,{\mathsf A}_2^S\cdot)_{S_0}=({\mathsf P}_1{\mathsf Q}_1\omega,{\mathsf P}_1{\mathsf A}_2^S\cdot)_{V_1}=-(\omega,\Delta_1\cdot)_{V_1}=-\langle {\mathsf A}_1^V\omega,\Delta_1\cdot\rangle_{V_{-1},V_1},$$
and  \eqref{alternQk} is then proved for $k=1$.
\par
Let us assume that \eqref{alternQk} holds true at the step $k-2$ for some integer $k\geqslant 2$. According to \eqref{extend-kneg_2}, ${\mathsf Q}_{-k}{\mathsf A}_{-k+2}^V={\mathsf A}^S_{-k+1}{\mathsf Q}_{-k+2}$ whence, recalling the definition \eqref{eq:def_Akdual}:
$${\mathsf Q}_{-k}{\mathsf A}_{-k+2}^V\omega=\langle {\mathsf Q}_{-k+2}\omega,{\mathsf A}^S_{k+1}\cdot\rangle_{S_{-k+1},S_{k-1}}\forallt \omega\in V_{-k+2}.$$
Using the induction hypothesis for the operator ${\mathsf Q}_{-k+2}$, it comes:
$${\mathsf Q}_{-k}{\mathsf A}_{-k+2}^V\omega=-\langle  \omega,\Delta_{k-2}{\mathsf A}^S_{k+1}\cdot\rangle_{V_{-k+2},V_{k-2}}=
-\langle\omega,{\mathsf A}_k^V\Delta_k\cdot\rangle_{V_{-k+2},V_{k-2}}
\forallt \omega\in V_{-k+2},$$
where the latter identity results from \eqref{def_Bk}. Keeping in mind \eqref{eq:def_Akdual}, we have indeed proved \eqref{alternQk} at the step $k$ and 
complete the proof.
\end{proof}
%
\subsection{Biot-Savart operator}
With a slight abuse of terminology, the inverse of the operator $\Delta_k$  denoted by $\mathsf N_k$ will be referred to as the Biot-Savart operator.  
Quite surprisingly, the expression of this operator is independent from the fluid domain $\F$. We recall that the  fundamental solution of the Laplacian is 
the function:
$$\mathscr G(x)=\frac{1}{2\pi}\ln|x|,\qquad x\in\mathbb R^2\setminus\{0\}.$$
In the rest of the paper, we will denote generically by $\mathbf c$ the real constants that should arise in the estimates. The value of the constant 
may change  from line to line. The parameters the constant should depend on is indicated in subscript.
\begin{theorem}
\label{biot-savart_simple}
For every nonnegative index  $k$, the Biot-Savart operator $\mathsf N_k=(\Delta_k)^{-1}$ is simply the Newtonian potential defined by:
\begin{equation}
\label{express_newton}
\mathsf N_k:\omega\in V_k\longmapsto \mathsf N \omega\in S_{k+1},
\end{equation}
where $\mathsf N \omega=\mathscr G\ast \omega$ ($\omega$ is extended by $0$ outside $\F$).

\end{theorem}
\begin{proof}
Let $\omega$ be in $V_0$ (extended by $0$ outside $\F$) and denote by $\psi$ the Newtonian potential $\mathsf N\omega$. According to  classical properties of 
the Newtonian potential, $\Delta\psi=\omega$ in $\mathbb R^2$ and therefore
it suffices to verify that $\psi$ is constant on every connected part of the boundary $\Sigma$ and that  its normal derivative vanishes. Define 
the constant $\delta=2\max\{|x-y|\,:\,x\in \Sigma^-,\, y\in\Sigma^+\}$.
Let $j$ be in $\{1,\ldots,N\}$ and let $q$ be in $L^2(\Sigma^-_j)$. 
Notice now that there exists a positive constant $\mathbf c_{\Sigma^-_j}$ such that:
\begin{align*}
\int_{\Sigma^-_j}\int_{\mathbb R^2}|\mathscr G(x-y)\omega(y)q(x)|{\rm d}y\ds_x&=
\int_{\Sigma^-_j}\int_{B(0,\delta)}|\mathscr G(z)\omega(x-z)q(x)|{\rm d}z\ds_x\\
&\leqslant \mathbf c_{\Sigma^-_j}\|\omega\|_{V_0}\|\mathscr G\|_{L^2(B(0,\delta))}\|q\|_{L^2(\Sigma^-_j)}.
\end{align*}
We are then allowed to apply Fubini's theorem, which yields:
$$\int_{\mathbb R^2}\omega(y)\left(-\int_{\Sigma^-_j}\mathscr G(x-y)q(x)\ds_x\right){\rm d}y+
\int_{\Sigma^-_j}\left(\int_{\mathbb R^2}\mathscr G(x-y)\omega(y){\rm d}y\right)q(x)\ds_x=0,
$$
and this identity can be rewritten as:
\begin{equation}
\label{eq:fubini}
\int_\F \Delta\psi (\mathsf S_j q)\dx+\int_{\Sigma^-_j}\psi\,q\ds=0,
\end{equation}
where $\mathsf S_j q$ is the single layer potential of density $q$ supported on the boundary $\Sigma_j^-$, that is:
$$\mathsf S_j q(x)=-\int_{\Sigma^-_j}\mathscr G(x-y)q(y)\ds_y\forallt x\in\mathbb R^2\setminus \Sigma_j^-.$$
The simple layer potential $\mathsf S_j q$ is harmonic in $\mathbb R^2\setminus \Sigma_j^-$ (we refer to the book of McLean \cite{McLean:2000aa} 
for details about layer potentials) and we denote respectively by $\mathsf S_j q^+$ and 
$\mathsf S_j q^-$ the restriction of $\mathsf S_j q$ to the unbounded and bounded connected components of $\mathbb R^2\setminus \Sigma_j^-$. 
The functions $\mathsf S_j q^+$ and 
$\mathsf S_j q^-$ share the same trace on $\Sigma^-_j$ and $q$ is the jump of the normal derivative across the boundary $\Sigma^-_j$:
\begin{equation}
\label{eq:jump_formula}
q=\frac{\partial}{\partial n}\mathsf S_j^+q-\frac{\partial}{\partial n}\mathsf S_j^- q\qquad\text{on }\Sigma_j^-.
\end{equation}
From the obvious equality
$$\int_{\Sigma^-_j}\frac{\partial}{\partial n}\mathsf S_j^- q\ds=0,$$
we deduce that the harmonic function $\mathsf S_j^+q$ has zero mean flux through the boundary $\Sigma^-_j$ if and only if 
$$\int_{\Sigma^-_j}q \ds=0.$$
On the other hand, for indices $k\neq j$, we have also:
$$\int_{\Sigma^-_k}\frac{\partial}{\partial n}\mathsf S_j^+q\ds=0,$$
because the normal derivative of $\mathsf S_j^+q$ is continuous across $\Sigma^-_k$ and $\mathsf S_j^+q$ is harmonic inside $\Sigma^-_k$. From \eqref{eq:fubini}, we infer that for every $\omega\in V_0$ and 
every $q\in L^2(\Sigma^-_j)$ with zero mean value:
$$\int_{\Sigma^-_j}\psi\,q\ds=0,$$
and then that $\psi$ is constant on $\Sigma^-_j$. For every $q$ in $L^2(\Sigma^+)$, the corresponding single layer potential  $\mathsf S_0q$ supported on 
$\Sigma^+$ has zero mean flux through every inner boundary $\Sigma^-_k$ (for $k\in\{1,\ldots,N\}$) and we deduce from \eqref{eq:fubini} again that the trace of $\psi$ is nul on 
$\Sigma^+$. It follows that $\psi$ is in the space $S_0$.
\par
Let now $h$ be a harmonic function in $\F$ and assume that the normal derivative of $h$ belongs to $L^2(\Sigma)$. Then there exists $q_0\in L^2(\Sigma^+)$ 
and $q_j\in L^2(\Sigma^-_j)$ ($j=1,\ldots,N$) such that:
$$h=\mathsf S_0q_0+\sum_{j=1}^N\mathsf S_j q_j\qquad\text{in }\F.$$
Using again \eqref{eq:fubini} and the fact that $\psi$ is nul on $\Sigma^+$ and constant on $\Sigma^-$, we deduce that:
$$\int_\F\Delta\psi h\dx+\int_{\Sigma^-}\psi\frac{\partial h}{\partial n}\ds=0,$$
and therefore, integrating by parts, that:
$$\int_\Sigma \frac{\partial\psi}{\partial n} h\ds=0.$$
This last equality being true for every harmonic function $h$, it follows that the normal derivative of $\psi$ is nul on $\Sigma$ and 
therefore that $\psi$ belongs to $S_1$. The proof is now completed.
\end{proof}
%
The expression \eqref{express_newton} when $\F=\mathbb R^2$ can be found in the book \cite[\S 2.1]{Majda:2002aa}.
For the Euler equations in a domain with holes, the Biot-Savart operator (in the sense considered above, that is the operator allowing 
recovering the stream function) is given by (see  \cite{Lopes-Filho:2007aa} for a proof):
\begin{equation}
\label{biot-savart_euler}
\mathsf N^E\omega(x)=\int_\F \mathscr K(x,y)\omega(y)\,{\rm d}y+\sum_{j=1}^N\left(\varGamma_j-\alpha_j(\omega) \right)\xi_j(x)\forallt x\in\F.
\end{equation}
In this identity:
\begin{my_enumerate}
\item
  $\mathscr K:\F\times\F\to \mathbb R$ is the Green's function of the domain $\F$. It is defined by:
  $$\mathscr K(x,y)=\mathscr G(x-y)-\mathscr H(x,y)\forallt (x,y)\in\F\times\F\text{ s.t. }x\neq y,$$
  where, for every $x$ in $\F$, the function $\mathscr H(x,\cdot)$ is harmonic in $\F$ and satisfies 
  $$\mathscr H(x,\cdot)=\mathscr G(x-\cdot)\quad\text{ on }\Sigma.$$
\item The real constants $\alpha_j(\omega)$ are given by:
$$\begin{bmatrix}
\alpha_1(\omega)\\
\vdots\\
\alpha_N(\omega)
\end{bmatrix}
=
\begin{bmatrix}
(\xi_1,\xi_1)_{S_0}&\ldots&(\xi_1,\xi_N)_{S_0}\\
\vdots&&\vdots\\
(\xi_N,\xi_1)_{S_0}&\ldots&(\xi_N,\xi_N)_{S_0}
\end{bmatrix}^{-1}
\begin{bmatrix}
\int_\F\omega(y)\xi_1(y)\,{\rm d}y\\
\vdots\\
\int_\F\omega(y)\xi_N(y)\,{\rm d}y
\end{bmatrix}
,$$
where we recall the the functions $\xi_j$ ($j=1,\ldots,N$) are defined in Section~\ref{SEC:main_spaces}.
\item For every $j$, the scalar $\varGamma_j$ is the circulation of the fluid around the inner boundary $\Sigma^-_j$. 
\end{my_enumerate}
Notice that for the Euler equations in a multiply connected domain, both the vorticity and the circulation are necessary to recover the stream function.
\par
In case the vorticity is in $V_0$ and in the absence of circulation, then the Biot-Savart operator  for NS equations and the Biot-Savart operator for Euler equations give the same stream function:
\begin{prop}
\label{PROP:equiv_BS}
Let $\omega$ be in $V_0$ and assume that the flow is such that $\varGamma_j=0$ for every $j=1,\ldots,N$. Then $\mathsf N_0\omega$ 
(defined in \eqref{express_newton}) and $\mathsf N^E\omega$ (defined in \eqref{biot-savart_euler}) are equal.
\end{prop}
The proof relies on the following lemma in which we denote simply by $\mathsf S$ the simple layer potential supported on the whole boundary $\Sigma$.
\begin{lemma}
\label{LEM:pour_prop}
For every function $h\in\mathbb F_S$ (i.e. $h$ harmonic in $\F$ and $h$ belongs to $S_0$):
$$\left(\mathsf S\frac{\partial h}{\partial n}\Big|_\Sigma\right)(x)=h(x)\forallt x\in\F.$$
\end{lemma}
\begin{proof}
Let $h$ be in $S_0$. Basic results of potential theory ensures that there exists a unique $p\in H^{-\frac12}(\Sigma)$ 
such that $\mathsf Sp(x)=h(x)$ for every $x\in \F$. The single layer potential $\mathsf Sp$ is harmonic in $\mathbb R^2\setminus\Sigma$ and 
belongs to $H^1_{\ell oc}(\mathbb R^2)$ (which means that the trace of the function matches on both sides of the boundary $\Sigma$). Since 
the trace of $h$ is equal to 0 on $\Sigma^+$, the single layer potential vanishes identically on the unbounded connected component of $\mathbb R^2\setminus\Sigma^+$. For similar reasons, $\mathsf Sp$ is constant inside $\Sigma^-_j$ for every $j=1,\ldots,N$. According to the jump formula
\eqref{eq:jump_formula}, we obtain that $p=\partial h/\partial n$ on $\Sigma$ and the proof is completed.
\end{proof}
We can move on to the:
\begin{proof}[Proof of Proposition~\ref{PROP:equiv_BS}]
For every $x\in\F$, the function:
$$\mathscr H_0(x,\cdot)=\mathscr H(x,\cdot)-\sum_{j=1}^N(\nabla\mathscr H(x,\cdot),\nabla\hat\xi_j)_{\mathbf L^2(\F)}\hat\xi_j,$$
belongs to $\mathfrak H$. Integrating by parts the terms in the sum, we obtain for every $j=1,\ldots,N$:
$$(\nabla\mathscr H(x,\cdot),\nabla\hat\xi_j)_{\mathbf L^2(\F)}=\int_\Sigma\mathscr H(x,y)\frac{\partial\hat\xi_j}{\partial n}(y)\ds_y
=-\left(\mathsf S\frac{\partial\hat\xi_j}{\partial n}\right)(x)=-\hat\xi_j(x),$$
according to Lemma~ \ref{LEM:pour_prop}. Let now $\omega$ be in $V_0$. Then, providing that $\Gamma_j=0$ 
for every $j=1,\ldots,N$:
\begin{equation}
\label{eq:main_propp}
\mathsf N^E\omega(x)=\mathsf N_0\omega(x) 
-\int_\F\mathscr H_0(x,y)\omega(y)\,{\rm d}y 
+\sum_{j=1}^N\hat\alpha_j(\omega)\hat\xi_j(x)-\sum_{j=1}^N \alpha_j(\omega) \xi_j(x)\forallt x\in\F,
\end{equation}
where, for every $j=1,\ldots,N$:
$$\hat\alpha_j(\omega)=\int_\F\omega(y)\hat\xi_j(y)\,{\rm d}y.$$
The second term in the right hand side of \eqref{eq:main_propp} vanishes by definition of $V_0$ and both last terms cancel out since they 
stand for the same linear application expressed in two different bases of $\mathbb F_S$.
\end{proof}
%
 It remains now to link the spaces $S_k$ for the stream functions to the spaces $\mathbf J_k$ for the velocity fields. 
We recall the definitions \eqref{def:Jk} of the spaces $\mathbf J_0$ and $\mathbf J_1$. For every other integers $k$, the spaces $\mathbf J_k$ are 
classically defined from the Gelfand triple $\mathbf J_1\subset \mathbf J_0\subset \mathbf J_{-1}$, as well as the isometries ${\mathsf A}^{\mathbf J}_k:\mathbf J_k\to \mathbf J_{k-2}$. The following Lemma can be found in  \cite{Guermond:1994aa}:
 %
  \begin{lemma}
  \label{LEM:guermond}
  The operators $\nabla_0^\perp:\psi\in S_0\mapsto \nabla^\perp\psi\in \mathbf J_0$ and 
  $\nabla_1^\perp:\psi\in S_1\mapsto \nabla^\perp\psi\in \mathbf J_1$ are well defined and are isometries.
  \end{lemma}
   %
 Applying the abstract results of Section~\ref{isometric_chain}, we deduce:
 %
  \begin{lemma}
  \label{LEM:iso_stokes}
  For every index $k$, it can be defined an isometry:
 $$\nabla^\perp_k:S_k\longmapsto \mathbf J_k,$$
 such that, for every pair of indices $k\leqslant k'$, $\nabla_k^\perp=\nabla_{k'}^\perp$ in $S_{k'}$ and Diagram~\ref{diag_3}
  commutes.
  \end{lemma}
%
\begin{figure}[ht]
 $$
  \xymatrix @!0 @R=20mm @C=30mm {
 \mathbf J_{k+1}     \ar[r]^{{\mathsf A}^{\mathbf J}_{k+1}} 	  
 %
 & \mathbf J_{k-1}    \\
    S_{k+1}\ar[u]^{\nabla^\perp_{k+1}}
    \ar[r]^{{\mathsf A}^{S}_{k+1}} 	  
    &S_{k-1}\ar[u]^{\nabla^\perp_{k-1}}
  }
  $$
  \caption[The spaces ${\mathbf J}_{k}$, $S_k$ and associated operators]{\label{diag_3}The top row contains the function $\mathbf J_k$ for the velocity field  and 
  the bottom row contains the spaces $S_k$ for the stream functions. All the operators are isometries.}
  \end{figure}
\begin{rem}
Let be given a sequence $(\psi_n)_n$ in $S_0$ and $\bar\psi\in S_0$. Define the corresponding velocity fields $u_n=\nabla^\perp_0\psi_n$
 and $\bar u=\nabla^\perp_0\bar\psi$ and the vorticity fields 
$\omega_n=\Delta_{-1}\psi_n$ and $\bar\omega=\Delta_{-1}\bar\psi$. 
Then, the following assertions are equivalent:
\begin{subequations}
\label{eq:kekkiher}
\begin{align}
\psi_n &\rightharpoonup\bar\psi\qquad \text{in }S_0,\\
u_n&\rightharpoonup\bar u\qquad \text{in }\mathbf J_0,\\
\omega_n&\rightharpoonup\bar\omega\qquad \text{in }V_{-1}.
\end{align}
\end{subequations}
Let  a time $T>0$ be given and suppose now that $(\psi_n)_n$ is a sequence in $L^\infty([0,T];S_0)$ and that $\bar\psi$ lies in $L^\infty([0,T];S_0)$. Then 
the velocity fields $u_n$ and $\bar u$ belongs to $L^\infty([0,T];\mathbf J_0)$ and the vorticity fields 
$\omega_n$ and  $\bar\omega$ are in $L^\infty([0,T];V_{-1})$.
In the context of vanishing viscosity limit, assume that $\bar u$ is a solution to the Euler equations and that $u_n$ is 
a solution to the NS equations with a viscosity that tends to zero along with $n$.
Following  Kelliher \cite{Kelliher:2008aa}, the vanishing viscosity limit holds when $u_n\rightharpoonup\bar u$ in $\mathbf J_0$, uniformly on $[0,T]$.
According to \eqref{eq:kekkiher}, this conditions is then equivalent to either
$\psi_n \rightharpoonup\bar\psi$ in $S_0$ , uniformly on $[0,T]$ or to $\omega_n \rightharpoonup\bar\omega$ in $V_{-1}$, uniformly on $[0,T]$; 
see also Remark~\ref{first_kato:rem}.
\end{rem}
Most of the material elaborated so far in this section is summarized in the commutative diagram of Fig.~\ref{diag_2}, which contains 
the main operators and their relations. 
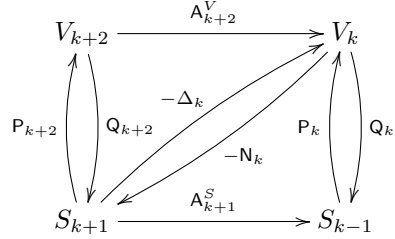
\begin{figure}[ht]
 $$
  \xymatrix @!0 @R=25mm @C=35mm {
 V_{k+2}     \ar[r]^{{\mathsf A}^{V}_{k+2}} 
 	 \ar@/^/[d]^{{\mathsf Q}_{k+2}}  
 %
 & V_{k}  
  \ar@/^/[d]^{{\mathsf Q}_k} 
   \ar@/^/[ld]^{-\mathsf N_k} \\
    S_{k+1}\ar@/^/[u]^-{{\mathsf P}_{k+2}}
    \ar[r]^-{{\mathsf A}^{S}_{k+1}} 
    \ar@/^/[ru]^{-\Delta_k}
    &S_{k-1}\ar@/^/[u]^{{\mathsf P}_k}
  }
  $$
  \caption[The spaces ${V}_{k}$, $S_k$ and associated operators and the Biot-Savart operator]{\label{diag_2}The top row contains the function spaces   $V_k$  for the vorticity fields while 
  the bottom row contains the spaces $S_k$ for the stream functions. The operators ${\mathsf A}^V_k$ and ${\mathsf A}^S_k$ are Stokes operators (see the Cauchy problems 
  \eqref{vorticity_cauchy} and \eqref{stream_cauchy} in the next section). The operators $\Delta_k$ 
  link the stream functions to the corresponding vorticity fields.}
  \end{figure}
%
%
\subsection{A simple example: The unit disk}
%
In this subsection, we assume that $\F$ is the unit disk and 
we aim at computing the spectrum of the operator $\mathcal A^V_2$ (i.e. the operator ${\mathsf A}^V_2$ seen as an unbounded operator of domain $V_2$ in $V_0$; 
see \eqref{def:unboundAk}). 
\par
All the harmonic functions   in $\F$ are equal to the real part of a 
holomorphic function in $\F$. The holomorphic functions can be expanded as power series with convergence radius equal to $1$. 
%
%
%
%
It follows that a function $\omega\in L^2(\F)$ belongs to $V_0=\mathfrak H^\perp$ if and only if, for every nonnegative integer $k$:
$$\Re\left(\int_{\F} \omega(z)\, z^k\,{\rm d}|z|\right)=0\qquad\text{and}\qquad
\Im\left(\int_{\F} \omega(z)\, z^k\,{\rm d}|z|\right)=0.$$
Using the method of separation of variables  in polar coordinates, 
we find first that a function $\omega (r,\theta)=\rho(r)\varTheta(\theta)$ is in $V_0$ when:
$$\left(\int_0^1\rho(r) r^{k+1}\,{\rm d}r\right)\left(\int_0^{2\pi}\varTheta(\theta) e^{ik\theta}\,{\rm d}\theta\right)=0\forallt k\in\mathbb N.$$
Then, providing  that $-\Delta \omega=\lambda\omega$ in $\F$ for some positive real number $\lambda$, we 
deduce the expression of the function $\omega$, namely: 
$$(r,\theta)\longmapsto\rho_k(r)\cos(k\theta)\qquad\text{or}\qquad
(r,\theta)\longmapsto\rho_k(r)\sin(k\theta)$$ 
for some nonnegative integer $k$. The function $\rho_k$ solves the differential equation in $(0,1)$:
\begin{subequations}
\begin{equation}
\label{diff_equ}
\rho_k''(r)+\frac{1}{r}\rho_k'(r)+\left(\lambda -\frac{k^2}{r^2}\right)\rho_k(r)=0\qquad r\in (0,1),
\end{equation}
 and satisfies:
\begin{equation}
\label{rho_in_V0}
\int_0^1\rho_k(r)r^{k+1}\,{\rm d}r=0.
\end{equation}
\end{subequations}
The solution  of \eqref{diff_equ} (regular at $r=0$) is $\rho_k(r)=J_k(\sqrt{\lambda}r)$ where $J_k$ is the Bessel  function  of the first kind. Multiplying 
the equation \eqref{diff_equ} by $r^{k+1}$ and integrating over the interval $(0,1)$, we show that   \eqref{rho_in_V0} is equivalent to:
$$\sqrt{\lambda_k}J'_k(\sqrt{\lambda})-kJ_k(\sqrt{\lambda})=0.$$
Using the identity $J_k'(r)=kJ_k(r)/r-J_{k+1}(r)$, the condition above can be rewritten as:
\begin{equation}
\label{eq:valeurs_propres_bessel_disque}
J_{k+1}(\sqrt{\lambda})=0.
\end{equation}
We denote by $\alpha^j_k$ (for every integers $j,k\geqslant 1$) the $j$-th zero of the Bessel function $J_k$ and we set 
$$\lambda_k^j=(\alpha^j_{k+1})^2\qquad \text{for all }k\geqslant 0\text{ and }j\geqslant 1.$$
\begin{prop}
The eigenvalues of ${\mathsf A}^V_2$ (and then also of $\mathsf A^V_k$, $\mathsf A^S_k$ and $\mathsf A^{\mathbf J}_k$ for 
every index $k$, since they all have the same spectrum)  are the real positive numbers $\lambda_k^j$ ($k\geqslant 0$, $j\geqslant 1$). 
The eigenspaces corresponding to $\lambda_0^j$ ($j\geqslant 1$) are of dimension 1, spanned by the eigenfunctions:
\begin{subequations}
\label{def:eigen}
\begin{equation}
(r,\theta)\longmapsto J_0\Big(\sqrt{\lambda_0^j} r\Big).
\end{equation}
The
eigenspaces of the other eigenvalues $\lambda_k^j$ (for $k\geqslant 1$) are of dimension 2, spanned by the eigenfunctions:
\begin{equation}
(r,\theta)\longmapsto J_{k}\Big(\sqrt{\lambda_k^j} r\Big)\cos(k\theta)
\qquad\text{and}\qquad
(r,\theta)\longmapsto J_{k}\Big(\sqrt{\lambda_k^j}  r\Big)\sin(k\theta).
\end{equation}
\end{subequations}
\end{prop}
\begin{proof}
By construction, the functions defined in \eqref{def:eigen} are indeed eigenfunctions of ${\mathsf A}^V_2$.
To prove that every eigenfunctions of this operator is of the form \eqref{def:eigen}, it suffices to follow the lines of the proof of \cite[\S 8.1.1d.]{Dautray:1990aa} for the Dirichlet operator in the unit disk.
\end{proof}
We recover the spectrum of the Stokes operator as computed for instance in \cite{Kelliher:2009aa}.
%
\section{Lifting operators of the boundary data}
\label{SEC:lift_oper}
\subsection{Lifting operators for the stream functions}
\label{SEC:Nonhomogeneous}
Considering \eqref{eq_stokes:bound}, the velocity field $u$ solution to the NS equations in primitive variables is assumed to satisfy Dirichlet boundary conditions on $\Sigma$, the trace of $u$ 
on $\Sigma$ being denoted by $b$. Classically, this constraint is dealt with by means of a lifting operator. We refer to \cite{Raymond:2007aa} and 
references therein for a quite 
comprehensive survey on this topic. In nonprimitive variables, as already mentioned earlier in \eqref{eq:helmhotlz}, \eqref{eq:neumann_phi} 
and \eqref{eq:edp_psi}, the Dirichlet conditions for $u$ translate into Neumann boundary conditions for both the potential and the stream function,
namely:
\begin{subequations}
\label{def:boundary_cond}
\begin{equation}
\label{neumann_phi_psi}
\frac{\partial\varphi}{\partial n}=b\cdot n\quad\text{ and }
\quad\frac{\partial\psi}{\partial n}=\frac{\partial\varphi}{\partial \tau}-b\cdot\tau\quad\text{ on }\Sigma.
\end{equation}
Around every inner boundaries $\Sigma^-_j$,  the circulation of the fluid is classically defined by:
\begin{equation}
\label{circucu}
\varGamma_j=\int_{\Sigma^-_j}b\cdot\tau\ds=-\int_{\Sigma^-_j}\frac{\partial\psi}{\partial n}\ds\qquad (j=1,\ldots,N).
\end{equation}
This being reminded, identities \eqref{neumann_phi_psi} and \eqref{circucu} suggest that instead of the field $b$, the prescribed data on the boundary shall rather be 
given at every moment under the form of  a triple $(g_n,g_\tau,\varGamma)$ where $g_n$ and $g_\tau$ are scalar functions defined on $\Sigma$
and $\varGamma=(\varGamma_1,\ldots,\varGamma_N)$ is a vector in $\mathbb R^N$ 
in such a way that:
\begin{equation}
\label{b_decomp}
b=g_n n +\bigg(g_\tau-\sum_{j=1}^N\varGamma_j\frac{\partial\xi_j}{\partial n}\bigg)\tau
\quad\text{ with }\int_\Sigma g_n \ds=0\quad\text{and}\quad
\int_{\Sigma^-_j} g_\tau\ds=0\qquad (j=1,\ldots,N).
\end{equation}
\end{subequations}
We recall that $n$ and $\tau$ stand respectively for the unit outer normal and unit tangent vectors to $\Sigma$. 
The definition  of suitable function spaces for $g_n$ and $g_\tau$ requires introducing the following indices 
used to make precise the regularity of the boundary $\Sigma$. Thus, for every integer $k$, we define:
\begin{subequations}
\label{def_indices}
\begin{alignat}{2}
I_1(k)&=\left|k-{\frac12}\right|-{\frac12},\qquad&J_1(k)&=\left|k-{\frac12}\right|+{\frac12}=\max\{I_1(k-1),I_1(k+1)\},\\
I_2(k)&=|k-1|+1,&
J_2(k)&=||k|-1|+2=\begin{cases} \max\{I_2(k-1),I_2(k+1)\}&\text{if }k\geqslant 0\\ I_2(k+1)&\text{if }k\leqslant -1,\end{cases}
\end{alignat}
\end{subequations}
and we can now state:
\begin{definition}
\label{def_def_bkgk}
Let $k$ be an integer. Assuming that $\Sigma$ is of class $\mathcal C^{I_1(k),1}$, it makes sense to define:
\begin{subequations}
\label{def_GSk}
\begin{align}
\label{def_GSn}
G_k^n&=\bigg\{g\in H^{k-{\frac12}}(\Sigma)\,:\,\int_\Sigma g\ds=0\bigg\}\text{ if }k\geqslant -1\quad\text{ and }\quad G_k^n=G_{-1}^n\text{ otherwise}\\
\text{and}\qquad G_k^\tau&=\bigg\{g\in H^{k-{\frac12}}(\Sigma)\,:\,\int_{\Sigma^-_j}g\ds=0,\quad j=1,\ldots,N\bigg\},
\end{align}
\end{subequations}
where the boundary integrals are understood according to the rule of notation \eqref{rem:brackets}.
\end{definition}
The only purpose of setting $G_k^n=G_{-1}^n$ when $k\leqslant -2$ in \eqref{def_GSn}  is to simplify the statement of the next results.
\par
The problem of lifting the normal component $g_n$ by the harmonic Kirchhoff potential function is  addressed in the lemma below, where, for every 
nonnegative integer $k$:
\begin{equation}
\label{def_HK}
\mathfrak H_K^k= 
\bigg\{\varphi\in H^k(\F)\,:\,\Delta \varphi=0\text{ in }\mathcal D'(\F),\,\int_\F \varphi\dx=0\text{ and }\int_\Sigma\frac{\partial\varphi}{\partial n}=0\bigg\}.
\end{equation}
%
\begin{lemma}
\label{lift_potential}
Assume that $\Sigma$ is of class ${\mathcal C}^{|k|,1}$ for some   integer $k\geqslant -1$ and that $g_n$ belongs to $G_k^n$.
Then the operator 
\begin{equation}
\mathsf L_k^n:g_n\in G_k^n\longmapsto  \varphi\in  \mathfrak H_K^{k+1}\qquad\text{where}\quad \frac{\partial\varphi}{\partial n}\Big|_\Sigma=g_n,
\end{equation}
is well defined and bounded. The operator
\begin{equation}
{\mathsf T}_k:g_n\in G_k^n\longmapsto \frac{\partial\varphi}{\partial\tau}\Big|_\Sigma\in G_k^\tau,
\end{equation}
is bounded as well. Moreover,  as for the definition of $G_n^k$, the definition of ${\mathsf T}_k$ is extended to integers $k\leqslant -2$ by setting  ${\mathsf T}_k={\mathsf T}_{-1}$.
\end{lemma} 
\begin{proof}
Let us only consider  the weakest case, i.e. $k=-1$. We introduce the Hilbert space $E$ and its scalar product whose corresponding norm is 
equivalent in $E$ to the usual norm of $H^2(\F)$:
$$E=\bigg\{\theta\in H^2(\F)\,:\,\frac{\partial\theta}{\partial n}\Big|_\Sigma=0,\,\int_\Sigma\theta|_\Sigma\ds=0\bigg\},
\qquad(\theta_1,\theta_2)_E=\int_\F\Delta\theta_1\Delta\theta_2\dx.$$
According to Riesz representation Theorem, for every $g_n\in G_{-1}^n$, there exists a unique $\theta_g\in E$ such that:
$$(\theta,\theta_g)_E=-\int_\Sigma g_n\theta|_\Sigma\ds\forallt \theta\in E.$$
One easily verifies that the function $\varphi=\Delta \theta_g$ is in $L^2(\F)$ and satisfies $\int_\F\varphi\dx=0$ and $({\partial\varphi}/{\partial n})|_\Sigma=g_n$ in $H^{-\frac32}(\Sigma)$. The rest of 
the lemma being either classical or obvious, the proof is complete.
\end{proof}
%
The operator ${\mathsf T}_k$ is the tangential differential operator composed with the classical Neumann-to-Dirichlet map.
Regarding now the second identity in \eqref{neumann_phi_psi}, we seek a lifting operator valued  
in the kernel of the operator $\mathsf Q\Delta$, that is the kernel of the Stokes operator for the stream function (see Lemma~\ref{lem:sapce_S2}). 
Loosely speaking (disregarding regularity issues), this kernel is $\mathfrak B_S$, the space of the biharmonic stream functions defined in \eqref{def_Scurl}.
%
%
\begin{definition}
\label{def_stuff}
Let $k$ be an integer and assume that $\Sigma$ is of class ${\mathcal C}^{I_2(k),1}$. 
The space of biharmonic functions $\mathfrak B_S^k$ and the lifting 
operator 
$\mathsf L_k^\tau:G_k^\tau\to \mathfrak B_S^k$
are defined differently, depending upon the sign of $k$:
\begin{my_enumerate}
\item When $k\geqslant 1$, $\mathfrak B_S^k=\mathfrak B_S\cap H^{k+1}(\F)$ 
(and hence $\mathfrak B_S^1$ is simply equal to   $\mathfrak B_S$ defined   in \eqref{def_Scurl}) and for any $g_\tau\in G^\tau_k$, 
 $\mathsf L_k^\tau g_\tau$  is the unique stream
 function $\psi$ in $\mathfrak B^k_S$ satisfying   the Neumann boundary condition:
 $$\frac{\partial\psi}{\partial n}\bigg|_{\Sigma} =g_\tau\qquad\text{on }\Sigma.$$ 
\item When $k\leqslant 0$,  for any $g_\tau\in G^\tau_k$, $\mathsf L_{k}^\tau g_\tau$ is the element of the dual space $S_k$ given by:
\begin{equation}
\label{eq:def_Lktau}
\langle \mathsf L_{k}^\tau g_\tau,\theta \rangle_{S_{-k},S_{k}}=\int_{\Sigma}({\mathsf P}_{-k+1}\theta) g_\tau\ds\forallt \theta\in S_{-k},
\end{equation}
and the space $\mathfrak B^{k}_S$ is defined as the image of $\mathsf L_{k}^\tau$ in $S_{k}$. 
\end{my_enumerate}
\end{definition}
\begin{rem}
\label{BS_and_Sk}
\begin{my_enumerate}
\item
For $k\leqslant 0$, the operator $\mathsf L_{k}^\tau$ is well defined according to Lemma~\ref{regul_P0} and Lemma~\ref{lem:sapce_S2}, under the regularity 
assumption on the boundary $\Sigma$ of Definition~\ref{def_stuff}.
\item
For every integer $k$, the space $G_k^\tau$ is actually well defined as soon as the boundary $\Sigma$ is of class ${\mathcal C}^{I_1(k),1}$ (see Definition~\ref{def_def_bkgk}). However, 
further regularity is needed to define the lifting operator, namely ${\mathcal C}^{I_2(k),1}$.
\end{my_enumerate}
%
\end{rem}
%
For every pair of   integers $(k',k)$, both positive or both nonpositive, the inequality $k'\geqslant k$ entails the inclusion 
$\mathfrak B_S^{k'}\subset \mathfrak B_S^{k}$.
We shall prove that the inclusion $\mathfrak B_S^1\subset \mathfrak B_S^0$ still holds and that the diagram on Fig.~\ref{diag_7} commutes. 
Notice that $\mathsf L_k^\tau$ is clearly invertible when $k$ is positive.  
The question of invertibility for nonpositive indices $k$, or more precisely of injectivity (since surjectivity is obvious) is not clear. 
This amounts to determine whether  the traces of the functions of
$V_{-k+1}$ are dense in $H^{-k+\frac{1}{2}}(\Sigma)$.
%
\begin{figure}[ht]
 $$
  \xymatrix @!0 @R=13mm @C=8mm {
 G_{k'}^\tau    \ar[d]^{\mathsf L_{k'}^\tau}   
 %
 & \subset
 &G_k^\tau  \ar[d]^{\mathsf L_{k}^\tau}   \\
    \mathfrak B_S^{k'}
       &
       \subset
    &\mathfrak B_S^{k}
  }
  $$
  \caption[The spaces $G_k^\tau$, $\mathfrak B_S^{k}$ and associated operators]{\label{diag_7} The diagram commutes for any pair of integers $(k,k')$ such that $k'\geqslant k$.}
  \end{figure}
%
\begin{lemma}
\label{reg_Lk}
The operator  $\mathsf L_k^\tau$ is bounded for every integer $k$ and is an isomorphism when $k$ is positive.  
For every pair of integers $(k',k)$ such that $k'\geqslant k$, the restriction of $\mathsf L_k^\tau$ to 
$G_{k'}^\tau$ is equal to $\mathsf L_{k'}^\tau$ (providing that $\Sigma$ is of class $\mathcal C^{\max\{I_2(k),I_2(k')\},1}$).
\end{lemma}
\begin{proof}The boundedness is a consequence of Lemma~\ref{regul_P0}, Lemma~\ref{lem:sapce_S2} and the continuity of the trace operator.
\par
Let $\Sigma$ be of class ${\mathcal C}^{2,1}$, $g_\tau$ belong to $G_1^\tau$ and introduce the stream function $\psi=\mathsf L_1^\tau g_\tau$. Considering $\psi\in\mathfrak B_S^1$ as an element of $S_0$ identified with its dual space, we get:
%
%
%
$$(\psi,\theta)_{S_0}=(\nabla\psi,\nabla {\mathsf P}_1\theta)_{\mathbf L^2(\F)}=\int_{\Sigma}({\mathsf P}_1 \theta) g_\tau\ds-(\Delta\psi,{\mathsf P}_1\theta)_{L^2(\F)}
\forallt \theta\in S_0,$$
where the last term vanishes because $\Delta\psi$ belongs to $\mathfrak H$.  This proves that $\mathsf L_1^\tau =\mathsf  L_0^\tau  $ in $G_1^\tau$. The other 
cases derive straightforwardly and the proof is complete.
\end{proof}
%
\par
We can  gather Lemma~\ref{lift_potential} and Definition~\ref{def_stuff} in order to define a lifting operator taking into 
account the circulation of the fluid around the fixed obstacles. 
In view of \eqref{neumann_phi_psi} and \eqref{circucu}, we are led to set:
\begin{definition}
\label{def:Lk}
Let $k$ be any integer and assume that $\Sigma$ is of class $\mathcal C^{I_2(k),1}$ and that the triple $(g_n,g_\tau,\varGamma)$ is in 
$G_k^n\times G_k^\tau\times \mathbb R^N$ with   $\varGamma=(\varGamma_1,\ldots,\varGamma_N)$. We define the operator:
\begin{equation}
\label{def_LkS}
\mathsf L_k^S(g_n,g_\tau,\varGamma)=\mathsf L_k^\tau({\mathsf T}_k g_n-g_\tau)+\sum_{j=1}^N\varGamma_j \xi_j,
\end{equation}
which is valued in the space 
\begin{equation}
\label{def_mathfrakS}
 S^{\rm b}_k=\mathfrak B_S^k\oplus \mathbb F_S.
\end{equation}
\end{definition}
We can address   the case of time dependent spaces:
\begin{definition}
\label{def:Lk_2}
Let $T$ be a positive real number, $k$ be an integer and assume that $\Sigma$ is of class $\mathcal C^{J_1(k),1}$ (the expression of $J_1(k)$ is given in \eqref{def_indices}). 
We begin by introducing the spaces:
\begin{align*}
G_k^n(T)&= L^2(0,T; G^n_{k+1})\cap\mathcal C([0,T];G_k^n) \cap H^1(0,T; G^n_{k-1})\\
G_k^\tau(T)&=L^2(0,T; G^\tau_{k+1}))\cap\mathcal C([0,T];G_k^\tau) \cap H^1(0,T; G^\tau_{k-1}),
\end{align*}
and also:
\begin{equation}
\label{def_Gk}
G_k(T)=\begin{cases}
G_k^n(T)\times G_k^\tau(T)\times H^1(0,T;\mathbb R^N)&\text{when }k\geqslant 0,\\
L^2(0,T;G_{k+1}^n)\times L^2(0,T;G_{k+1}^\tau)\times L^2(0,T;\mathbb R^N) &\text{when }k\leqslant -1.
\end{cases}
\end{equation}
Assuming that $\Sigma$ is of class $\mathcal C^{J_2(k),1}$ (with $J_2(k)$  defined in \eqref{def_indices}) the operator 
$\mathsf L_{k+1}^S$   maps to space $G_k(T)$ into the space:
\begin{equation}
\label{def_mathfrakST} 
 S^{\rm b}_{k}(T)=
\begin{cases}
H^1(0,T; S^{\rm b}_{k-1})\cap \mathcal C([0,T]; S^{\rm b}_k)\cap L^2(0,T; S^{\rm b}_{k+1})&  \text{ if }k\geqslant 0,\\
 L^2(0,T; S^{\rm b}_{k+1}) & \text{ if }k\leqslant -1.
\end{cases}
\end{equation}
\end{definition}
As a direct consequence of Lemmas~\ref{lift_potential} and \ref{reg_Lk}, we can state: 
\begin{lemma}
\label{LEM:lift_stream}
Let $k$ be any integer and assume that $\Sigma$ is of class $\mathcal C^{I_2(k),1}$. Then the lifting operator for the stream function:
$$\mathsf L_k^S:G_k^n\times G_k^\tau\times \mathbb R^N\longrightarrow S^{\rm b}_k,$$
is well defined and is bounded. Moreover, if $k$ and $k'$ are two integers such that $k'\leqslant k$, then $\mathsf L_{k'}^S=\mathsf L_{k}^S$ in 
$G_k^n\times G_k^\tau\times \mathbb R^N$. It follows that for every positive  real number $T$ and every integer $k$, providing that $\Sigma$ 
is of class $C^{J_2(k),1}$, the operator:
$$\mathsf L_{k+1}^S:G_k(T) \longrightarrow  S^{\rm b}_{k}(T),$$
is well defined and bounded as well, the bound being uniform with respect to $T$.
\end{lemma}
%
%
%
\subsection{Additional function spaces}
\label{SEC:Nonhomogeneous_vorticity}
We aim now at building a lifting operator valued in vorticity spaces (i.e. we aim at giving the counterpart of Definitions~\ref{def:Lk}-\ref{def:Lk_2}  and 
Lemma~\ref{LEM:lift_stream} for the vorticity). 
We recall that, for every positive integer $k$, the lifting operator $\mathsf L^S_k$   is valued in $ S^{\rm b}_k$. For nonpositive integers $k$, $S^{\rm b}_k$ 
is a subspace of $S_k$ and therefore, the corresponding vorticity space is simply $V_{k-1}^{\rm b}=\Delta_{k-1}S^{\rm b}_k$. However, when 
$k$ is positive, $S_k^{\rm b}\cap S_k=\{0\}$. A somehow naive approach would consist in taking simply the Laplacian of $S_k^{\rm b}$ but one easily 
verifies that $\Delta S_k^{\rm b}\subset \mathfrak H$ and $\mathfrak H$ is in no space $V_j$ for any integer $j$. This difficulty is circumvented by 
noticing that $S_k^{\rm b}\subset S_0$ (still considering positive integers $k$). So $V^{\rm b}_k=\Delta_{-1}S^{\rm b}_k$ (with $\Delta_{-1}$ defined 
in \eqref{alternQk_2}) seems to be a good candidate 
for our purpose, an idea we are now going to elaborate on. More precisely, for every integer $k$, $S_k^{\rm b}$ is a subspace of $\bar S_k$ defined by:
\begin{equation}
\label{def_mathcal_Sk}
\bar S_k=S_0\cap H^{k+1}(\F)\quad\text{ if }k\geqslant 1\qquad
\text{and}\qquad \bar S_k=S_k\quad\text{ if }k\leqslant 0.
\end{equation}
The corresponding vorticity space is therefore in the image of $\bar S_k$ (seen as a subspace of $S_0$) by $\Delta_{-1}$  if $k\geqslant 1$ 
and by $\Delta_{k-1}$ if $k\leqslant 0$ (see Fig.~\ref{diag_2}). Thus we define:
\begin{equation}
\label{def_mathcal_Vk}
\bar V_k=\Delta_{-1}\bar S_{k+1}\quad\text{ if }k\geqslant 0\qquad
\text{and}\qquad \bar V_k=\Delta_k \bar S_{k+1}=V_k\quad\text{ if }k\leqslant -1.
\end{equation}
It is crucial to understand that, no matter how regular   the functions are, the spaces $\bar V_k$ are always dual spaces (for every integer $k$). They are subspaces 
of $V_{-1}$. We will show that $V_{-1}$ is the space of largest index that contains in some sense the harmonic functions. 
%
%
%
We shall focus our analysis on the pairs $(\bar S_1,\bar V_0)$, $(\bar S_2,\bar V_1)$ and $(\bar S_3,\bar V_2)$ only, the other cases being of less importance as it will 
becomes clear in the next section. 
\subsubsection*{The pair $(\bar S_1,\bar V_0)$}
The space $\bar S_1$ is provided with the scalar product:
\begin{equation}
\label{eq:scalar_product_N}
(\bar\psi_1,\bar\psi_2)_{\bar S_1}=(\Delta \bar\psi_1,\Delta\bar \psi_2)_{L^2(\F)}+\Gamma(\bar\psi_1)\cdot \Gamma(\bar\psi_2),\forallt \bar\psi_1,\bar\psi_2\in \bar S_1,
\end{equation}
where, for every $\theta\in H^2(\F)$:
$$\Gamma(\theta)=\big(\Gamma_1(\theta),\ldots,\Gamma_N(\theta)\big)^t\in\mathbb R^N\qquad\text{with}\qquad 
\Gamma_j(\theta)=-\int_{\Sigma_j^-}\frac{\partial \theta}{\partial n}\ds,\quad(j=1,\ldots,N).$$
%
%
\begin{lemma}
\label{decomp_S1}
The space $\bar S_1$ enjoys the following properties:
\begin{my_enumerate}
\item 
The norm induced by the scalar product \eqref{eq:scalar_product_N}
is equivalent in $\bar S_1$ to the usual norm of $H^2(\F)$. 
\item
The space $\bar S_1$ admits the following orthogonal decompositions:
\begin{equation}
\label{decomp_S02}
\bar S_1= Z_2\sumperp  \mathbb F_S=S_1\sumperp  \mathfrak B_S\sumperp \mathbb F_S,
\end{equation}
where we recall that the expression of the space $\mathfrak B_S$ is given in \eqref{def_Scurl} and that the finite dimensional space $\mathbb F_S$  is spanned by the functions $\xi_j$ ($j=1,\ldots,N$). 
\end{my_enumerate}
\end{lemma}
%
The orthogonal decomposition \eqref{decomp_S02} can be given a physical meaning: The subspace $S_1$ contains the stream functions 
with homogeneous boundary conditions while the subspace $\mathfrak B_S$ contains the stream functions 
that solve  stationary Stokes problems (with zero circulation though). Finally, the space $\mathbb F_S$ contains the harmonic stream functions accounting for the circulation of 
the fluid around the inner boundaries $\Sigma^-_j$ ($j=1,\ldots,N$).
\begin{proof}[Proof of Lemma~\ref{decomp_S1}]The equivalence of the norms derives from classical elliptic regularity results. 
On the other hand, according to \eqref{decomp_S2b}:
\begin{equation}
\label{eq:decomp_S2_2}
 Z_2=S_1\sumperp  \mathfrak B_S\subset \bar S_1.
\end{equation}
Applying the  fundamental homomorphism theorem to the surjective operator $(-\Delta):\bar S_1\longrightarrow  Z_0$ whose kernel is the space $\mathbb F_S$,
we next obtain that:
\begin{equation}
\label{eq:decomp_S2}
\bar S_1= Z_2\oplus \mathbb F_S.
\end{equation}
%
%
%
Finally, combining   \eqref{eq:decomp_S2_2} and \eqref{eq:decomp_S2} yields \eqref{decomp_S02} after verifying that the direct sum 
is orthogonal for the scalar product \eqref{eq:scalar_product_N}. The proof is then completed.
\end{proof}
\par
%
\par
Let us determine now the corresponding decomposition for the vorticity space $\bar V_0=\Delta_{-1}\bar S_1$ 
which is a subspace of the 
dual space $V_{-1}$ (see Fig.~\ref{diag_2}). 
We shall prove in particular that $\bar V_0$ contains $V_0$ whose expression is (seen as a subspace of $V_{-1}$):
\begin{equation}
\label{def_hatV0}
V_0=\big\{(\omega,{\mathsf Q}_1\cdot)_{L^2(\F)}\,:\,\omega\in \mathfrak H^\perp\big\}.
\end{equation}
Notice   that in \eqref{def_hatV0}, one would expect 
merely the term $(\omega, \cdot)_{L^2(\F)}$ in place of  $(\omega,{\mathsf Q}_1\cdot)_{L^2(\F)}$, but  both linear forms are equal in $V_1$.
We define below two additional  subspaces of $V_{-1}$:
\begin{equation}
\label{def_hatV1}
\mathfrak H_V=\big\{(\omega,{\mathsf Q}_1\cdot)_{L^2(\F)}\,:\,\omega\in \mathfrak H\big\}\qquad
\text{and}\qquad L^2_V=\big\{(\omega,{\mathsf Q}_1\cdot)_{L^2(\F)}\,:\,\omega\in L^2(\F)\big\}.
\end{equation}
The space $\mathfrak H_V$ contains in some sense the harmonic vorticity field. Finally, we introduce the finite dimensional subspace 
of $V_{-1}$: 
\begin{equation}
\label{def_FVast}
\mathbb F_V^\ast={\rm span\,}\{\zeta_j,\,j=1,\ldots,N\} ,
\end{equation}
where, for every $j=1,\ldots,N$:
\begin{equation}
\label{vortic_circu}
\langle \zeta_j,\omega\rangle_{V_{-1},V_1}=-(\nabla\xi_j,\nabla  \mathsf Q_1\omega)_{\mathbf L^2(\F)}=-\int_{\Sigma^-_j}\frac{\partial\xi_j}{\partial n} \mathsf Q_1\omega \ds\forallt\omega\in V_1.
\end{equation}
%
We can now state:
\begin{theorem}
\label{THEO:decompV0}
The space $\bar V_0$ can be decomposed as follows:
\begin{equation}
\label{decomp_V0}
\bar V_0= L^2_V\sumperp \mathbb F_V^\ast=V_0\sumperp V_0^{\rm b}\qquad\text{ where }\quad V_0^{\rm b}=\mathfrak H_V\sumperp \mathbb F_V^\ast.
\end{equation}
The direct sum above is orthogonal for the scalar product defined, for every $\bar\omega_1$ and $\bar\omega_2$ in $\bar V_0$ by: 
\begin{equation}
\label{norm_V10}
(\bar\omega_1,\bar\omega_2)_{\bar V_0}=(\omega_1,\omega_2)_{L^2(\F)}+\sum_{j=1}^N\alpha_{1,j}\alpha_{2,j},
\end{equation}
where, for $k=1,2$, $\bar\omega_k=(\omega_k,{\mathsf Q}_1\cdot)_{L^2(\F)}+\zeta^k$ 
  with $\omega_k\in L^2(\F)$ and $\zeta^k=\sum_{j=1}^N\alpha_{k,j}\zeta_j$
 in $\mathbb F_V^\ast$ ($\alpha_{k,j}\in\mathbb R$ for $j=1,\ldots,N$).
\par
Moreover, the restriction of $\Delta_{-1}$ to $\bar S_1$, denoted by $\bar\Delta_0$,  is an isometry  from $\bar S_1$ onto $\bar V_0$ (see Fig.~\ref{figS03}). 
\end{theorem}
\begin{rem}
\label{rem_singular_V0}
Let us emphasize that:
\begin{my_enumerate}
\item
 In the decomposition 
$\bar\omega=(\omega,{\mathsf Q}_1\cdot)_{L^2(\F)}+\zeta$ of every $\bar\omega$ of $\bar V_0$, the term $\omega$ which belongs to $L^2(\F)$ 
will be referred to as the {\it regular part} of $\bar\omega$ while $\zeta$ will stand for the {\it singular part}.
\item
Loosely speaking, the space $\bar V_0$ consists in  functions in $L^2(\F)$ and measures $\zeta_j$ ($j=1,\ldots,N$) supported on the boundaries $\Sigma^-_j$ (notice again that 
$\mathbb F_V^\ast$ is not a distributions space). This can be somehow understood from a physical point of view by observing that
 $-\zeta_j$ is the vorticity corresponding to the harmonic stream function $\xi_j$ which
accounts for the circulation of the fluid around $\Sigma^-_j$. Hence the vorticity is a measure supported on the boundary of the obstacle.
In connection with this topic, wondering 
how is vorticity imparted to the fluid when a stream flow past an obstacle, 
Lighthill answers in \cite{Rosenhead:1988aa}  that {\it the solid boundary is a distributed source of vorticity (just as, in some flows, it may be a distributed source of heat).}

\item
The fact that $\bar\Delta_0$ is an isometry asserts in particular that to any given vorticity in $\bar V_0$ corresponds a unique 
stream function $\bar\psi$ in $\bar S_1$ that can be uniquely decomposed as $\psi+\psi_S+\psi_C$ where $\nabla^\perp\psi=0$ on the boundary $\Sigma$, 
$\nabla^\perp\psi_S$ solves a stationary Stokes system and $\psi_C$ is harmonic in $\F$ and accounts for the circulation of the fluid 
around the boundaries $\Sigma^-_j$ 
($j=1,\ldots,N$). 
\end{my_enumerate}
\end{rem}
The rest of this subsection is dedicated to the proof of Theorem~\ref{THEO:decompV0}. In order to determine the image 
of $\bar S_1$ by the operator  $\Delta_{-1}$, 
the factorization $\Delta_{-1}=-{\mathsf A}^V_1{\mathsf P}_1$ suggests to determine first the expression of the space ${\mathsf P}_1\bar S_1$. 
This space   is provided with the scalar product:
$$(\omega_1,\omega_2)_{{{\mathsf P}_1\bar S_1}}=(\Delta \omega_1,\Delta \omega_2)_{L^2(\F)}+\Gamma(\omega_1)\cdot \Gamma(\omega_2),\forallt 
\omega_1,\omega_2\in {{\mathsf P}_1\bar S_1},$$
and we denote by $\bar{\mathsf P}_2$ the restriction of ${\mathsf P}_1$ to ${\bar S_1}$.
 \begin{rem}
 According to Lemma~\ref{regul_P0}, when $\Sigma$ is of class ${\mathcal C}^{3,1}$, the space ${{\mathsf P}_1\bar S_1}$ is simply equal 
 to $V_1\cap H^2(\F)$. 
 When $\Sigma$ is less regular Remark~\ref{regul_V2} applies replacing $V_2$ with 
${{\mathsf P}_1\bar S_1}$.
 \end{rem}
The decomposition 
\eqref{decomp_S02} leads to introducing the spaces:
\begin{equation}
\label{def_V2flat}
\mathfrak B_V={\mathsf P}_1\mathfrak B_S
 \qquad\text{and}\qquad \mathbb F_V ={\mathsf P}_1\mathbb F_S.
 \end{equation}
 %
%
Nothing more than
$\mathfrak B_V=\mathfrak B\cap V_0$ (where $\mathfrak B$ is defined in \eqref{def_BV})
can be said on the space $\mathfrak B_V$. The space $\mathbb F_V$ however can be bound to the space ${\mathbb B}_S$ spanned by the 
functions $\chi_j$ ($j=1,\ldots,N$) defined in \eqref{def_chi} (see the definition 
below the identity \eqref{split_S1}). 
\begin{lemma}
\begin{my_enumerate}
\item
The space ${\mathbb B}_S$ is a subspace of $S_2$ and  $\mathbb F_V=\Delta_1 {\mathbb B}_S$.
\item
The space ${{\mathsf P}_1\bar S_1}$ admits the following orthogonal decomposition:
\begin{equation}
\label{exp:V21}
{{\mathsf P}_1\bar S_1}=V_2\sumperp  \mathfrak B_V\sumperp  \mathbb F_V.
\end{equation}
Moreover, the operator $\bar{\mathsf P}_2$ is an isometry from $\bar S_1$ onto ${{\mathsf P}_1\bar S_1}$ (see Fig.~\ref{figS03}).
\end{my_enumerate}
\end{lemma}
\begin{proof}
For every $\theta\in {\mathcal D}$, an integration by parts yields:
\begin{subequations}
\label{eq:multi_opp}
\begin{equation}
(\chi_j,\theta)_{S_1}=-(\mathsf Q_1\varOmega_j, \theta)_{S_0}\forallt j=1,\ldots,N,
\end{equation}
where $\varOmega_j=\Delta \chi_j$, because $\chi_j$ is of class ${\mathcal C}^\infty$ in the support of $\theta$. On the other hand, for every pair 
of indices $j,k\in\{1,\ldots,N\}$, we have also:
\begin{equation}
(\chi_j,\chi_k)_{S_1}=\int_{\Sigma}\mathsf Q_1\varOmega_j\frac{\partial\chi_k}{\partial n}\ds-(\mathsf Q_1\varOmega_j, \chi_k)_{S_0}=-(\mathsf Q_1\varOmega_j, \chi_k)_{S_0},
\end{equation}
\end{subequations}
where we have used   the rule of notation \eqref{rem:brackets}
(as being harmonic in $L^2(\F)$, the trace of $\varOmega_j$ on $\Sigma$ is well defined in $H^{-{\frac12}}(\Sigma)$).
Since the space ${\mathcal D}(\F)\oplus {\mathbb B}_S$ is dense in $S_1$ according to the decomposition \eqref{split_S1}, we deduce from the 
identities \eqref{eq:multi_opp} that ${\mathbb B}_S\subset S_2$ (recall the $S_2$ is the preimage of $S_0$ by $\mathsf A^S_1$).
\par
Notice now that the functions ${\mathsf P}_1\xi_j$ ($j=1,\ldots,N$) belong to $V_1$, 
are harmonic in $\F$ and according 
to the second point of Remark~\ref{rem_fluxA} they satisfy the same fluxes conditions \eqref{flux_cond_xi} as the functions $\xi_j$. All these properties are also shared by the functions 
$\varOmega_j$  whence we   deduce first that:
\begin{equation}
\label{def_omegaj}
{\mathsf P}_1\xi_j=\varOmega_j\qquad (j=1,\ldots,N),
\end{equation}
and then that $\mathbb F_V=\Delta_1 {\mathbb B}_S$, which is the first point of the lemma. The second point can easily be deduced   
from \eqref{decomp_S02} and the second occurence of Remark~\ref{rem_fluxA}.
\end{proof}
%
Yet it remains  to apply the operator ${\mathsf A}^V_1$ to the equality \eqref{exp:V21} in order to get the expression of $\bar V_0=\Delta_{-1} \bar S_1$. 
Since ${\mathsf A}_1^V$ is an isometry, the decomposition \eqref{decomp_V0} is a direct consequence of the decomposition \eqref{exp:V21} and 
the following lemma, where the spaces $\mathfrak H_V$ and $\mathbb F_V^\ast$ are defined respectively in \eqref{def_hatV1} and \eqref{def_FVast}.
%
\begin{lemma}
The following equalities hold:
\begin{equation}
\label{both_identi}
{\mathsf A}^V_1\mathfrak B_V=\mathfrak H_V\qquad\text{and}\qquad
{\mathsf A}^V_1\mathbb F_V=\mathbb F_V^\ast.
\end{equation}
\end{lemma}
%
\begin{proof}
By definition, every element of $\mathfrak B_V$ can be written ${\mathsf P}_1\psi$ for some $\psi\in\mathfrak B_S$. The definitions of the operator 
${\mathsf A}^V_1$  and of the scalar product in $V_1$ lead to:
$$\langle {\mathsf A}^V_1{\mathsf P}_1\psi,\theta\rangle_{V_{-1},V_1}=({\mathsf P}_1\psi,\theta)_{V_1}=(\psi,{\mathsf Q}_1\theta)_{ Z_1},\forallt \theta\in V_1.$$
But $\mathfrak B_S$ is a subspace of $ Z_2$ according to \eqref{decomp_S2b} and ${\mathsf Q}_1V_1= Z_1$. It follows that:
$$(\psi,{\mathsf Q}_1\theta)_{ Z_1}=({\mathsf A}^{ Z}_2\psi,{\mathsf Q}_1\theta)_{ Z_0}=(\omega,{\mathsf Q}_1\theta)_{L^2(\F)},\forallt \theta\in V_1,$$
where $\omega={\mathsf A}^{ Z}_2\psi$ belongs to $\mathfrak H$ according to Lemma~\ref{decomp_S2}. The first equality in \eqref{both_identi} 
being proven, let us address the latter. For every $j=1,\ldots,N$, some elementary algebra yields:
$$\langle {\mathsf A}_1^V{\mathsf P}_1\xi_j,\omega\rangle_{V_{-1},V_1}=({\mathsf P}_1\xi_j,\omega)_{V_1}=(\nabla\xi_j,\nabla {\mathsf Q}_1\omega)_{\mathbf L^2(\F)}=
(\xi_j,{\mathsf Q}_1\omega)_{S_0},\forallt \omega\in V_1.$$
Comparing with \eqref{vortic_circu}, we obtain indeed that ${\mathsf A}_1^V{\mathsf P}_1\xi_j=-\zeta_j$ and recalling the definition of $\mathbb F_V$ given in \eqref{def_V2flat}, 
we are done with both identities in \eqref{both_identi} and the proof is completed.
\end{proof}
%
As we did for $\bar S_1$ and ${{\mathsf P}_1\bar S_1}$, the space $\bar V_0$ can be provided with a norm stronger than the one of the ambiant space $V_{-1}$, 
namely the norm which derives from the scalar product \eqref{norm_V10}.  One easily verifies that the direct sum \eqref{decomp_V0} 
is indeed orthogonal for this scalar product.  Furthermore, the operator $\bar{\mathsf A}_2^V:{{\mathsf P}_1\bar S_1}\to \bar V_0$ which is 
the restriction of $\mathsf A_1^V$ to $\mathsf P_1\bar S_1$ is an isometry. Since $\bar\Delta_0=-\bar{\mathsf A}_2^V\bar{\mathsf P}_2$ and the operators $\bar{\mathsf A}_2^V$ 
and $\bar{\mathsf P}_2$ are both isometries, we can draw the same conclusion for $\bar\Delta_0$. The proof of the theorem is now completed.
$\hfill\square$ 
%
%
\begin{figure}[ht]
\centerline{\xymatrix  @R=8mm @C=10mm {
    V_2\sumperp \mathfrak B_V\sumperp\mathbb  F_V \ar[r]^-{\bar{\mathsf A}_2^V} 
      &\mbox{$\bar V_0=\overunderbraces{&&&\br{3}{ V^{\rm b}_0}}%
{&V_0&\sumperp& \mathfrak H_V&\sumperp& \mathbb F_V^\ast}%
{&\br{3}{ L^2_V}}$}    \\
   \mbox{$\bar S_1=\overunderbraces{&&&\br{3}{ S^{\rm b}_1}}%
{&S_{1}&\sumperp& \mathfrak B_S^1&\sumperp& \mathbb F_S}%
{&\br{3}{ Z_2}}$}\ar[u]^-{\bar{\mathsf P}_2}  \ar[ur]^{-\bar\Delta_0 }
    }}
     \caption[The spaces $\bar{V}_0$, $\bar{S}_1$ their expression and associated operators]{\label{figS03}Some function  spaces and   isometric operators appearing in the statement of 
     Theorem~\ref{THEO:decompV0} and its proof. As usual, the top row contains the vorticity spaces while 
     the bottom row contains the spaces for the stream functions.}
    \end{figure}
%
%
\subsubsection*{The pair $(\bar S_2,\bar V_1)$}
We assume that $\Sigma$ is of class ${\mathcal C}^{2,1}$ and we consider the spaces:
\begin{equation}
\label{def_S03}
\bar S_2=S_0\cap H^3(\F)\qquad\text{and}\qquad
\bar V_1=\Delta_{-1}\bar S_2.
\end{equation}
The analysis of these spaces being very similar to those of $\bar S_1$ and $\bar V_0$, 
we shall skip the  details and focus on the main results. 
%
%
\begin{lemma}
\label{LEM:barS2}
The spaces $\bar S_2$ and ${{\mathsf P}_1\bar S_2}$ admit respectively the following orthogonal decompositions:
\begin{equation}
\label{decomp_barS2}
\bar S_2=S_{2}\sumperp  \mathfrak B_S^2\sumperp  \mathbb  F_S\qquad
\text{and}\qquad
{{\mathsf P}_1\bar S_2}=V_3\sumperp\mathfrak B_V^3\sumperp \mathbb F_V,
\end{equation}
where $\mathfrak B_S^2=\mathfrak B_S\cap H^3(\F)$ was introduced in Definition~\ref{def_stuff} and $\mathfrak B_V^3={\mathsf P}_1\mathfrak B_S^2$.
The spaces $\bar S_2$ and ${\mathsf P}_1\bar S_2$  are provided with the same scalar product, namely:
\begin{align*}
(\psi_1,\psi_2)_{\bar S_2}&=(\Delta \psi_1,\Delta \psi_2)_{H^1}^V+\Gamma(\psi_1)\cdot\Gamma(\psi_2)
\forallt \psi_1,\psi_2\in \bar S_2,\\
(\omega_1,\omega_2)_{{{\mathsf P}_1\bar S_2}}&=(\Delta \omega_1,\Delta \omega_2)_{H^1}^V+\Gamma(\omega_1)\cdot\Gamma(\omega_2)
\forallt \omega_1,\omega_2\in {{\mathsf P}_1\bar S_2},
\end{align*}
the scalar product $(\cdot,\cdot)_{H^1}^V$ being defined in \eqref{def_H1V1}.
\par
Finally, the operator $\bar{\mathsf P}_3$ defined as 
the restriction of ${\mathsf P}_{1}$ to $\bar S_2$ is an isometry from $\bar S_2$ onto ${{\mathsf P}_1\bar S_2}$ (see Fig.~\ref{figV13}). 
\end{lemma}
%
We turn now our attention to $\bar V_1=\Delta_{-1}\bar S_2$, which is a subspace of the dual space $V_{-1}$.
As a subspace of $V_{-1}$ the spaces $V_1$ is identified with $\{(\omega,{\mathsf Q}_1\cdot)_{L^2(\F)}\,:\,\omega\in V_1\}$ and we define as well:
$$
\mathfrak H^1_V= \big\{(\omega,{\mathsf Q}_1\cdot)_{L^2(\F)}\,:\,\omega\in \mathfrak H^1\big\},
$$
where we recall that $\mathfrak H^1=\mathfrak H\cap H^1(\F)$ (defined  in Subsection~\ref{SEC:main_spaces}).
Finally, in the same way, we introduce:
$$ H^1_V=\{(\omega,{\mathsf Q}_1\cdot)_{L^2(\F)}\,:\,\omega\in H^1(\F)\},$$
that can be compared with the space $ L^2_V$ defined in \eqref{def_hatV1}.
%
\begin{theorem}
\label{THEO:barV1}
The space $\bar V_1$ is a subspace of $V_{-1}$ which can be decomposed as follows:
\begin{equation}
\label{space_V_11}
\bar V_1=H^1_V\sumperp \mathbb F_V^\ast=V_{1}\sumperp V_1^{\rm b}\qquad\text{with }\quad V_1^{\rm b}=\mathfrak H_V^{1}\sumperp \mathbb F_V^\ast.
\end{equation}
It is provided with the scalar product, defined for every $\bar\omega_1,\bar\omega_2\in\bar V_1$ by:
$$(\bar\omega_1,\bar\omega_2)_{\bar V_1}=(\omega_1,\omega_2)_{H^1}^V+\sum_{j=1}^N\alpha_{1,j}\alpha_{2,j},$$
where, for $k=1,2$, $\bar\omega_k=(\omega_k,{\mathsf Q}_1\cdot)_{L^2(\F)}+\zeta^k$ 
 with $\omega_k\in H^1(\F)$ and $\zeta^k=\sum_{j=1}^N\alpha_{k,j}\zeta_j$  in $\mathbb F_V^\ast$ ($\alpha_{k,j} \in\mathbb R$ for $j=1,\ldots,N$).
\par
Finally, the operator $\bar\Delta_1$ which is the restriction of $\Delta_{-1}$ to $\bar S_2$ is an isometry  from $\bar S_2$ 
onto $\bar V_1$ (see Fig.~\ref{figV13}). 
\end{theorem}
%
\begin{rem}
\label{rem_singular_V1}
As in Remark~\ref{rem_singular_V0}, 
in the decomposition 
$\bar\omega=(\omega,{\mathsf Q}_1\cdot)_{L^2(\F)}+\zeta$ of every vorticity field $\bar\omega$ in $\bar V_1$, the term $\omega$  (belonging  to $H^1(\F)$)
will be called the {\it regular part} of $\bar\omega$ and $\zeta$, the {\it singular part}.
\end{rem}
%
These results can be summarized in the commutative diagram on Fig.~\ref{figV13} where the operator 
$\bar{\mathsf A}_3^V$ defined as the restriction of ${\mathsf A}^V_1$ 
to the space ${{\mathsf P}_1\bar S_2}$ is an isometry from ${{\mathsf P}_1\bar S_2}$ onto $\bar V_1$. 
%
\begin{figure}[ht]
\centerline{\xymatrix  @R=8mm @C=10mm {
       V_3\sumperp\mathfrak B_V^{3} \sumperp \mathbb F_V \ar[r]^-{\bar{\mathsf A}_3^V}  &
       \mbox{$\bar V_1=\overunderbraces{&&&\br{3}{ V^{\rm b}_1}}%
{&V_1&\sumperp& \mathfrak H_V^1&\sumperp& \mathbb F_V^\ast}%
{&\br{3}{ H^1_V}}$}  \\
   \mbox{$\bar S_2=\overunderbraces{&&&\br{3}{ S^{\rm b}_2}}%
{&S_{2}&\sumperp& \mathfrak B_S^2&\sumperp& \mathbb F_S}{}$} \ar[u]^-{\bar{\mathsf P}_3}  \ar[ur]^{-\bar\Delta_1 } 
    }}
     \caption[The spaces $\bar{V}_1$, $\bar{S}_2$ their expression and associated operators]{\label{figV13}Some function spaces and isometric operators appearing in the statement of Lemma~\ref{LEM:barS2} 
     and Theorem~\ref{THEO:barV1}. This diagram is worth being compared with the diagrams on Fig.~\ref{diag_2} and 
     Fig.~\ref{figS03}. 
     In particular, the following inclusions hold: $\bar S_2\subset \bar S_1\subset S_0$ and $\bar V_1\subset \bar V_0\subset V_{-1}$.}
    \end{figure}
%
%
%
\subsubsection*{The pair $(\bar S_3,\bar V_2)$}
The decompositions of $\bar V_0$ and $\bar V_1$ rested mainly on the simple equalities $L^2(\F)=V_0\oplus \mathfrak H$ and $H^1(\F)=V_1\oplus \mathfrak H^1$. However, 
$H^2(\F)$ is not equal to $V_2\oplus\mathfrak H^2$. Indeed, according to Fig.~\ref{figS03}, the correct decomposition is more complex, namely:
$$H^2(\F)=V_2\oplus\mathfrak B_V\oplus\mathbb F_V\oplus \mathfrak H^2.$$
We are led to define:
$$H^2_V=\{(\omega,\mathsf Q_1\cdot)_{L^2(\F)}\,:\,\omega\in H^2(\F)\}\qquad\text{and}\qquad \mathfrak H^2_V=
\{(\omega,\mathsf Q_1\cdot)_{L^2(\F)}\,:\,\omega\in \mathfrak H^2\}.$$
Since $V_2$, $\mathfrak B_V$ and $\mathbb F_V$ are subspaces of the pivot space $V_0$, they are identified  to subspaces of $V_{-1}$ and the following 
decompositions hold:
\begin{equation}
\label{decomp_barV2}
\bar V_2=H^2_V\oplus\mathbb F_V^\ast=V_2\oplus\mathfrak B_V\oplus\mathbb F_V\oplus V_2^{\rm b}
\quad\text{with }\quad  H^2_V=V_2\oplus\mathfrak B_V\oplus\mathbb F_V\oplus \mathfrak H^2_V\quad\text{and}\quad
V_2^{\rm b}=\mathfrak H^2_V\oplus\mathbb F_V^\ast. 
\end{equation}
This direct sum is orthogonal once $\bar V_2$ is provided with the scalar product:
$$(\bar \omega_1,\bar\omega_2)_{\bar V_2}=(\Delta \omega_1,\Delta\omega_2)_{L^2(\F)}+(\mathsf P^\perp\omega_1,\mathsf P^\perp\omega_2)_{\mathfrak H^2}+\Gamma(\omega_1)\cdot\Gamma(\omega_2)
+\sum_{j=1}^N\alpha_{1,j}\alpha_{2,j},$$
for every $\bar\omega_k\in\bar V_2$ such that $\bar\omega_k=(\omega_k,\mathsf Q_1\cdot)_{L^2(\F)}+\sum_{j=1}^N\alpha_{k,j}\zeta_j$ 
with $\omega_k\in H^2(\F)$ and $\alpha_{k,j}\in\mathbb R$ for $k=1,2$. The decompositions \eqref{decomp_barV2} will play an important role in Section~\ref{SEC:more_regular} and in particular the fact that $\bar V_2$ is not equal to $V_2\oplus V_2^{\rm b}$.

We do not need to enter into the details of the decomposition of $\bar S_3$.  Let us just make precise the norm this  space is equipped with, namely:
$$(\psi_1,\psi_2)_{\bar S_3}=(\Delta^2\psi_1,\Delta^2\psi_2)_{L^2(\F)}+(\mathsf P^\perp\Delta\psi_1,\mathsf P^\perp\Delta\psi_2)_{\mathfrak H^2}+
\Gamma(\Delta\psi_1)\cdot\Gamma(\Delta\psi_2)+\Gamma(\psi_1)\cdot\Gamma(\psi_2),$$
for every $\psi_1,\psi_2$ in $\bar S_3$. As usual, we denote by $\bar{\Delta}_2$ the restriction of $\Delta_{-1}$ to $\bar S_3$ and we let is to the reader 
to verify that:
\begin{lemma}
\label{prop_bardelta2}
The operator $\bar{\Delta}_2$ is an isometry from $\bar S_3$ onto $\bar V_2$.
\end{lemma}
Notice that obviously $\bar S_3$ contains the space $S^{\rm b}_3$.
\subsection{Lifting operators for the vorticity field}
The expressions of the lifting operators  for the vorticity derive    straightforwardly from Fig.~\ref{figS03} and Fig.~\ref{figV13}. Following the lines 
of Definition~\ref{def:Lk} and recalling that the indices $I_2(k)$ and $J_2(k)$ are defined 
in \eqref{def_indices}, we can write:
\begin{definition}
\label{def:LkV}
Let $k$ be an integer such that $k\leqslant 2$ and assume that $\Sigma$ is of class $\mathcal C^{I_2(k+1),1}$. For every triple
$(g_n,g_\tau,\varGamma)$  in 
$G_{k+1}^n\times G_{k+1}^\tau\times \mathbb R^N$ with   $\varGamma=(\varGamma_1,\ldots,\varGamma_N)$ we define:
\begin{subequations}
\label{def_LV}
\begin{alignat}{3}
\mathsf L_{k}^V(g_n,g_\tau,\varGamma)&=\bar\Delta_{k}\mathsf L_{k+1}^S(g_n,g_\tau,\varGamma)=\bar\Delta_{k} L_{k+1}^\tau({\mathsf T}_{k+1} g_n-g_\tau)+\sum_{j=1}^N\varGamma_j\zeta_j&\qquad& \text{if }k=0,1,2,\\
\label{def_LVK}
\mathsf L_{k}^V(g_n,g_\tau,\varGamma)&=\Delta_{k}\mathsf L_{k+1}^S(g_n,g_\tau,\varGamma)&& \text{if }k\leqslant -1.
\end{alignat}
\end{subequations}
The operator $\mathsf L_{k}^V$ is valued in the space $ V^{\rm b}_k$ defined by:
$$ V^{\rm b}_k=\bar\Delta_k S^{\rm b}_{k+1}=\mathfrak H_V^{k}\oplus \mathbb F_V^\ast\quad\text{ if }k=0,1,2\qquad\text{ and }\qquad
 V^{\rm b}_k=\Delta_{k} S^{\rm b}_{k+1}\subset V_k\quad\text{ if }k\leqslant -1.$$
Let $T$ be a positive real number, $k$ be an integer such that $k\leqslant 1$ and assume that $\Sigma$ is of class $\mathcal C^{J_2(k+1),1}$. 
The operator $\mathsf L^V_{k+1}$ maps the space $G_{k+1}(T)$ (defined in \eqref{def_Gk}) into the space:
\begin{equation}
\label{def_mathfrakVT} 
 V^{\rm b}_k(T)=\begin{cases}
 H^1(0,T; V^{\rm b}_{k-1})\cap \mathcal C([0,T];V^{\rm b}_k)\cap
L^2(0,T; V^{\rm b}_{k+1})&\text{if }k=-1,0,1,\\
L^2(0,T; V^{\rm b}_{k+1})&\text{if }k\leqslant -2.
\end{cases}
\end{equation}
\end{definition}
As a direct consequence of Lemmas~\ref{LEM:lift_stream}, we are allowed to claim:
\begin{lemma}
\label{LEM:lift_vorti}
Let $k$ be an integer such that $k\leqslant 2$ and assume that $\Sigma$ is of class $\mathcal C^{I_2(k+1),1}$. Then the lifting operator for the vorticity:
$$
\mathsf L_k^V:G_{k+1}^n\times G_{k+1}^\tau\times \mathbb R^N\longrightarrow  V^{\rm b}_k,
$$
is well defined and is bounded. Moreover if $k$ and $k'$ are two integers such that $k'\leqslant k\leqslant 2$, then $\mathsf L_{k'}^V=\mathsf L_{k}^V$ in 
$G_{k+1}^n\times G_{k+1}^\tau\times \mathbb R^N$. It follows that for every positive  real number $T$ and every $k\leqslant 1$, providing that $\Sigma$ 
is of class $C^{J_2(k+1),1}$, the operator:
$$
L_{k+1}^V:G_{k+1}(T) \longrightarrow  V^{\rm b}_k(T),
$$
is well defined and bounded as well, the bound being uniform with respect to $T$.
\end{lemma}
This lemma  makes precise the expression of the vorticity corresponding to any prescribed boundary Dirichlet conditions for the 
velocity field on $\Sigma$.
\par
Definition~\ref{def:LkV} and Lemma~ \ref{LEM:lift_vorti} justify the lengthy construction of the the spaces $\bar V_k$ ($k=0,1,2$) carried out in Subsection~\ref{SEC:Nonhomogeneous_vorticity}. As already mentioned, 
the naive approach consisting in taking the Laplacian of a lifting stream function does not result in the correct result, first because 
the correct vorticity (in both cases of Fig.~\ref{figS03} and Fig.~\ref{figV13}) belongs actually to dual spaces, the expressions of which requires the construction
of the spaces $V_k$ and $\bar V_k$ and second because the circulation would vanish at the 
vorticity level. 
\section{Evolution Stokes problem in nonprimitive variables}
\label{SEC:evol_stokes}
The evolution Stokes problem, stated in the original primitive variables $(u,p)$, reads:
\begin{subequations}
\label{eq:main_stokes:1}
\begin{alignat}{3}
\label{eq:main_stokes:u_1a}
\partial_t  u-\nu \Delta u+\nabla \Big(\frac{p}{\varrho}\Big)&=  f&\quad&\text{in }\F_T\\ 
\nabla\cdot u&=0&&\text{in }\F_T\\ 
u&=b&&\text{on }\Sigma_T\\ 
\label{eq_stokes:initi}
u(0)&=u^{\rm i}&&\text{in }\F,
\end{alignat}
\end{subequations}
where the source term $f$, the boundary data $b$ and the initial data $u^{\rm i}$ are prescribed. 
We recall that the constant $\nu>0$ is the   kinematic viscosity of the fluid.
%
%
%
\subsection{Homogeneous boundary conditions}
We have  at our disposal all the material allowing to deal with the evolution Stokes problem in terms of both the vorticity field and the stream function. 
\begin{definition}
in terms of the stream function, the evolution Stokes  problem (called $\psi-$Stokes problem) can be stated as follows:
Let $k$ be any integer, $T$ be a positive real number, $\psi^{\rm i}$ be in $S_k$ and $f_S$ be an element of $L^2(0,T;S_{k-1})$.
The Cauchy problem for the stream function with homogeneous boundary conditions, at  regularity level $k$, reads:
\begin{subequations}
\label{stream_cauchy}
\begin{alignat}{3}
\partial_t\psi+\nu {\mathsf A}_{k+1}^S\psi&=f_S&\quad&\text{in }\F_T,\\
\psi(0)&=\psi^{\rm i}&&\text{in }\F.
\end{alignat}
\end{subequations}
\par
Problem \eqref{stream_cauchy} can be rephrased in terms of the vorticity field: Let $k$ be any integer, $T$ be a positive real number, $\omega^{\rm i}$ be in $V_k$ and $f_V$ be an element of $L^2(0,T;V_{k-1})$.
The Cauchy problem for the vorticity field, called $\omega-$Stokes problem, at regularity level $k$ reads:
\begin{subequations}
\label{vorticity_cauchy}
\begin{alignat}{3}
\partial_t\omega+\nu {\mathsf A}_{k+1}^V\omega&=f_V&\quad&\text{in }\F_T,\\
\omega(0)&=\omega^{\rm i}&&\text{in }\F.
\end{alignat}
\end{subequations}
 \end{definition}
%
For every integer $k$, we introduce the function spaces:
\begin{subequations}
\label{def_home_spaces}
\begin{align}
\label{defST}
S_k(T)&=H^1(0,T;S_{k-1})\cap {\mathcal C}([0,T];S_k)\cap L^2(0,T;S_{k+1}),\\
V_k(T)&=H^1(0,T;V_{k-1})\cap {\mathcal C}([0,T];V_k)\cap L^2(0,T;V_{k+1}).
\end{align}
\end{subequations}
Invoking for instance \cite[Theorem 4.1]{Lions:1972aa} or simply Proposition~\ref{prop:cauchy_abstract} (we felt somewhat uncomfortable with quoting general results on semigroups in Banach spaces in such a
simple case for which everything can be shown ``by hand''; see the short subsection~ \ref{SUB:Semigroup}), we claim:
%
\begin{prop}
\label{sol_space_O}
For every integer $k$, every $T>0$, every $\psi^{\rm i}\in S_k$ and every $f_S\in L^2(0,T;S_{k-1})$, there exists a unique solution $\psi$ to problem \eqref{stream_cauchy} in the space $S_k(T)$.
Moreover, there exists a real positive constant $\mathbf c_\nu$ (depending on $\nu$ but uniform in $\F$, $k$ and $T$) such that:
\begin{subequations}
\begin{equation}
\label{estim_psi}
\|\psi\|_{S_k(T)}\leqslant \mathbf c_\nu\Big(\|\psi^{\rm i}\|_{S_k}+\|f_S\|_{L^2(0,T;S_{k-1})}\Big).
\end{equation}
For every integer $k$, every $T>0$, every $\omega^{\rm i}\in V_k$ and every $f_V\in L^2(0,T;V_{k-1})$, there exists a unique solution $\omega$ to problem \eqref{vorticity_cauchy} in the space $V_k(T)$.
Moreover, the following estimate holds with the same constant $\mathbf c_\nu$ as in \eqref{estim_psi}:
\begin{equation}
\|\omega\|_{V_k(T)}\leqslant \mathbf c_\nu\Big(\|\omega^{\rm i}\|_{V_k}+\|f_V\|_{L^2(0,T;V_{k-1})}\Big).
\end{equation}
\end{subequations}
\end{prop}
%
The solutions $\psi$ and $\omega$ to problems \eqref{stream_cauchy} and \eqref{vorticity_cauchy} respectively, satisfy the following 
exponential decay estimates:
\begin{lemma}
\label{LEM:exp_decay}
Let $\psi$ be a solution to the Cauchy problem \eqref{stream_cauchy} in the space $S_k(T)$ for some integer $k$, some source term $f_S\in L^2(0,T;S_{k-1})$ 
and some initial condition $\psi^{\rm i}\in S_k$. Then, the following estimate holds:
\begin{subequations}
\begin{equation}
\label{eq:expo_decr_stream}
\|\psi(t)\|_{S_k} \leqslant e^{-[ \nu(1-\varepsilon)\lambda_\F] t}\Big[\|\psi^{\rm i}\|_{S_k}^2+\frac{1}{2\nu\varepsilon}\|f_S\|^2_{L^2(0,T;S_{k-1})}\Big]^{{\frac12}}
\text{ for all }t\in[0,T] \text{ and }\varepsilon\in (0,1),
\end{equation}
where $\lambda_\F>0$ is the constant defined in Corollary~\ref{COR:lambda}. If $f_S=0$, we can choose $\varepsilon=0$ in \eqref{eq:expo_decr_stream}.
\par
Let $\omega$ be a solution to the Cauchy problem \eqref{vorticity_cauchy} in the space $V_k(T)$ for some integer $k$, some source term $f_V\in L^2(0,T;V_{k-1})$ 
and some initial condition $\omega^{\rm i}\in V_k$. Then, the following estimate holds:
\begin{equation}
\label{eq:expo_decr}
\|\omega(t)\|_{V_k} \leqslant e^{-[ \nu(1-\varepsilon)\lambda_\F] t}\Big[\|\omega^{\rm i}\|_{V_k}^2+\frac{1}{2\nu\varepsilon}\|f_V\|^2_{L^2(0,T;V_{k-1})}\Big]^{{\frac12}}
\text{ for all }t\in[0,T] \text{ and }\varepsilon\in (0,1).
\end{equation}
\end{subequations}
If $f_V=0$, we can choose $\varepsilon=0$ in \eqref{eq:expo_decr}.
\end{lemma}
%
We can easily connect problems \eqref{stream_cauchy} and  \eqref{vorticity_cauchy} by means of either the operators ${\mathsf P}_k$ and ${\mathsf Q}_k$ or with 
the operator $\Delta_k$. The proof is straightforward, resting  on the commutative diagrams of Fig.~\ref{diag_3} and Fig.~\ref{diag_2}:
%
\begin{theorem} 
\label{theo:equiv}
Let $k$ and $k'$ be two integers and $T$ be a positive real number. Let $\psi$ be 
the solution in $S_k(T)$ to Problem \eqref{stream_cauchy} with source term $f_S\in L^2(0,T;S_{k-1})$ and initial 
condition $\psi^{\rm i}\in S_k$. Let $\omega$ be the solution in $V_{k'}(T)$ to Problem \eqref{vorticity_cauchy} with source term $f_V\in L^2(0,T;V_{k'-1})$ and initial 
condition $\omega^{\rm i}\in V_{k'}$. 
\par
\noindent If $k'=k+1$, then the following assertions are equivalent:
\begin{my_enumerate}
\item $\omega^{\rm i}={\mathsf P}_{k+1}\psi^{\rm i}$ and for a.e. $t\in(0,T)$, $f_V(t)={\mathsf P}_kf_S(t)$;
\item For a.e. $t\in(0,T)$, $\omega(t)={\mathsf P}_{k+2}\psi(t)$.
\end{my_enumerate}
If $k'=k-1$ 
then the following assertions are equivalent:
\begin{my_enumerate}
\item $\omega^{\rm i}=\Delta_{k-1}\psi^{\rm i}$ and for a.e. $t\in(0,T)$,  $f_V(t)=\Delta_{k-2}f_S(t)$;
\item For a.e. $t\in(0,T)$, $\omega(t)=\Delta_{k}\psi(t)$.
\end{my_enumerate}
In addition to the data already introduced, let  %
$$u\in H^1(0,T;\mathbf J_{k-1})\cap {\mathcal C}([0,T];\mathbf J_{k})\cap L^2(0,T;\mathbf J_{k+1})$$
be the solution to Problem \eqref{pb:test3} for some  $f_{\mathbf J}\in L^2(0,T; \mathbf J_{k-1})$ and $u^{\rm i}$ in $\mathbf J_{k}$.
The following assertions are equivalent:
\begin{my_enumerate}
\item $u^{\rm i}=\nabla_k^\perp \psi^{\rm i}$ and  for a.e. $t\in(0,T)$, $f_{\mathbf J}(t)=\nabla^\perp_{k-1} f_S(t)$;
\item For a.e. $t\in(0,T)$, $u(t)=\nabla^\perp_{k+1}\psi(t)$.
\end{my_enumerate} 
\end{theorem}
%
\subsection{Nonhomogeneous boundary conditions}
%
Following the definition \eqref{def_home_spaces} of $S_k(T)$ and $V_k(T)$, 
\eqref{def_mathfrakST} of  $ S^{\rm b}_k(T)$ and \eqref{def_mathfrakVT} of  $ V^{\rm b}_k(T)$ we introduce for 
every real positive number $T$ and every integer $k\leqslant 2$:
\begin{subequations}
\begin{equation}
\label{defSbarT}
\bar S_k(T)=H^1(0,T;\bar S_{k-1})\cap {\mathcal C}([0,T];\bar S_k)\cap L^2(0,T;\bar S_{k+1}),
\end{equation}
where we recall that the spaces  $\bar S_k$  are defined in \eqref{def_mathcal_Sk}.
The  counterpart stated in terms of the vorticity is,  for every integer $k\leqslant 1$, the space (see Fig.~\ref{figS03} and Fig.~\ref{figV13}):
\begin{equation}
\bar V_k(T)=H^1(0,T;\bar V_{k-1})\cap {\mathcal C}([0,T];\bar V_{k})\cap L^2(0,T;\bar V_{k+1}),
\end{equation}
where the spaces $\bar V_k$ are defined in \eqref{def_mathcal_Vk}.
\end{subequations}
%
%
\begin{definition}
\label{def:stokes_unhom}
Let a positive real number $T$, an integer $k\leqslant 1$, a source term $f_S\in L^2(0,T;S_{k-1})$, an initial data $\psi^{\rm i}\in \bar S_k$ 
  and a triple $(g_n,g_\tau,\varGamma)\in G_k(T)$ be given. Define 
  $\psi^{\rm i}_0=\psi^{\rm i}-\mathsf L_k^S(g_n(0),g_\tau(0),\varGamma(0))$ when $k=0,1$ and $\psi^{\rm i}_0=\psi^{\rm i}$ when $k\leqslant -1$. 
  Finally, assume that $\Sigma$ is of class $C^{J_2(k),1}$
and that the following compatibility condition holds:
\begin{equation}
\psi^{\rm i}_0\in S_k\qquad \text{if }k=1.
\end{equation}
We say that a function $\psi\in \bar S_k(T)$ is 
solution of the evolution $\psi-$Stokes problem   satisfying the Dirichlet boundary conditions on $\Sigma_T$ as described in \eqref{def:boundary_cond} 
by the triple  $(g_n,g_\tau,\varGamma)$ if:
\begin{enumerate}
\item When $k=0$ or $k=1$: There exists $\psi_0\in S_k(T)$ solution to the homogeneous $\psi-$Stokes Cauchy problem
\begin{subequations}
\begin{alignat}{3}
\partial_t \psi_0+\nu {\mathsf A}^S_{k+1}\psi_0&=-\partial_t\mathsf L_{k+1}^S(g_n,g_\tau,\varGamma)+f_S&\qquad&\text{in }\F_T,\\
\psi_0(0)&=\psi^{\rm i}_0&&\text{in }\F,
\end{alignat}
\end{subequations}
such that $\psi=\psi_0+\mathsf L_{k+1}^S(g_n,g_\tau,\varGamma)$.
\item When $k\leqslant -1$:  The function $\psi$ is the solution to the Cauchy problem:
 \begin{subequations}
 \label{weak_sol_nonhomo}
\begin{alignat}{3}
\label{weak_sol_nonhomo:1}
\partial_t \psi +\nu {\mathsf A}^S_{k+1}\psi &=\nu\mathsf A_{k+1}^S\mathsf L_{k+1}^S(g_n,g_\tau,\varGamma)+f_S
&\qquad&\text{in }\F_T,\\
\psi (0)&=\psi^{\rm i}_0&&\text{in }\F.
\end{alignat}
\end{subequations}
\end{enumerate}
\end{definition}
The case $k=2$ is more involved and will be treated in Section~\ref{SEC:more_regular}. 
The difference of definition depending on the level of regularity $k$ is worth some additional explanation. Before that, combining Proposition~\ref{sol_space_O} 
and Lemma~\ref{LEM:lift_vorti}, we are allowed to claim:
%
\begin{prop}
\label{exist:sol:stokes}
Every $\psi-$Stokes problem as stated in Definition~\ref{def:stokes_unhom} admits a unique solution. 
Moreover, there exists a positive constant $\mathbf c_{[k,\F,\nu]}$ uniform in $T$ such that the solution 
$\psi\in \bar S_k(T)$ satisfies the estimate:
\begin{equation}
\label{estim:ini}
\|\psi\|_{\bar S_k(T)}\leqslant\mathbf c_{[k,\F,\nu]}\Big[\|\psi^{\rm i}_0\|^2_{S_k}+\|f_S\|^2_{L^2(0,T;S_{k-1})}+\|(g_n,g_\tau,\varGamma)\|^2_{G_k(T)}\Big]^{{\frac12}}.
\end{equation}
%
\end{prop}
%
The consistency of Definition~\ref{def:stokes_unhom} is asserted by the following results:
%
%
\begin{prop}
\label{prop:uniqueness_def}
Let a positive real number $T$, an integer $k$, a source term $f_S\in L^2(0,T;S_{k-1})$, an initial data $\psi^{\rm i}\in \bar S_k$ 
  and a triple $(g_n,g_\tau,\varGamma)\in G_k(T)$ be given as in Definition~\ref{def:stokes_unhom}. Denote by  
  $\psi^k$   the solution 
  whose existence and uniqueness  in the space $\bar S_k(T)$ are asserted in Proposition~\ref{exist:sol:stokes}.
  \par
Let  $k'$ be any integer lower than $k$ and, all other data remaining equal,  denote by $\psi^{k'}$ the corresponding solution in $\bar S_{k'}(T)$.
 Then $\psi^k=\psi^{k'}$.
\end{prop}
\begin{proof}
The proposition is obvious when $k$ and $k'$ are both nonnegative or when $k$ and $k'$ are both negative, so let us focus on the case 
$k=0$ and $k'=-1$ and compare the solutions $\psi^0$ and $\psi^{-1}$.
\par
By definition, the function $\psi^0$ solves the Cauchy problem:
\begin{alignat*}{3}
\partial_t\big(\psi^0-\mathsf L_{1}^S(g_n,g_\tau,\varGamma)\big)
+\nu {\mathsf A}^S_{1}\big(\psi^0-\mathsf L_1^S(g_n,g_\tau,\varGamma)\big)&=-\partial_t\mathsf L_1^S(g_n,g_\tau,\varGamma)+f_S&\qquad&\text{in }\F_T,\\
\psi^0(0)&=\psi^{\rm i}&&\text{in }\F.
\end{alignat*}
Since the operator $\mathsf A^S_{0}$ extends the operator $\mathsf A^S_{1}$ to $S_{0}$, the function $\psi^0$ belongs to $L^2(0,T;\bar S_1)
\subset L^2(0,T;S_0)$ 
 and the lifting 
operator $\mathsf L^S_{1}$ is valued in $ S^{\rm b}_{1}$, which is a subspace of $S_{0}$, we are allowed to write that:
$${\mathsf A}^S_{1}\big(\psi^0-\mathsf L_1^S(g_n,g_\tau,\varGamma)\big)={\mathsf A}^S_0\psi^0-{\mathsf A}^S_0\mathsf L_1^S(g_n,g_\tau,\varGamma).$$
It follows that $\psi^0$ solves as well the Cauchy problem:
\begin{alignat*}{3}
\partial_t \psi^0 +\nu {\mathsf A}^S_{0}\psi^0 &=\nu\mathsf A_{0}^S\mathsf L_1^S(g_n,g_\tau,\varGamma)+f_S
&\qquad&\text{in }\F_T,\\
\psi^0(0)&=\psi^{\rm i}&&\text{in }\F,
\end{alignat*}
a solution of which is $\psi^{-1}$. The proof is now completed.
\end{proof}
The definition of weak solutions (i.e. for negative integers $k$) with nonhomogeneous boundary conditions given in Definition~\ref{def:stokes_unhom} can be rephrased 
 by means of the duality method (or transposition method; see \cite{Raymond:2007aa} and references therein).
%
\begin{prop}
\label{prop:def_transposition}
Let data be given as in Definition~\ref{def:stokes_unhom} and assume that $k$ is a negative integer and $f_S=0$. Denote by $\psi$ the unique solution 
to the corresponding Cauchy nonhomogeneous $\psi-$Stokes problem \eqref{weak_sol_nonhomo}. Then for 
every $\vartheta\in L^2(0,T;S_{-k-1})$  and $\theta\in S_{-k}(T)$ solution to the backward Cauchy problem:
\begin{alignat*}{3}
-\partial_t \theta+\nu {\mathsf A}^S_{-k+1}\theta &=\vartheta
&\qquad&\text{in }\F_T,\\
\theta(T)&=0&&\text{in }\F,
\end{alignat*}
the following identity holds:
\begin{equation}
\label{eq:duality}
\int_0^T\big\langle \psi,\vartheta\big\rangle_{S_{k+1},S_{-k-1}}\dt=\big\langle \psi^{\rm i},\theta(0)\big\rangle_{S_k,S_{-k}}-
\nu\int_0^T\big\langle \Delta_{-k}\theta\big|_\Sigma,b_\tau\big\rangle_{H^{-k-\frac12}(\Sigma),H^{k+\frac12}(\Sigma)}\dt,
\end{equation}
where (see identities \eqref{def:boundary_cond}):
$$b_\tau=\mathsf T_{k+1}g_n-g_\tau+\sum_{j=1}^N\varGamma_j \frac{\partial\xi_j}{\partial n}\bigg|_\Sigma.$$
\end{prop}
\begin{proof}
Equation \eqref{weak_sol_nonhomo:1} holds in $L^2(0,T;S_{k-1})$, which is the dual space of $L^2(0,T;S_{-k+1})$. Forming the duality pairing  
of \eqref{weak_sol_nonhomo:1} with $\theta$ yields:
$$\int_0^T\big\langle\partial_t \psi,\theta\big\rangle_{S_{k-1},S_{-k+1}}\dt +\nu \int_0^T\big\langle{\mathsf A}^S_{k+1}\psi,\theta\big\rangle_{S_{k-1},S_{-k+1}}\dt  =
\nu\int_0^T\big\langle
\mathsf A_{k+1}^S\mathsf L_{k+1}^S(g_n,g_\tau,\varGamma),\theta\big\rangle_{S_{k-1},S_{-k+1}}\dt.$$
Integrating by parts and using the definition of the operator $\mathsf A_{k+1}^S$, we obtain for the term in the left hand side:
\begin{multline}
\label{left_hand}
\int_0^T\big\langle\partial_t \psi,\theta\big\rangle_{S_{k-1},S_{-k+1}}\dt +\nu \int_0^T\big\langle{\mathsf A}^S_{k+1}\psi,\theta\big\rangle_{S_{k-1},S_{-k+1}}\dt =\\
-\big\langle \psi^{\rm i},\theta(0)\big\rangle_{S_k,S_{-k}}+
\int_0^T\big\langle\psi,-\partial_t \theta+\mathsf A^S_{-k+1}\theta\big\rangle_{S_{k+1},S_{-k-1}}\dt.
\end{multline}
The right hand side term is dealt with as follows:
$$
\int_0^T\big\langle
\mathsf A_{k+1}^S\mathsf L_{k+1}^S(g_n,g_\tau,\varGamma),\theta\big\rangle_{S_{k-1},S_{-k+1}}\dt
=\int_0^T\big\langle
\mathsf L_{k+1}^S(g_n,g_\tau,\varGamma),\mathsf A_{-k+1}^S\theta\big\rangle_{S_{k+1},S_{-k-1}}\dt,$$
where, by definition (see \eqref{def_LkS}):
\begin{multline}
\label{right_hand}
\int_0^T\Big\langle
\mathsf L_{k+1}^S(g_n,g_\tau,\varGamma),\mathsf A_{-k+1}^S\theta\Big\rangle_{S_{k+1},S_{-k-1}}\dt\\=
\int_0^T \Big\langle\mathsf L_{k+1}^\tau({\mathsf T}_{k+1} g_n-g_\tau)+\sum_{j=1}^N\varGamma_j \xi_j,\mathsf A_{-k+1}^S\theta
\Big\rangle_{S_{k+1},S_{-k-1}}\dt.
\end{multline}
Since the index $k$ is negative, $\mathsf A_{-k+1}^S\theta$ belongs to $S_0$. It follows that for 
every $j\in\{1,\ldots,N\}$:
\begin{equation}
\label{eq:last_one}
\big\langle \xi_j,\mathsf A_{-k+1}^S\theta\big\rangle_{S_{k+1},S_{-k-1}}=-(\nabla  \xi_j,\nabla \Delta_{-k} \theta)_{\mathbf L^2(\F)}
=-\int_\Sigma \frac{\partial\xi_j}{\partial n}\Delta_{-k}\theta\ds.
\end{equation}
Recalling the definition \eqref{eq:def_Lktau} of the operator $\mathsf L_{k+1}^\tau$ and the factorization \eqref{def_Bk}
 of $\Delta_k$ and then gathering \eqref{left_hand}, \eqref{right_hand} 
and \eqref{eq:last_one}, we obtain indeed  \eqref{eq:duality} and complete the proof.
\end{proof}
Definition~\ref{def:stokes_unhom} and Propositions~\ref{exist:sol:stokes}, \ref{prop:uniqueness_def} and \ref{prop:def_transposition} can 
be restated in terms of the vorticity field.
\begin{definition}
\label{def:stokes_unhom_omega}
Let a positive real number $T$, a nonpositive integer $k$, a source term $f_V\in L^2(0,T;V_{k-1})$, an initial data $\omega^{\rm i}\in \bar V_k$ 
  and a triple $(g_n,g_\tau,\varGamma)\in G_{k+1}(T)$ be given. Define 
  $\omega^{\rm i}_0=\omega^{\rm i}-\mathsf L_k^V(g_n(0),g_\tau(0),\varGamma(0))$ when $k=-1,0$ and $\omega^{\rm i}_0=\omega^{\rm i}$ when $k\leqslant -2$. 
  Finally, assume that $\Sigma$ is of class $C^{J_2(k+1),1}$
and that the following compatibility condition holds:
\begin{equation}
\omega^{\rm i}_0\in V_k\qquad \text{if }k=0.
\end{equation}
We say that a function $\omega\in \bar V_k(T)$ is 
solution of the evolution $\omega-$Stokes problem   satisfying the Dirichlet boundary conditions on $\Sigma_T$ as described in \eqref{def:boundary_cond} 
by the triple  $(g_n,g_\tau,\varGamma)$ if:
\begin{enumerate}
\item When $k=-1$ or $k=0$: There exists $\omega_0\in V_k(T)$ solution to the homogeneous $\omega-$Stokes Cauchy problem
\begin{subequations}
\begin{alignat}{3}
\label{eq:stokes_omega}
\partial_t \omega_0+\nu {\mathsf A}^V_{k+1}\omega_0&=-\partial_t\mathsf L_{k+1}^V(g_n,g_\tau,\varGamma)+f_V&\qquad&\text{in }\F_T,\\
\omega_0(0)&=\omega^{\rm i}_0&&\text{in }\F,
\end{alignat}
\end{subequations}
such that $\omega=\omega_0+\mathsf L_{k+1}^V(g_n,g_\tau,\varGamma)$.
\item When $k\leqslant -2$:  The function $\omega$ is the solution to the Cauchy problem:
 \begin{subequations}
 \label{weak_sol_nonhomo_omega}
\begin{alignat}{3}
\label{weak_sol_nonhomo_omega:1}
\partial_t \omega +\nu {\mathsf A}^V_{k+1}\omega&=\nu\mathsf A_{k+1}^V\mathsf L_{k+1}^V(g_n,g_\tau,\varGamma)+f_V
&\qquad&\text{in }\F_T,\\
\omega (0)&=\omega^{\rm i}_0&&\text{in }\F.
\end{alignat}
\end{subequations}
\end{enumerate}
\end{definition}
\begin{prop}
\label{exist:sol:stokes_omega}
Every $\omega-$Stokes problem as stated in Definition~\ref{def:stokes_unhom_omega} admits a unique solution. 
Moreover, there exists a positive constant $\mathbf c_{[k,\F,\nu]}$   (uniform in $T$) such that the solution 
$\omega\in \bar S_k(T)$ satisfies the estimate:
\begin{equation}
\|\omega\|_{\bar V_k(T)}\leqslant \mathbf c_{[k,\F,\nu]}\Big[\|\omega^{\rm i}_0\|_{V_k}^2+\|f_V\|_{L^2(0,T;V_{k-1})}^2+\|(g_n,g_\tau,\varGamma)\|_{G_{k+1}(T)}^2\Big]^{\frac12}.
\end{equation}
\end{prop}
%
\begin{prop}
\label{prop:uniqueness_def_omega}
Let a positive real number $T$, an integer $k$, a source term $f_V\in L^2(0,T;V_{k-1})$, an initial data $\omega^{\rm i}\in \bar V_k$ 
  and a triple $(g_n,g_\tau,\varGamma)\in G_{k+1}(T)$ be given as in  Definition~\ref{def:stokes_unhom_omega}. Denote by  
  $\omega^k$   the solution 
  whose existence and uniqueness  in the space $\bar V_k(T)$ are asserted in Proposition~\ref{exist:sol:stokes_omega}.
  \par
Let  $k'$ be any integer lower than $k$ and, all other data remaining equal,  denote by $\omega^{k'}$ the corresponding solution in $\bar V_{k'}(T)$.
 Then $\omega^k=\omega^{k'}$.
\end{prop}
\begin{prop}
\label{prop:def_transposition_omega}
Let data be given as in Definition~\ref{def:stokes_unhom_omega} with $k\leqslant -2$ and $f_V=0$. Denote by $\omega$ the unique solution 
to the corresponding Cauchy nonhomogeneous $\omega-$Stokes problem \eqref{weak_sol_nonhomo_omega}. Then for 
every $\vartheta\in L^2(0,T;V_{-k-1})$  and $\theta\in V_{-k}(T)$ solution to the backward Cauchy problem:
\begin{alignat*}{3}
-\partial_t \theta+\nu {\mathsf A}^V_{-k+1}\theta &=\vartheta
&\qquad&\text{in }\F_T,\\
 \theta(T)&=0&&\text{in }\F,
\end{alignat*}
the following identity holds:
\begin{equation}
\label{eq:duality_1}
\int_0^T\big\langle \omega,\vartheta\big\rangle_{V_{k+1},V_{-k-1}}\dt=\big\langle \omega^{\rm i},\theta(0)\big\rangle_{V_k,V_{-k}}-
\nu\int_0^T\big\langle A_{-k+1}^V\theta\big|_\Sigma,b_\tau\big\rangle_{H^{-k-\frac32}(\Sigma),H^{k+\frac32}(\Sigma)}\dt,
\end{equation}
where (see identities \eqref{def:boundary_cond}):
$$b_\tau=\mathsf T_{k+2}g_n-g_\tau+\sum_{j=1}^N\varGamma_j \frac{\partial\xi_j}{\partial n}\bigg|_\Sigma.$$
\end{prop}
\begin{proof}
We form the duality pairing of \eqref{weak_sol_nonhomo_omega:1} in $L^2(0,T;V_{k-1})$ with $\theta$ in $L^2(0,T;V_{-k+1})$, then the proof follows mainly 
the lines of the proof of Proposition~\ref{prop:def_transposition}. Let us focus on the right hand side term only, namely:
$$\int_0^T\Big\langle \mathsf A^V_{k+1}\mathsf L^V_{k+1}(g_n,g_\tau,\varGamma),\theta\Big\rangle_{V_{k-1},V_{-k+1}}\ds.$$
According to \eqref{eq:def_Akdual}, the duality pairing can be turned into:
$$\Big\langle \mathsf A^V_{k+1}\mathsf L^V_{k+1}(g_n,g_\tau,\varGamma),\theta\Big\rangle_{V_{k-1},V_{-k+1}}=
\Big\langle \mathsf L^V_{k+1}(g_n,g_\tau,\varGamma),\mathsf A^V_{-k+1}\theta\Big\rangle_{V_{k+1},V_{-k-1}}.$$
Then, using the definition \eqref{def_LVK} of $\mathsf L_{k+1}^V$ and the second formula in Lemma \ref{LEM:prop:delta_negatif}, we obtain:
$$\Big\langle \mathsf L^V_{k+1}(g_n,g_\tau,\varGamma),\mathsf A^V_{-k+1}\theta\Big\rangle_{V_{k+1},V_{-k-1}}
=\Big\langle\mathsf L^S_{k+2}(g_n,g_\tau,\varGamma),\mathsf Q_{-k-1} A^V_{-k+1}\theta\Big\rangle_{S_{k+2},S_{-k-2}}.$$
Resting on formula \eqref{extend-kneg_2}, the last term is proven to be equal to:
$$ \Big\langle\mathsf L^S_{k+2}(g_n,g_\tau,\varGamma),A_{-k}^S \mathsf Q_{-k+1}  \theta\Big\rangle_{S_{k+2},S_{-k-2}},$$
and therefore it is very much alike the left hand side in \eqref{right_hand}. The proof is then completed after noticing that $-\Delta_{-k-1}\mathsf Q_{-k+1}=\mathsf A^V_{-k+1}$ 
(see for instance Fig.~\ref{diag_2}). 
\end{proof}
\section{Navier-Stokes equations in nonprimitive variables}
\label{SEC:Navier-Stokes}
%
%
\subsection{Estimates for the nonlinear advection term}
Following our rules of notation, we define for every positive integer $k$ and every positive time $T$, the time dependent space for the Kirchhoff potential:
\begin{equation}
\label{def:HKT}
\mathfrak H^k_K(T)=H^1(0,T;\mathfrak H^{k-1}_K)\cap \mathcal C([0,T];\mathfrak H^k_K)\cap L^2(0,T;\mathfrak H^{k+1}_K),
\end{equation}
where we recall that the spaces $\mathfrak H^k_K$ were defined in \eqref{def_HK}.
Then, we aim at establishing some (very classical) estimates for the nonlinear advection term of the 
Navier-Stokes equations. Denoting by $u$ a smooth 
velocity field 
defined in $\F$, an integration par parts yields  the equality:
$$(\nabla u u,\nabla^\perp\theta)_{\mathbf L^2(\F)}=(D^2\theta u,u^\perp)_{\mathbf L^2(\F)}\forallt \theta\in S_1.$$
The main estimates satisfied by the right hand side term are summarized in the lemma below:
%
\begin{lemma}
\label{estim_non_linear_term}
Let a stream function $\bar\psi$ be in $\bar S_1$ (this space being defined in \eqref{decomp_S02}) and a Kirchhoff 
potential $\varphi$ be in $\mathfrak H^2_K$ (defined in \eqref{def_HK}). Then the linear form:
\begin{equation}
\label{def_Lambda_S}
\varLambda^S_1(\bar\psi,\varphi):\theta\in S_1\longmapsto -(D^2\theta\nabla\bar\psi,\nabla^\perp\bar\psi)_{\mathbf L^2(\F)}
+(D^2\theta\nabla^\perp\varphi,\nabla^\perp\bar\psi)_{\mathbf L^2(\F)}-(D^2\theta\nabla\varphi,\nabla\bar\psi)_{\mathbf L^2(\F)}\in\mathbb R,
\end{equation}
is well defined and bounded. Moreover, there exists a positive constant $\mathbf c_{[\F,\nu]}$ such that:
\begin{subequations}
\begin{enumerate}
\item For every $\theta\in S_1$:
\begin{multline}
\label{estim:NLT}
|\langle \varLambda^S_1(\bar\psi+\theta,\varphi), \theta\rangle_{S_{-1},S_1}|\leqslant \frac{\nu}{2}\|\theta\|_{S_1}^2
+\mathbf c_{[\F,\nu]}\Big(\|\bar\psi\|^2_{\bar S_1}\|\bar\psi\|^{2}_{S_0}+\|\varphi\|_{\mathfrak H^2_K}^2\|\varphi\|_{\mathfrak H^1_K}^2\Big)\|\theta\|_{S_0}^2\\
+\mathbf c_{[\F,\nu]}\Big(\|\bar\psi\|^2_{\bar S_1}\|\bar\psi\|^{2}_{S_0}+\|\varphi\|_{\mathfrak H^2_K}^2\|\varphi\|_{\mathfrak H^1_K}^2\Big).
\end{multline}
\item 
\label{point_2}
For every pair $(\theta_1,\theta_2)\in S_1\times S_1$:
\begin{multline}
\label{eq:diff_theta}
|\langle \varLambda^S_1(\bar\psi+\theta_2,\varphi)-\varLambda^S_1(\bar\psi+\theta_1,\varphi), \varTheta\rangle_{S_{-1},S_1}|
\leqslant \frac{\nu}{2}\|\varTheta\|_{S_1}^2\\
+\mathbf c_{[\F,\nu]}\Big[\|\bar\psi\|^2_{\bar S_1}\|\bar\psi\|^{2}_{S_0}+\|\varphi\|_{\mathfrak H^2_K}^2\|\varphi\|_{\mathfrak H^1_K}^2
+\|\theta_1\|^2_{S_1}\|\theta_1\|^{2}_{S_0}\Big]\|\varTheta\|_{S_0}^2,
\end{multline}
where $\varTheta=\theta_2-\theta_1$.
\end{enumerate}
If  for some positive real number $T$, $\bar\psi$ belongs to ${\mathcal C}([0,T];S_0)\cap L^2(0,T;\bar S_1)$ and $\varphi$ to ${\mathcal C}([0,T];\mathfrak H_K^1)\cap L^2(0,T;\mathfrak H_K^2)$  then $\varLambda^S_1(\bar\psi,\varphi)$ is 
in $L^2(0,T;S_{-1})$ 
and there exists a positive constant $\mathbf c_\F$ such that:
\begin{equation}
\label{estim:NLT2}
\|\varLambda^S_1(\bar\psi,\varphi)\|_{L^2(0,T;S_{-1})}\leqslant \mathbf c_\F\Big[\|\bar\psi\|_{{\mathcal C}([0,T];S_0)}\|\bar\psi\|_{L^2(0,T;\bar S_1)}
+
\|\varphi\|_{{\mathcal C}([0,T];\mathfrak H_K^1)}\|\varphi\|_{L^2(0,T;\mathfrak H_K^2)}\Big].
\end{equation}
\end{subequations}
\end{lemma}
%
\begin{proof}
Considering the first term in the right hand side of \eqref{def_Lambda_S}, H\"older's inequality yields:
\begin{subequations}
\begin{equation}
\label{eq:estim_nonlin1}
|(D^2\theta\nabla\bar\psi,\nabla^\perp\bar\psi)_{\mathbf L^2(\F)}|\leqslant 
\|\nabla\bar\psi\|_{\mathbf L^4(\F)}^2\|\theta\|_{S_1}\forallt \theta\in S_1.
\end{equation}
Then Sobolev embedding Theorem followed by an interpolation inequality  between the spaces $L^2(\F)$ and $H^1(\F)$ leads to: 
$$\|u\|_{L^4(\F)}\leqslant \mathbf  c_\F\|u\|_{H^{{\frac12}}(\F)}\leqslant \mathbf c_\F\|u\|_{L^2(\F)}^{{\frac12}}\|u\|_{H^1(\F)}^{{\frac12}}
\forallt u\in H^1(\F).$$
Combining this inequality with Lemma~\ref{decomp_S1} (equivalence of the norms in $\bar S_1$ and $H^2(\F)$), we obtain first:
$$
\|\nabla \bar\psi \|_{\mathbf L^4(\F)}\leqslant \mathbf c_\F\|\bar\psi \|_{S_0}^{{\frac12}}\|\bar\psi \|_{\bar S_1}^{{\frac12}}.
$$
Once plugged in \eqref{eq:estim_nonlin1}, it gives rise to:
\begin{equation}
\label{eq:start_over}
|(D^2\theta\nabla\bar\psi,\nabla^\perp\bar\psi)_{\mathbf L^2(\F)}|\leqslant 
\mathbf c_\F\|\bar\psi \|_{S_0} \|\bar\psi \|_{\bar S_1} \|\theta\|_{S_1}\forallt \theta\in S_1.
\end{equation}
\end{subequations}
Based on the same arguments, it is then straightforward to prove the existence of a positive constant $\mathbf c_\F$ such that:
$$|\langle \varLambda^S_1(\bar\psi,\varphi), \theta\rangle_{S_{-1},S_1}|\leqslant \mathbf c_\F\Big(\|\bar\psi\|_{\bar S_1}\|\bar\psi\|_{S_0} + 
\|\varphi\|_{\mathfrak H^2_K}\|\varphi\|_{\mathfrak H^1_K}\Big)\|\theta\|_{S_1}\forallt \theta\in S_1.
$$
This shows that the linear form $\varLambda^S_1(\bar\psi,\varphi)$ is indeed bounded and satisfies:
\begin{equation}
\label{eq:estim_normS-1}
\|\varLambda^S_1(\bar\psi,\varphi)\|_{S_{-1}}\leqslant \mathbf c_\F\Big(\|\bar\psi\|_{\bar S_1}\|\bar\psi\|_{S_0} + 
\|\varphi\|_{\mathfrak H^2_K}\|\varphi\|_{\mathfrak H^1_K}\Big).
\end{equation}
Let us move on to the estimate \eqref{estim:NLT}. Some of the terms vanishing after an integration by parts, we end up with 
the following equality:
\begin{multline}
\label{eq:sum_terms}
\langle \varLambda^S_1(\bar\psi+\theta,\varphi), \theta\rangle_{S_{-1},S_1}=-(D^2\theta\nabla\bar\psi,\nabla^\perp\bar\psi)_{\mathbf L^2(\F)}
+(D^2\theta\nabla^\perp\varphi,\nabla^\perp\bar\psi)_{\mathbf L^2(\F)}
\\
+(D^2\theta\nabla^\perp\varphi,\nabla^\perp\theta)_{\mathbf L^2(\F)}-(D^2\theta\nabla\varphi,\nabla\bar\psi)_{\mathbf L^2(\F)}.
\end{multline}
Addressing the first term in the right hand side, let us start over from the inequality \eqref{eq:start_over} to which 
we apply Young's inequality:
$$
|(D^2\theta\nabla\bar\psi,\nabla^\perp\bar\psi)_{\mathbf L^2(\F)}|\leqslant \frac{\nu}{8} \|\theta\|_{S_1}^2
+\mathbf c_{[\F,\nu]}  \|\bar\psi\|_{\bar S_1}^2\|\bar\psi\|^2_{S_0}.
$$
The four remaining terms in the right hand side of \eqref{eq:sum_terms} can be handle the same way and summing the resulting estimates yields 
\eqref{estim:NLT}.
\par
With the notation of the occurence \eqref {point_2} of the Lemma, some elementary algebra leads to:
$$\langle \varLambda^S_1(\bar\psi+\theta_2,\varphi)-\varLambda^S_1(\bar\psi+\theta_1,\varphi), \varTheta\rangle_{S_{-1},S_1}
=-(D^2\varTheta\nabla(\bar\psi+\theta_1),\nabla^\perp\varTheta)_{\mathbf L^2(\F)}
+(D^2\varTheta\nabla^\perp\varphi,\nabla^\perp\varTheta)_{\mathbf L^2(\F)},$$
and proceeding as for \eqref{estim:NLT} we quickly  obtain \eqref{eq:diff_theta}.
\par
Finally \eqref{estim:NLT2} derives straightforwardly  from \eqref{eq:estim_normS-1} and the proof is completed.
\end{proof}
In case the stream function is more regular, the nonlinear term satisfies better estimates:
\begin{lemma}
Let a stream function $\bar\psi$ be in $\bar S_2$ (see \eqref{decomp_barS2} for a definition) and a Kirchhoff 
potential $\varphi$ be in $\mathfrak H^2_K$ (see \eqref{def_HK}). Then the linear form $\varLambda^S_1(\bar\psi,\varphi)$ 
defined in \eqref{def_Lambda_S} extends to a continuous linear form in $S_0$  whose expression is:
\begin{equation}
\label{def_Lambda_S_strong}
\varLambda^S_0(\bar\psi,\varphi):\theta\in S_0\longmapsto (\Delta\bar\psi (\nabla^\perp\bar\psi+\nabla\varphi),\nabla\theta)_{\mathbf L^2(\F)}\in\mathbb R.
\end{equation}
Moreover, there exists a positive constant $\mathbf c_\F$ such that:
\begin{subequations}
\begin{equation}
\label{estim_varlambda_strong}
\|\varLambda^S_0(\bar\psi,\varphi)\|_{S_0}\leqslant \mathbf c_\F\Big(\|\bar\psi\|^2_{\bar S_1}+\|\varphi\|^2_{\mathfrak H^2_K}\Big)^{{\frac12}}
\|\bar\psi\|_{\bar S_1}^{\frac15}\| \bar\psi\|_{\bar S_2}^{\frac45}.
\end{equation}
Let $T$ be a positive real number and let $\bar\psi$ be in $\bar S_1(T)$, $\varphi$ be in 
$\mathfrak H^2_K(T)$ and $\theta$ be in 
$S_1(T)$ (these spaces being defined respectively in \eqref{defSbarT}, \eqref{def:HKT} and \eqref{defST}). Then $\varLambda^S_0(\bar\psi+\theta,\varphi)$ belongs to 
$L^2(0,T;S_0)$ and:
\begin{equation}
\label{estim_varlambda_strong_2}
\|\varLambda^S_0(\bar\psi+\theta,\varphi)\|_{L^2(0,T;S_0)}\leqslant \mathbf c_\F T^{\frac{1}{10}}\Big(\|\bar\psi\|^2_{\bar S_1(T)}+
\|\varphi\|^2_{\mathfrak H_K^{2}(T)}+
\|\theta\|^2_{S_1(T)}\Big).
\end{equation}
Finally, for every pair $(\theta_1,\theta_2)\in S_1(T)\times S_1(T)$:
\begin{multline}
\label{estim_varlambda_strong_3}
\|\varLambda^S_0(\bar\psi+\theta_2,\varphi)-\varLambda^S_0(\bar\psi+\theta_1,\varphi)\|_{L^2(0,T;S_0)}\\
\leqslant \mathbf c_\F T^{\frac{1}{10}}\Big(\|\bar\psi\|^2_{\bar S_1(T)}+\|\theta_1\|^2_{S_1(T)}+
\|\theta_2\|^2_{S_1(T)}+\|\varphi\|^2_{\mathfrak H_K^{2}(T)}\Big)^{{\frac12}}\|\theta_2-\theta_1\|_{S_1(T)}.
\end{multline}
\end{subequations}
\end{lemma}
\begin{proof}
Assume that $\theta$ belongs to $S_1$. Then, integrating by parts, we obtain:
\begin{multline}
\label{eq:grande_eq}
(\Delta\bar\psi (\nabla^\perp\bar\psi+\nabla\varphi),\nabla\theta)_{\mathbf L^2(\F)}=
-(D^2\varphi\nabla\bar\psi,\nabla\theta)_{\mathbf L^2(\F)}-(D^2\theta\nabla\bar\psi,\nabla^\perp\bar\psi)_{\mathbf L^2(\F)}
-(D^2\theta\nabla\bar\psi,\nabla\varphi)_{\mathbf L^2(\F)}.
\end{multline}
Integrating by parts again the first term in the right hand side, it comes:
\begin{subequations}
\begin{equation}
\label{eq:int_parts_1}
(D^2\varphi\nabla\bar\psi,\nabla\theta)_{\mathbf L^2(\F)}=-(\Delta\theta\nabla\varphi,\nabla\bar\psi)_{\mathbf L^2(\F)}
-(D^2\bar\psi\nabla\varphi,\nabla\theta)_{\mathbf L^2(\F)},
\end{equation}
while the last term can be rewritten as follows:
\begin{equation}
\label{eq:int_parts_2}
(D^2\bar\psi\nabla\varphi,\nabla\theta)_{\mathbf L^2(\F)}=(\nabla(\nabla\theta\cdot\nabla\bar\psi),\nabla\varphi)_{\mathbf L^2(\F)}
-(D^2\theta\nabla\bar\psi,\nabla\varphi)_{\mathbf L^2(\F)},
\end{equation}
\end{subequations}
and  the first term in the right hand side vanishes. Gathering \eqref{eq:int_parts_1} and \eqref{eq:int_parts_2} yields:
\begin{equation}
(D^2\varphi\nabla\bar\psi,\nabla\theta)_{\mathbf L^2(\F)}=-(\Delta\theta\nabla\varphi,\nabla\bar\psi)_{\mathbf L^2(\F)}+
(D^2\theta\nabla\bar\psi,\nabla\varphi)_{\mathbf L^2(\F)}=-(D^2\theta\nabla^\perp\bar\psi,\nabla^\perp\varphi)_{\mathbf L^2(\F)}.
\end{equation}
Replacing   this expression in \eqref{eq:grande_eq}, we recover indeed the definition \eqref{def_Lambda_S} of 
$\varLambda_1^S(\bar\psi,\varphi)$.
\par
Applying H\"older's inequality yields:
$$\|\Delta\bar\psi(\nabla^\perp\bar\psi+\nabla\varphi)\|_{\mathbf L^2(\F)}^2\leqslant \mathbf c_\F
\|\Delta\bar\psi\|_{L^2(\F)}^{\frac{2}{5}}\|\Delta\bar\psi\|_{L^4(\F)}^{\frac{8}{5}}\Big(\|\nabla\bar\psi\|_{\mathbf L^5(\F)}^2+
\|\nabla\varphi\|_{\mathbf L^5(\F)}^2\Big),$$
and then, Sobolev embedding Theorem leads straightforwardly to \eqref{estim_varlambda_strong}.
\par
From \eqref{estim_varlambda_strong}, we deduce that:
$$\|\varLambda^S_0(\bar\psi+\theta,\varphi)\|_{S_0}\leqslant
\mathbf c_\F\Big(\|\bar\psi\|^2_{\bar S_1}+\|\varphi\|^2_{\mathfrak H^2_K}+\|\theta\|_{S_1}^2\Big)^{\frac{3}{5}}\Big(
\|\bar\psi\|^2_{\bar S_2}+\|\theta\|_{S_2}^2\Big)^{\frac{2}{5}},$$
and therefore, in particular:
$$\|\varLambda^S_0(\bar\psi+\theta,\varphi)\|_{L^{\frac{5}{2}}(0,T;S_0)}\leqslant
\mathbf c_\F\Big(\|\bar\psi\|^2_{\bar S_1(T)}+
\|\varphi\|^2_{\mathfrak H_K^{2}(T)}+
\|\theta\|^2_{S_1(T)}\Big).$$
The estimate \eqref{estim_varlambda_strong_2} follows with H\"older's inequality. The last inequality \eqref{estim_varlambda_strong_3} is proved 
the same way.
\end{proof}
\begin{rem}
\begin{enumerate}
\item 
Denoting $u=\nabla^\perp\bar\psi+\nabla\varphi$ in the statement of the lemma, it is classical to verify that:
\begin{equation}
\label{alter_def_lambda0}
(\varLambda_0^S(\bar\psi,\varphi),\theta)_{S_0}=(\nabla u u ,\nabla^\perp\theta)_{\mathbf L^2(\F)}\forallt\theta\in S_0.
\end{equation}
\item
In this lemma, the assumption $\varphi\in \mathfrak H^2_K(T)$ is too strong. Indeed, the Kirchhoff  potential is not required to belong 
to $L^2(0,T;\mathfrak H^3_K)$ and one can verify that $\varphi\in \mathcal C([0,T];\mathfrak H^2_K)$ would 
be suffisant. However, the hypothesis $\varphi\in \mathfrak H^2_K(T)$ will be necessary later on.
\end{enumerate}
\end{rem}
%
\subsection{Weak solutions}
\begin{definition}
\label{defi:Navier_Stokes}
Let a positive real number $T$,   a source term $f_S\in L^2(0,T;S_{-1})$, an initial data $\psi^{\rm i}\in  S_0$ 
  and a triple $(g_n,g_\tau,\varGamma)\in G_0(T)$ be given.  Define 
  $\psi^{\rm i}_0=\psi^{\rm i}-\mathsf L_0^S(g_n(0),g_\tau(0),\varGamma(0))$ and assume that $\Sigma$ is of class $C^{3,1}$.
\par
We say that a stream  function $\psi\in  \bar S_0(T)$ is a  weak (or Leray)
solution to the $\psi-$Navier-Stokes equations satisfying the Dirichlet boundary conditions on $\Sigma_T$ as described in \eqref{def:boundary_cond} 
by the triple  $(g_n,g_\tau,\varGamma)$ if $\psi=\psi_b+\psi_{\ell}+\psi_\varLambda$ where:
\begin{enumerate}
\item The function $\psi_b$ accounts for the boundary conditions. It belongs to $S_0^b(T)$ defined in \eqref{def_mathfrakST} 
and is equal to $\mathsf L_1^S(g_n,g_\tau,\varGamma)$;
\item The function $\psi_{\ell}$ accounts for the source term and the initial condition. It is defined as the unique solution in  $S_0(T)$ 
of the homogeneous (linear) $\psi-$Stokes Cauchy problem
\begin{subequations}
\label{cauchy_for_psi_L}
\begin{alignat}{3}
\partial_t \psi_{\ell}+\nu {\mathsf A}^S_{1}\psi_{\ell}&=-\partial_t\psi_b+f_S&\qquad&\text{in }\F_T,\\
\psi_{\ell}(0)&=\psi^{\rm i}_0&&\text{in }\F.
\end{alignat}
\end{subequations}
\item The function $\psi_\varLambda$ accounts for the nonlinear advection term. It belongs to the space $S_0(T)$ and solves the nonlinear Cauchy problem:
\begin{subequations}
\label{cauchy_for_spi_lambda}
\begin{alignat}{3}
\label{eq1:cauchy_for_spi_lambda}
\partial_t \psi_\varLambda+\nu {\mathsf A}^S_{1}\psi_\varLambda&=-\varLambda^S_1(\psi_b+\psi_{\ell}+\psi_\varLambda,\varphi)&\qquad&\text{in }\F_T,\\
\psi_\varLambda(0)&=0&&\text{in }\F,
\end{alignat}
\end{subequations}
where $\varphi=\mathsf L_1^n g_n$ is the Kirchhoff potential that belongs to $\mathfrak H^1_K(T)$.
\end{enumerate}
\end{definition}
\begin{theorem}
\label{THEO:EXIST_NS}
For any set of data as described in Definition~\ref{defi:Navier_Stokes}, there exists a unique (weak) solution in $\bar S_0(T)$ to the $\psi-$Navier-Stokes equations.
\end{theorem}
\begin{proof}
The existence and uniqueness of $\psi_{\ell}$ is asserted by Proposition~\ref{exist:sol:stokes}.  Lemma~\ref{estim_non_linear_term} being granted, the proof of  existence and uniqueness of the function $\psi_\varLambda$
is quite similar to the proof \cite[Chap.~1, Section 6]{Lions_book1969}. Let us focus on the 
main differences and omit some details.
\par
Denote by $\bar\psi$ the function in $\bar S_0(T)$ equal to the sum $\psi_b+\psi_{\ell}$  (where the functions $\psi_b$ and $\psi_{\ell}$ are given) and notice that, according to Proposition~\ref{exist:sol:stokes} 
and Lemma~\ref{lift_potential}:
\begin{subequations}
\label{eq:estim_rhs}
\begin{align}
\|\bar\psi\|_{\bar S_0(T)}&\leqslant \mathbf c_{[\F,\nu]}\Big[\|\psi_0^{\rm i}\|^2_{S_0}+\|f_S\|^2_{L^2(0,T;S_{-1})}+\|(g_n,g_\tau,\varGamma)\|^2_{G_0(T)}\Big]^{{\frac12}}\\
\|\varphi\|_{\mathfrak H^1_K(T)}&\leqslant \mathbf c_\F\|g_n\|_{G^n_0(T)}\leqslant  \mathbf c_\F\|(g_n,g_\tau,\varGamma)\|_{G_0(T)}.
\end{align}
\end{subequations}

Then, for every positive integer $m$, 
introduce $\mathbb S_1^m$, the finite 
dimensional subspace of $S_1$ spanned by  the $m-$th first eigenvalues of $\mathcal A^S_1$ (loosely speaking, 
this operator is equal to $\mathsf A^S_1$ seen as an unbounded operator 
in $S_{-1}$ of domain $S_1$; see \eqref{def:unboundAk}). Denote by $\Pi_m$ the orthogonal projector from $S_1$ onto $\mathbb S_1^m$ and by 
$\Pi_m^\ast$ its adjoint for the duality pairing $S_{-1}\times S_1$. Finally, let $\psi_\varLambda^m$ be the unique solution 
in $\mathbb S_1^m$ of the Cauchy problem:
\begin{subequations}
\label{cauchy_for_spi_lambda_m}
\begin{alignat}{3}
\label{eq1:cauchy_for_spi_lambda_m}
\partial_t \psi_\varLambda^m+\nu {\mathsf A}^S_{1}\psi_\varLambda^m&=-\Pi_m^\ast\varLambda^S_1(\bar\psi+\psi_\varLambda^m,\varphi)&\qquad&\text{in }\F_T,\\
\psi_\varLambda^m(0)&=0&&\text{in }\F.
\end{alignat}
\end{subequations}
The existence and uniqueness of $\psi_\varLambda^m\in\mathcal C^1([0,T_m];S_1)$ on a time interval $(0,T_m)$ is guaranteed by Cauchy-Lipschitz Theorem. Forming now for any $s\in(0,T_m)$ the duality pairing of equation \eqref{eq1:cauchy_for_spi_lambda} set in $S_{-1}$ with $\psi_\varLambda^m(s)$ in $S_1$ and 
using the estimate \eqref{estim:NLT} for the nonlinear term,
we obtain:
\begin{equation}
\label{eq:energie}
\frac{\rm d}{{\rm d}t}\|\psi_\varLambda^m(s)\|_{S_0}^2+\nu\|\psi_\varLambda^m(s)\|_{S_1}^2\\
\leqslant \Phi(s)\|\psi_\varLambda^m(s)\|_{S_0}^2+
\Phi(s),
\end{equation}
where, for every $s$ in $(0,T)$:
$$
\Phi(s)=\mathbf c_{[\F,\nu]}\Big[\|\bar\psi(s)\|^2_{\bar S_1}\|\bar\psi(s)\|^{2}_{S_0}+
\|\varphi(s)\|_{\mathfrak H_K^1}^2\|\varphi(s)\|_{\mathfrak H_K^2}^2\Big].
$$
One easily verifies that $\Phi$ belongs to $L^1(0,T)$ and that, according to the estimates \eqref{eq:estim_rhs} above:
\begin{equation}
\label{eq:estim_Phi}
\|\Phi\|_{L^1(0,T)}\leqslant \mathbf c_{[\F,\nu]}\Big[\|\psi_0^{\rm i}\|^2_{S_0}+\|f_S\|^2_{L^2(0,T;S_{-1})}+\|(g_n,g_\tau,\varGamma)\|^2_{G_0(T)}\Big].
\end{equation}
Then, integrating \eqref{eq:energie} over $(0,t)$ for any $t\in(0,T_m)$ and introducing the constant $\lambda_\F$   defined in \eqref{estim_eigen} yields the estimate:
$$
\|\psi_\varLambda^m(t)\|_{S_0}^2+\int_0^t\big(\nu\lambda_\F-\Phi(s)\big)\|\psi_\varLambda^m(s)\|_{S_0}^2\ds\\
\leqslant \int_0^t\Phi(s)\ds\forallt t \in(0,T),
$$
which, with Gr\"onwall's inequality, leads to the estimate below, uniform in $t$ according to \eqref{eq:estim_Phi}: 
\begin{subequations}
\label{eq:estimC00}
\begin{equation}
\label{eq:estimC0}
\|\psi_\varLambda^m(t)\|_{S_0}\leqslant \|\Phi\|_{L^1(0,T)}^{{\frac12}}e^{\frac12\|\Phi\|_{L^1(0,T)}}\leqslant \mathbf c_{[\F,\nu,
\|\psi_0^{\rm i}\|_{S_0},\|f_S\|_{L^2(0,T;S_{-1})},\|(g_n,g_\tau,\varGamma)\|_{G_0(T)}]}.
\end{equation}
We deduce that  $T_m$ can be chosen equal to $T$.
Going back to inequality \eqref{eq:energie}, integrating it again over the time interval $(0,T)$ and using the estimate \eqref{eq:estimC0}, we get 
another estimate uniform in $t$:
\begin{equation}
\label{eq:estimC1}
\|\psi_\varLambda^m\|_{L^2(0,T;S_1)}\leqslant \mathbf c_{[\F,\nu,
\|\psi_0^{\rm i}\|_{S_0},\|f_S\|_{L^2(0,T;S_{-1})},\|(g_n,g_\tau,\varGamma)\|_{G_0(T)}]}.
\end{equation}
From identity \eqref{eq1:cauchy_for_spi_lambda_m}, we deduce now that:
$$\|\partial_t\psi^m_\varLambda\|_{L^2(0,T;S_{-1})}\leqslant \nu \|\psi^m_\varLambda\|_{L^2(0,T;S_1)}+\|\varLambda(\bar\psi+\psi^m_\varLambda,\varphi)\|_{L^2(0,T;S_{-1})}.$$
Combining \eqref{estim:NLT2}, \eqref{eq:estimC0} and \eqref{eq:estimC1} allows us to deduce that:
\begin{equation}
\|\partial_t\psi^m_\varLambda\|_{L^2(0,T;S_{-1})}\leqslant  \mathbf c_{[\F,\nu,
\|\psi_0^{\rm i}\|_{S_0},\|f_S\|_{L^2(0,T;S_{-1})},\|(g_n,g_\tau,\varGamma)\|_{G_0(T)}]}.
\end{equation}
\end{subequations}
It follows from the estimates \eqref{eq:estimC00} 
that the sequence $(\psi^m_\varLambda)_{m\geqslant 1}$ remains in a ball of $S_0(T)$, centered at the origin and whose radius depends only on 
$\F$, $\nu$ and the norms of the data $\|\psi_0^{\rm i}\|_{S_0}$, $\|f_S\|_{L^2(0,T;S_{-1})}$ and $\|(g_n,g_\tau,\varGamma)\|_{G_0(T)}$. The existence of a solution as limit of a subsequence of $(\psi^m_\varLambda)_m$ is next obtained, 
following exactly the lines of the proof \cite[Chap.~1, Section 6]{Lions_book1969}. 
\par
Let us address now the uniqueness of the solution.  We denote by $\varPsi_\varLambda$ the difference $\psi^2_\varLambda-\psi^1_\varLambda$ 
between two solutions  to the Cauchy problem \eqref{cauchy_for_spi_lambda} and this function satisfies:
$$\partial_t\varPsi_\varLambda+\nu\mathsf A^S_1\varPsi_\varLambda=-\varLambda^S_1(\bar\psi+\psi^2_\varLambda,\varphi)+
\varLambda^S_1(\bar\psi+\psi^1_\varLambda,\varphi)\qquad\text{in }\F_T.$$
Forming, for a.e. $t\in(0,T)$, the duality pairing of this identity set in $S_{-1}$ with $\varPsi_\varLambda(t)\in S_1$ and using the inequality \eqref{eq:diff_theta} 
 results in the estimate:
$$\frac{\rm d}{{\rm d}t}\|\varPsi_\varLambda(t)\|_{S_0}^2+\Big[\nu\lambda_\F-\mathbf c_{[\F,\nu]}\big(\|\bar\psi\|^2_{\bar S_1}\|\bar\psi\|^{2}_{S_0}+\|\varphi\|_{\mathfrak H^2_K}^2\|\varphi\|_{\mathfrak H^1_K}^2
+\|\psi_\varLambda^1\|^2_{S_1}\|\psi_\varLambda^1\|^{2}_{S_0}\big)\Big]\|\varPsi_\varLambda(t)\|_{S_0}^2
\leqslant 0.$$
The conclusion follows with Gr\"onswall's inequality, keeping in mind that inequalities \eqref{eq:estimC0} and \eqref{eq:estimC1} hold 
for $\psi^1_\varLambda$ as well. The proof is now completed.
\end{proof}
%
%
Definition~\ref{defi:Navier_Stokes} and Theorem~\ref{THEO:EXIST_NS} can easily be rephrased in terms of the vorticity field. The nonlinear advection term is 
defined for every $\bar\omega\in \bar V_0$ and $\varphi\in\mathfrak H^2_K$ as an element of $V_{-2}$ by:
$$\varLambda^V_1(\bar\omega,\varphi)=\Delta_{-2}\varLambda^S_1(\bar\Delta_0^{-1}\bar\omega,\varphi)=-
\langle \varLambda^S_1(\bar\Delta_0^{-1}\bar\omega,\varphi),\mathsf Q_2\cdot\rangle_{S_{-1},S_1},$$
the latter identity being deduced from \eqref{alternQk_2}. 
%
\begin{definition}
\label{defi:Navier_Stokes_vorticity}
Let a positive real number $T$,   a source term $f_V\in L^2(0,T;V_{-2})$, an initial data $\omega^{\rm i}\in  V_{-1}$ 
  and a triple $(g_n,g_\tau,\varGamma)\in G_0(T)$ be given. Define 
  $\omega^{\rm i}_0=\omega^{\rm i}-\mathsf L_{-1}^V(g_n(0),g_\tau(0),\varGamma(0))$ and assume that $\Sigma$ is of class $C^{3,1}$.
\par
We say that a vorticity  function $\omega\in  \bar V_{-1}(T)$ is a (weak)
solution to the $\omega-$Navier-Stokes equations satisfying the Dirichlet boundary conditions on $\Sigma_T$ as described in \eqref{def:boundary_cond} 
by the triple  $(g_n,g_\tau,\varGamma)$ if $\omega=\omega_b+\omega_\ell+\omega_\varLambda$ where:
\begin{enumerate}
\item The function $\omega_b$ accounts for the boundary conditions. It belongs to $V^b_{-1}(T)$ (defined in\eqref{def_mathfrakVT} )  and is equal to $\mathsf L_{0}^V(g_n,g_\tau,\varGamma)$;
\item The function $\omega_\ell$ accounts for the source term and the initial condition. It is defined as the unique solution in  $V_{-1}(T)$ 
of the homogeneous (linear) $\omega-$Stokes Cauchy problem
\begin{subequations}
\label{cauchy_for_psi_L_omega}
\begin{alignat}{3}
\partial_t \omega_\ell+\nu {\mathsf A}^V_{0}\omega_\ell&=-\partial_t\omega_b+f_V&\qquad&\text{in }\F_T,\\
\omega_\ell(0)&=\omega^{\rm i}_0&&\text{in }\F.
\end{alignat}
\end{subequations}
\item The function $\omega_\varLambda$ accounts for the nonlinear advection term. It belongs to the space $V_{-1}(T)$ and solves the nonlinear Cauchy problem:
\begin{subequations}
\label{cauchy_for_spi_lambda_omega}
\begin{alignat}{3}
\label{eq1:cauchy_for_spi_lambda_omega}
\partial_t \omega_\varLambda+\nu {\mathsf A}^V_{0}\omega_\varLambda&=-\varLambda^V_1(\omega_b+\omega_\ell+\omega_\varLambda,\varphi)&\qquad&\text{in }\F_T,\\
\omega_\varLambda(0)&=0&&\text{in }\F,
\end{alignat}
\end{subequations}
where $\varphi=\mathsf L_1^n g_n$ is the Kirchhoff potential that belongs to $\mathfrak H^1_K(T)$.
\end{enumerate}
\end{definition}
%
\begin{theorem}
With any set of data as described in Definition~\ref{defi:Navier_Stokes_vorticity}, there exists a unique solution 
$\omega$ to the $\omega-$Navier-stokes equations.
Moreover, if $\psi$ is the unique solution to the $\psi-$Navier-Stokes equations as defined in Definition~\ref{defi:Navier_Stokes} and 
$\omega^{\rm i}=\Delta_{-1}\psi^{\rm i}$, $f_V=\Delta_{-2}f_S$, all the other data being equal, then $\omega=\bar\Delta_{0}\psi$.
\end{theorem}
%
\begin{proof}
It suffices to apply the operator $\Delta_{-2}$ to \eqref{cauchy_for_psi_L} and \eqref{cauchy_for_spi_lambda} to obtain \eqref{cauchy_for_psi_L_omega} and \eqref{cauchy_for_spi_lambda_omega}, because, according to the commutative diagram of Fig.~\ref{diag_2}, $\Delta_{-2}\mathsf A^S_1=\mathsf A^V_0\Delta_0$.
\end{proof}
In case of homogeneous boundary conditions, we recover the exponential decay estimates as stated in Lemma~\ref{LEM:exp_decay} for the $\psi-$Stokes Cauchy problem, 
namely:
\begin{cor}
\label{energy_decay_weak}
Assume that $\psi$ is solution to the $\psi-$Navier-Stokes equations with homogeneous boundary conditions (i.e. $\psi_b=0$ in Definition~\ref{defi:Navier_Stokes}). 
Then  the exponential decay \eqref{eq:expo_decr_stream} holds true with $k=0$. 
\par
Similarly, if $\omega$  is solution to the $\omega-$Navier-Stokes equations with homogeneous boundary conditions (i.e. $\omega_b=0$ in Definition~\ref{defi:Navier_Stokes_vorticity}), 
then  the exponential decay \eqref{eq:expo_decr} holds true with $k=-1$. 
\end{cor}
\begin{proof}
In case of homogeneous boundary conditions, $\psi$ solves the Cauchy problem:
\begin{alignat*}{3}
\partial_t \psi+\nu {\mathsf A}^S_{1}\psi&=-\varLambda^S_1(\psi,0)+f_S&\qquad&\text{in }\F_T,\\
\psi(0)&=\psi^{\rm i}&&\text{in }\F.
\end{alignat*}
It suffices to form the duality pairing with $\psi\in S_1$, notice that $\langle\varLambda^S_1(\psi,0),\psi\rangle_{S_{-1},S_1}=0$ and apply Gr\"onwall's inequality 
to complete the proof.
\end{proof}
\begin{rem}
In  Definition~\ref{defi:Navier_Stokes_vorticity}, the initial condition $\omega^{\rm i}$ can be taken in the dual space $V_{-1}$. This
space contains, for every $1<p<2$:
$$L^p_V=\{(\omega,\mathsf Q_1\cdot)_{L^2(\F)}\,:\,\omega\in L^p(\Omega)\},$$
which can be identified with $L^p(\F)$. The space $V_{-1}$ contains also, for every Lipschitz curve $\mathscr C$ included in $\F$ and 
for every $q\in H^{-\frac12}(\mathscr C)$ what 
can be identified as a vorticity filament:
$$\omega_q:\theta\in V_1\mapsto \int_{\mathscr C} q\,\mathsf Q_1\theta\ds.$$
\end{rem}
\subsection{Strong solutions}
\begin{definition}
\label{defi:Navier_Stokes_strong}
Let a positive real number $T$,   a source term $f_S\in L^2(0,T;S_{0})$, an initial data $\psi^{\rm i}\in  \bar S_1$ 
  and a triple $(g_n,g_\tau,\varGamma)\in G_1(T)$ be given. Define 
  $\psi^{\rm i}_0=\psi^{\rm i}-\mathsf L_1^S(g_n(0),g_\tau(0),\varGamma(0))$ and assume that $\Sigma$ is of class $C^{2,1}$ and that the compatibility 
  condition:
  \begin{equation}
  \psi^{\rm i}_0\in S_1\quad\text{ or equivalently   that}\quad\frac{\partial\psi^{\rm i}_0}{\partial n}\bigg|_\Sigma=0,
  \end{equation}
  is satisfied.
\par
We say that a stream  function $\psi\in  \bar S_1(T)$ is a strong (or Kato)
solution to the $\psi-$Navier-Stokes equations satisfying the Dirichlet boundary conditions on $\Sigma_T$ as described in \eqref{def:boundary_cond} 
by the triple  $(g_n,g_\tau,\varGamma)$ if $\psi=\psi_b+\psi_{\ell}+\psi_\varLambda$ where:
\begin{enumerate}
\item The function $\psi_b$ accounts for the boundary conditions. It belongs to $S^b_1(T)$ and is equal to $\mathsf L_2^S(g_n,g_\tau,\varGamma)$;
\item The function $\psi_{\ell}$ accounts for the source term and the initial condition. It is defined as the unique solution in  $S_1(T)$ 
of the homogeneous (linear) $\psi-$Stokes Cauchy problem
\begin{subequations}
\label{cauchy_for_psi_L_strong}
\begin{alignat}{3}
\partial_t \psi_{\ell}+\nu {\mathsf A}^S_{2}\psi_{\ell}&=-\partial_t\psi_b+f_S&\qquad&\text{in }\F_T,\\
\psi_{\ell}(0)&=\psi^{\rm i}_0&&\text{in }\F.
\end{alignat}
\end{subequations}
\item The function $\psi_\varLambda$ accounts for the nonlinear advection term. It belongs to the space $S_1(T)$ and solves the nonlinear Cauchy problem:
\begin{subequations}
\label{cauchy_for_spi_lambda_strong}
\begin{alignat}{3}
\label{eq1:cauchy_for_spi_lambda_n}
\partial_t \psi_\varLambda+\nu {\mathsf A}^S_{2}\psi_\varLambda&=-\varLambda^S_0(\psi_b+\psi_{\ell}+\psi_\varLambda,\varphi)&\qquad&\text{in }\F_T,\\
\psi_\varLambda(0)&=0&&\text{in }\F,
\end{alignat}
\end{subequations}
where $\varphi=\mathsf L_2^n g_n$ is the Kirchhoff potential that belongs to $\mathfrak H^2_K(T)$.
\end{enumerate}
\end{definition}
\begin{theorem}
\label{THEO:EXIST_NS_STRONG}
With any set of data as described in Definition~\ref{defi:Navier_Stokes_strong}, there exists a unique (strong) solution in $\bar S_1(T)$ to the $\psi-$Navier-Stokes equations.
\end{theorem}
%
\begin{proof}[Proof of Theorem~\ref{THEO:EXIST_NS_STRONG}] The proof is based on a fixed point argument.
The existence and uniqueness of $\psi_b$ and $\psi_{\ell}$ being granted,  denote by $\bar\psi$ the sum $\psi_b+\psi_{\ell}$ that belongs to $\bar S_1(T)$ and introduce the constant:
$$R_0=\Big[\|\psi^{\rm i}_0\|^2_{S_1}+\|f_S\|^2_{L^2(0,T;S_{0})}+\|(g_n,g_\tau,\varGamma)\|^2_{G_1(T)}\Big]^{{\frac12}}.$$

Then, define three maps:
\begin{enumerate}
\item $\mathsf X_T:L^2(0,T;S_0)\longrightarrow S_1(T)$ where, 
for every $f\in L^2(0,T;S_0)$, $\theta=\mathsf X_Tf$ is the unique solution   in $S_1(T)$ to the Cauchy problem:
\begin{alignat*}{3}
\label{eq1:cauchy_for_spi_lambda_2}
\partial_t \theta+\nu {\mathsf A}^S_{2}\theta&=f&\qquad&\text{in }\F_T,\\
\theta(0)&=0&&\text{in }\F.
\end{alignat*}
\item $\mathsf Y_T:  S_1(T)\longrightarrow L^2(0,T;S_0)$ where, for every $\theta\in S_1(T)$,  $\mathsf Y_T(\theta)=-\varLambda^S_0(\bar\psi+\theta,\varphi)$ 
(remind that the Kirchhoff potential $\varphi=\mathsf L_2^n g_n$ is  given).
\item $\mathsf Z_T=\mathsf Y_T\circ\mathsf X_T:L^2(0,T;S_0)\longrightarrow L^2(0,T;S_0)$.
\end{enumerate}
Combining the estimates \eqref{estim_varlambda_strong_2}, \eqref{estim_varlambda_strong_3} and \eqref{estim:ini}, we deduce that for every $f,f_1,f_2\in 
L^2(0,T;S_0)$:
\begin{align*}
\|\mathsf Z_T f\|_{L^2(0,T,S_0)}&\leqslant \mathbf c_{[\F,\nu]} T^{\frac{1}{10}}\Big(R_0^2+
\|f\|^2_{L^2(0,T;S_0)}\Big),\\
\|\mathsf Z_T f_2-\mathsf Z_T f_1\|_{L^2(0,T,S_0)}&\leqslant\mathbf c_{[\F,\nu]} T^{\frac{1}{10}}\Big(R_0+\|f_1\|_{L^2(0,T;S_0)}+\|f_2\|_{L^2(0,T;S_0)}\Big)
\|f_2-f_1\|_{L^2(0,T,S_0)}.
\end{align*}
Let now $R$ be equal to $2R_0$. Then, for every $f,f_1,f_2$ in $B_R$, the ball of center $0$ and radius $R$ in $L^2(0,T;S_0)$:
\begin{align*}
\|\mathsf Z_T f\|_{L^2(0,T,S_0)}&\leqslant \mathbf c_{[\F,\nu]} T^{\frac{1}{10}} R^2,\\
\|\mathsf Z_T f_2-\mathsf Z_T f_1\|_{L^2(0,T,S_0)}&\leqslant\mathbf c_{[\F,\nu]} T^{\frac{1}{10}} R
\|f_2-f_1\|_{L^2(0,T,S_0)}.
\end{align*}
Thus, for $T_0=1/(4\mathbf c_{[\F,\nu]}R_0)^{10}$, the mapping $\mathsf Z_{T_0}$ is a contraction from $B_R$ into itself. 
Banach fixed point Theorem asserts that   $\mathsf Z_{T_0}$ admits a unique fixed point whose image by $\mathsf X_{T_0}$ 
yields a solution $\psi_\varLambda$ to the Cauchy problem \eqref{cauchy_for_spi_lambda_strong} on $(0,T_0)$.  Let $(0,T^\ast)$ be the larger time interval to which the 
solution $\psi=\psi_\varLambda+\psi_\ell+\psi_b$ can be extended. The time of existence $T_0$ depending only on $R_0$, standard arguments ensure that the following alternative holds:
\begin{equation}
\label{occurence}
\text{Either }\quad T^\ast=T\qquad\text{or}\qquad\lim_{t\to T^\ast}\|\psi(t)\|_{\bar S_1}=+\infty.
\end{equation}
As being a weak solution, estimates \eqref{eq:estimC0} and \eqref{eq:estimC1} hold for $\psi_\varLambda$, namely
$\|\psi_\varLambda\|_{\mathcal C(0,T;S_0)}$ and 
$\|\psi_\varLambda\|_{L^2(0,T;S_1)}$ are bounded.
On the other hand, forming the scalar product of equation \eqref{eq1:cauchy_for_spi_lambda_n} with $\mathsf A^S_2 \psi_\varLambda$ in $S_0$, 
we obtain:
$$
\frac{1}{2}\frac{\rm d}{\rm dt}\|\psi_\varLambda\|_{S_1}^2+\nu\|\psi_\varLambda\|_{S_2}^2
=
-(\varLambda_0(\psi,\varphi),\mathsf A^S_2 \psi_\varLambda)_{S_0}\qquad\text{on }(0,T^\ast).
$$
Considering the nonlinear term in the right hand side, H\"older's inequality yields:
\begin{equation}
|(\varLambda_0(\psi,\varphi),\mathsf A^S_2 \psi_\varLambda)_{S_0}|\\
\leqslant 
\|\Delta \psi\|_{L^4(\F)}\|\nabla^\perp\psi+\nabla\varphi\|_{\mathbf L^4(\F)}
\|\psi_\varLambda\|_{S_2}.
\end{equation}
whence we deduce, proceeding as in the proof of Lemma~\ref{estim_non_linear_term}, that:
$$|(\varLambda_0\psi,\mathsf A^S_2 \psi_\varLambda)_{S_0}|
\leqslant \Big[\frac{\nu}{2}+\mathbf c_{[\F,\nu]}\varTheta_1\Big]\|\psi_\varLambda\|_{S_2}^2+\mathbf c_{[\F,\nu]}\varTheta_2\qquad\text{on }(0,T^\ast),$$
with $\varTheta_1=\big[\|\psi_\varLambda\|_{S_1}^2+\|\psi_b\|_{\bar S_1}^2+\|\psi_\ell\|_{S_1}^2+\|\varphi\|_{\mathfrak H^2_K}^2\big]$ and 
$\varTheta_2=\big[\|\psi_\varLambda\|_{S_0}^2+\|\psi_b\|_{S_0}^2+\|\psi_\ell\|_{ S_0}^2+\|\varphi\|_{\mathfrak H^1_K}^2\big]+
\big[\|\psi_b\|_{\bar S_2}^2+\|\psi_\ell\|_{S_2}^2\big]$.
The functions $\varTheta_1$ and $\varTheta_2$ both belong to $L^1(0,T)$. It follows that:
$$\frac{\rm d}{\rm dt}\|\psi_\varLambda\|_{S_1}^2+\big(\nu-\mathbf c_{[\F,\nu]}\varTheta_1\big)\|\psi_\varLambda\|_{S_2}^2\leqslant 
\mathbf c_{[\F,\nu]}\varTheta_2\qquad\text{on }(0,T^\ast),$$
and by Gr\"onwall's inequality, we conclude that $\|\psi_\varLambda\|_{S_1}$ is bounded on $[0,T^\ast)$. The latter occurence in \eqref{occurence} may not 
happen and therefore $T^\ast=T$.
\end{proof}
%
%
Definition~\ref{defi:Navier_Stokes_strong} and Theorem~\ref{THEO:EXIST_NS_STRONG} can easily be rephrased in terms of the vorticity field. 
The nonlinear advection term is 
defined for every $\bar\omega\in \bar V_1$ and $\varphi\in\mathfrak H^2_K$ as an element of $V_{-1}$ by:
\begin{equation}
\label{def_lambda_0V}
\varLambda^V_0(\bar\omega,\varphi)=\Delta_{-1}\varLambda^S_0(\bar\Delta_1^{-1}\bar\omega,\varphi)=-
(\omega(\nabla^\perp\bar\psi+\nabla\varphi),\nabla\mathsf Q_1\cdot)_{\mathbf L^2(\F)},
\end{equation}
the latter identity being deduced from \eqref{alternQk_2}. In \eqref{def_lambda_0V}, $\omega$   stands for the {\it regular} part of 
$\bar\omega$ (see Remark~\ref{rem_singular_V1}) and $\bar\psi=\bar\Delta_{1}^{-1}\bar\omega$. 
%
\begin{rem}
\label{rem_advection}
Notice that even at this level of regularity, the nonlinear term 
of the vorticity equation cannot be written in the most common form, namely as the advection term $u\cdot\nabla\omega$ (see Section~\ref{SEC:more_regular}).
\end{rem}
%
%
\begin{definition}
\label{defi:Navier_Stokes_strong_vorticity}
Let a positive real number $T$,   a source term $f_V\in L^2(0,T;V_{-1})$, an initial data $\omega^{\rm i}\in  \bar V_0$ 
  and a triple $(g_n,g_\tau,\varGamma)\in G_1(T)$ be given. Define 
  $\omega^{\rm i}_0=\omega^{\rm i}-\mathsf L_0^V(g_n(0),g_\tau(0),\varGamma(0))$ and assume that $\Sigma$ is of class $C^{2,1}$ and that the compatibility 
  condition
 $ \omega^{\rm i}_0\in V_0$
  is satisfied.
\par
We say that a vorticity  function $\omega\in  \bar V_0(T)$ is a strong (or Kato)
solution to the $\omega-$Navier-Stokes equations satisfying the Dirichlet boundary conditions on $\Sigma_T$ as described in \eqref{def:boundary_cond} 
by the triple  $(g_n,g_\tau,\varGamma)$ if $\omega=\omega_b+\omega_\ell+\omega_\varLambda$ where:
\begin{enumerate}
\item The function $\omega_b$ accounts for the boundary conditions. It lies in $V^b_0(T)$ and is equal to $\mathsf L_1^V(g_n,g_\tau,\varGamma)$;
\item The function $\omega_\ell$ accounts for the source term and the initial condition. It is defined as the unique solution in  $V_0(T)$ 
of the homogeneous (linear) $\omega-$Stokes Cauchy problem
\begin{subequations}
\label{cauchy_for_psi_L_strong_v}
\begin{alignat}{3}
\label{cauchy_for_psi_L_strong_v_edp}
\partial_t \omega_\ell+\nu {\mathsf A}^V_{1}\omega_\ell&=-\partial_t\omega_b+f_V&\qquad&\text{in }\F_T,\\
\omega_\ell(0)&=\omega^{\rm i}_0&&\text{in }\F.
\end{alignat}
\end{subequations}
\item The function $\omega_\varLambda$ accounts for the nonlinear advection term. It belongs to the space $V_0(T)$ and solves the nonlinear Cauchy problem:
\begin{subequations}
\label{cauchy_for_spi_lambda_strong_V}
\begin{alignat}{3}
\label{eq1:cauchy_for_spi_lambda_V}
\partial_t \omega_\varLambda+\nu {\mathsf A}^V_{1}\omega_\varLambda&=-\varLambda^V_0(\omega_b+\omega_\ell+\omega_\varLambda,\varphi)&\qquad&\text{in }\F_T,\\
\omega_\varLambda(0)&=0&&\text{in }\F,
\end{alignat}
\end{subequations}
where $\varphi=\mathsf L_2^n g_n$ is the Kirchhoff potential that belongs to $\mathfrak H^2_K(T)$.
\end{enumerate}
\end{definition}
The counterpart of Theorem~\ref{THEO:EXIST_NS_STRONG} reads:
\begin{theorem}
\label{THEO:EXIST_NS_STRONG_V}
For any set of data as described in Definition~\ref{defi:Navier_Stokes_strong_vorticity}, there exists a unique (strong) solution in $\bar V_0(T)$ to the $\omega-$Navier-Stokes equations. Moreover, if $\psi$ is the unique solution to the $\psi-$Navier-Stokes equations as defined in Definition~\ref{defi:Navier_Stokes_strong} and 
$\omega^{\rm i}=\bar\Delta_{0}\psi^{\rm i}$, $f_V=\Delta_{-1}f_S$, all the other data being equal, then $\omega=\bar\Delta_{1}\psi$.
\end{theorem}
%
%
Once again, we point out that Equations \eqref{cauchy_for_psi_L_strong_v_edp} and \eqref{eq1:cauchy_for_spi_lambda_V} 
are set in $L^2(0,T;V_{-1})$ where $V_{-1}$ is not a distribution space. As very well explained in \cite{Simon:2010aa}, 
this may be the cause of  numerus mistakes and misunderstandings. Inspired by Guermond and Quartapelle in \cite{Guermond:1997aa}, let us elaborate 
a ``distribution-based'' reformulation of Systems \eqref{cauchy_for_psi_L_strong_v}-\eqref{cauchy_for_spi_lambda_strong_V}. Any solution 
$\omega$ to the $\omega-$NS equations can be decomposed into:
\begin{subequations}
\label{eq:material}
\begin{equation}
\label{eq:material_1}
\omega=\omega_\varLambda+\omega_\ell+\omega_{b}\qquad\text{where}\qquad\omega_{b}=\omega_{ b}^{\mathfrak H}
+\zeta_b\quad\text{ with }\quad\zeta_b=\sum_{j=1}^N\varGamma_j\zeta_j.
\end{equation}
In these sums, $\omega_\varLambda$ and $\omega_\ell$ belong to $V_0(T)$, $\omega_b^{\mathfrak H}$ is in 
$H^1(0,T;V_{-1})\cap \mathcal C([0,T],L^2_V)\cap L^2(0,T;H^1_V)$ and $\varGamma_j\in H^1(0,T)$ for every $j=1,\ldots,N$. In 
the splitting \eqref{eq:material_1} $\zeta_b$ is identified as the singular part of $\omega$ while the ``regular part'' is:
\begin{equation}
\omega_{\mathsf r}=\omega_\varLambda+\omega_\ell+\omega_{ b}^{\mathfrak H}.
\end{equation}
Recalling the decomposition 
\eqref{decom_S0_1} of the space $S_0$, namely: 
$$
S_0=H^1_0(\F)\sumperp \mathbb F_S,
$$
where the finite dimensional space $\mathbb F_S$  is spanned by the functions $\xi_j$ ($j=1,\ldots,N$), we deduce that:
$$
V_1=\mathsf P_1H^1_0(\F)\sumperp \mathbb F_V.$$
From any source term $f_V\in L^2(0,T;V_{-1})$, we define $f_V^{\mathsf r}\in L^2(0,T;H^{-1}(\F))$ by setting:
\begin{equation}
f_V^{\mathsf r}=\langle f_V,\mathsf P_1\cdot\rangle_{V_{-1},V_1}.
\end{equation}
\end{subequations}
%
\begin{theorem}
Let $\omega$ be a solution to the $\omega-$NS equations as described in Definition~\ref{defi:Navier_Stokes_strong_vorticity} 
and introduce $\omega_{\mathsf r}$, $\zeta_b$ and $f_V^{\mathsf r}$ as explained in the relations \eqref{eq:material}. 
Then $\omega_{\mathsf r}$ obeys the equation:
\begin{subequations}
\label{eq:distrib_based}
\begin{equation}
\label{eq:distrib_based_1}
\partial_t\omega_{\mathsf r}-\nu\Delta\omega_{\mathsf r}+\nabla\cdot\big[\omega_{\mathsf r}(\nabla^\perp\psi+\nabla\varphi)\big]=f_V^{\mathsf r}
\qquad\text{in }\quad L^2(0,T;H^{-1}(\F)),
\end{equation}
and for every $j=1,\ldots,N$:
\begin{equation}
\label{eq:distrib_based_2}
\varGamma_j'+\nu(\nabla\omega_{\mathsf r},\nabla\xi_j)_{\mathbf L^2(\F)}-(\omega_{\mathsf r}(\nabla^\perp\psi+\nabla\varphi),\nabla\xi_j)_{\mathbf L^2(\F)}=\langle f_V,\mathsf P_1\xi_j\rangle_{V_{-1},V_1}\quad \text{in}\quad L^2(0,T).
\end{equation}
\end{subequations}
\end{theorem}
%
\begin{proof}
By definition of a strong solution to the $\omega-$NS equations, the following equality holds for every $\theta\in V_1$:
$$\frac{\rm d}{\rm dt}\langle\omega_{\mathsf r},\theta\rangle_{V_{-1},V_1}+\sum_{j=1}^N\varGamma_j'\langle \zeta_j,\theta\rangle_{V_{-1},V_1}
+\nu(\omega_{\mathsf r},\theta)_{V_1}-(\omega_{\mathsf r}(\nabla^\perp\psi+\nabla\varphi),\nabla\mathsf Q_1\theta)_{\mathbf L^2(\F)}
=\langle f_V,\theta\rangle_{V_{-1},V_1}\quad\text{ on }(0,T).$$
Notice now that 
$$\langle\omega_{\mathsf r},\theta\rangle_{V_{-1},V_1}=(\omega_\varLambda+\omega_\ell,\theta)_{V_0}+(\omega_b^{\mathfrak H},
\mathsf Q_1\theta)_{L^2(\F)}=(\omega_{\mathsf r},\mathsf Q_1\theta)_{L^2(\F)}.$$
Choosing the test function $\theta$ in $\mathsf P_1H^1_0(\F)$, we obtain \eqref{eq:distrib_based_1} and choosing $\theta$ in $\mathbb F_V$ leads to \eqref{eq:distrib_based_2}.
\end{proof}
Appart from the nonlinear advection term, formulation \eqref{eq:distrib_based} is quite similar to System~\eqref{eq:main_vorti:1} displayed 
at the beginning of this paper. In Section~ \ref{SEC:more_regular}, we shall seek  more regular solutions to the $\omega-$NS system 
in order to obtain Identity \eqref{eq:distrib_based_1} satisfied in $L^2(0,T;L^2(\F))$.
\section{The pressure}
\label{SEC:pressure}
The purpose of this section is to explain how the pressure can be recovered from the stream function or the vorticity field, i.e. 
to derive Bernoulli-like formulas for the $\psi-$NS equations. In the literature, the existence of the pressure field is usually 
deduced from the Helmholtz-Weyl  decomposition and no expression is supplied.
%
\subsection{Hilbertian framework for the velocity field}
The following Lebesgue spaces shall enter  the definition of the pressure:
\begin{equation}
L^2_{\rm m}=\Big\{f\in L^2(\F)\,:\,\int_\F f\dx=0\Big\}
\qquad\text{and}\qquad 
{\mathfrak H}_{\rm m}=\mathfrak H\cap L^2_{\rm m},
\end{equation}
as well as the Sobolev spaces below:
\begin{equation}
{H}^1_{\rm m}=H^1(\F)\cap  L^2_{\rm m}
\qquad\text{and}\qquad 
  H^2_{\rm m}=\Big\{f\in H^2(\F)\cap  L^2_{\rm m}\,:\,\frac{\partial f}{\partial n}\Big|_\Sigma=0\Big\}.
 \end{equation}
The last two spaces are provided with the norms:
 $$(f_1,f_2)_{H^1_{\rm m}}=(\nabla f_1,\nabla f_2)_{\mathbf L^2(\F)}\forallt f_1,f_2\in H^1_{\rm m},$$
 and
 $$(f_1,f_2)_{H^2_{\rm m}}=(\Delta f_1,\Delta f_2)_{L^2(\F)}\forallt f_1,f_2\in H^2_{\rm m}.$$
 We recall that the lifting operators $\mathsf L^\tau_k$ (for every integer $k$) were introduced in Definition~\ref{def_stuff}.
\begin{definition}
For every $f\in L^2_{\rm m}$ we denote by $\varTheta_f$ the unique function in $ H^2_{\rm m}$ satisfying:
$$\Delta \varTheta_f=f\quad\text{in }\F,$$
and we denote by $\varPsi_f$ the unique preimage of $f$ in $ Z_2$ by the operator $\mathsf A^{ Z}_2$ (see Lemma~\ref{LEM:S2A2}). %
\par
Then, we define the operator $\mathsf H: L^2_{\rm m}\longrightarrow  L^2_{\rm m}$   by: 
$$\mathsf Hf=\Delta\mathsf L_{1}^\tau \frac{\partial\varTheta_f}{\partial\tau}\Big|_\Sigma\qquad\forallt f\in  L^2_{\rm m}.$$
It is worth noticing the obvious equality:
\begin{equation}
\varPsi_{\mathsf Hf}=\mathsf L_{1}^\tau \frac{\partial\varTheta_f}{\partial\tau}\Big|_\Sigma\qquad\forallt f\in  L^2_{\rm m}.
\end{equation}
\end{definition}
The operator $\mathsf H$ will come in handy for defining the pressure from the stream function. The main properties of $\mathsf H$ are 
summarized in the following lemma:
\begin{lemma}
\label{prop:H}
The operator $\mathsf H$ is bounded, $\Im \mathsf H= {\mathfrak H}_{\rm m}$ and $\ker \mathsf H=V_0$, what entails
that $\mathsf H$ is an isomorphism from $ {\mathfrak H}_{\rm m}$ 
onto $ {\mathfrak H}_{\rm m}$. 
\par
Denoting classically by $\mathsf H^\ast$ the adjoint of $\mathsf H$, we deduce that $\Im\mathsf H^\ast= {\mathfrak H}_{\rm m}$, 
$\ker \mathsf H^\ast=V_0$ and $\mathsf H^\ast$ is an isomorphism from $ {\mathfrak H}_{\rm m}$ onto itself. Furthermore, for every 
$f\in  {\mathfrak H}_{\rm m}$, the function $\mathsf H^\ast f$ is the harmonic conjugate of $f$ i.e. the unique function in $\mathfrak H_{\rm m}$ 
such that the complex function 
$$z=(x_1+ix_2)\longmapsto f(x_1,x_2)+i\,(\mathsf H^\ast f)(x_1,x_2)$$ is holomorphic 
in $\F$.
\end{lemma}
\begin{proof}
The boundedness results from elliptic regularity results for $\varTheta_f$  and from the boundedness of the operator $\mathsf L_{1}^\tau$ (we recall 
that by default $\Sigma$ is assumed to be at least of class $\mathcal C^{1,1}$). By construction, 
$\mathsf H$ is valued in $ {\mathfrak H}_{\rm m}$ so let a function $h$ be given in $ {\mathfrak H}_{\rm m}$.
According to Lemma~\ref{decomp_S2}, $\varPsi_h$ belongs to $\mathfrak B_S$ and:
$$\int_{\Sigma^+}\frac{\partial\varPsi_h}{\partial n}\ds=\int_\F h\dx-\sum_{j=1}^N \int_{\Sigma^-_j}\frac{\partial\varPsi_h}{\partial n}\ds=0.$$
We can then define $g\in H^{\frac32}(\Sigma)$ such that $\partial g/\partial\tau=\partial \varPsi_h/\partial n$ on $\Sigma$. We denote now by $\theta_h$ the biharmonic 
function in $H^2(\F)$ such that $\partial\theta_h/\partial n=0$  and $\theta_h=g$ on $\Sigma$. One easily verifies that $\Delta\theta_h$ belongs to 
$ L^2_{\rm m}$ and $\mathsf H\Delta\theta_h=h$. This proves that $\Im \mathsf H= {\mathfrak H}_{\rm m}$.
\par
According to Lemma~\ref{reg_Lk}, the operator $\mathsf L_{1}^\tau$ is an isomorphism from $G^\tau_1$ onto $\mathfrak B_S$. Therefore, 
if $\mathsf Hf=0$ for some $f$, then $\varTheta_f$ is in $S_1$ and hence $f=\Delta\varTheta_f$ is in $V_0$, which means that indeed $\ker \mathsf H=V_0$.
\par
Let now $h$ be in ${\mathfrak H}_{\rm m}$ and $f$ be in $ L^2_{\rm m}$. Then:
$$(h,\mathsf Hf)_{L^2(\F)}=(h,\Delta\varPsi_{\mathsf Hf})_{L^2(\F)}=\int_\Sigma h\frac{\partial\varPsi_{\mathsf Hf}}{\partial n}\ds=
\int_\Sigma h\frac{\partial\varTheta_f}{\partial \tau}\ds=-\int_\Sigma  \frac{\partial h}{\partial \tau} \varTheta_f\ds.$$
Introducing $\bar h$ the harmonic conjugate of $h$, we deduce that:
$$(h,\mathsf Hf)_{L^2(\F)}=-\int_\Sigma  \frac{\partial \bar h}{\partial n} \varTheta_f\ds=(\bar h,\Delta\varTheta_f)_{L^2(\F)}=(\bar h,f)_{L^2(\F)},$$
and the proof is  completed.
\end{proof}
%
%
We turn now our attention to the Gelfand triple:
\begin{equation}
\label{gelf_H10}
\mathbf H_1\subset \mathbf H_0\subset \mathbf H_{-1},
\end{equation}
where $\mathbf H_1=\mathbf H^1_0(\F)$, $\mathbf H_0=\mathbf L^2(\F)$ is the pivot space and $\mathbf H_{-1}=\mathbf H^{-1}(\F)$ is the dual space 
of $\mathbf H_1$.
The space $\mathbf H_1$ is provided with its usual scalar product, namely:
\begin{equation}
\label{class_scalar}
(u,v)_{\mathbf H_1}=\int_\F\nabla u:\nabla v\dx\forallt u,v\in\mathbf H_1.
\end{equation}
\begin{theorem}
\label{theo:decompH10}
For every $u\in \mathbf H_1$ there exists a unique triple $(\psi,\phi,h)\in S_1\times S_1\times  {\mathfrak H}_{\rm m}$ such that $h$ is the harmonic Bergman
projection of the divergence of $u$ and 
\begin{equation}
\label{eq:decompH10}
u=\nabla^\perp\psi+\nabla^\perp\varPsi_{\mathsf Hh}+\nabla\varTheta_h+\nabla\phi\qquad\text{in }\F.
\end{equation}
It follows that the divergence and the curl of $u$ are given respectively by:
\begin{equation}
\label{eq:div_curl}
\nabla\cdot u = \Delta \phi+h\qquad\text{and}\qquad \nabla^\perp\cdot u = \Delta \psi+\mathsf Hh\qquad\text{in }\F,
\end{equation}
and these decompositions in $L^2(\F)$ of $\nabla\cdot u$ and $\nabla^\perp\cdot u$ agrees with the orthogonal decomposition $V_0\sumperp\mathfrak H$ 
of the space $L^2(\F)$.%
\par
Finally, let $u_1,u_2$ be in $\mathbf H_1$ and denote $\psi_1,\psi_2,\phi_1,\phi_2$ the functions in $S_1$ and $h_1,h_2$ the functions in $ {\mathfrak H}_{\rm m}$ such that:
\begin{equation}
\label{decomp:uk}
u_k=\nabla^\perp\psi_k+\nabla^\perp\varPsi_{\mathsf H h_k}+\nabla\varTheta_{h_k}+\nabla\phi_k\quad\text{in }\F\qquad(k=1,2).
\end{equation}
Then, the  scalar product \eqref{class_scalar} can be expanded as follows:
\begin{align}
\nonumber
(u_1,u_2)_{\mathbf H_1}&=(\nabla \cdot u_1,\nabla \cdot u_2)_{L^2(\F)}+(\nabla^\perp \cdot u_1,\nabla^\perp \cdot u_2)_{L^2(\F)}\\
\label{decomp_scal}
&=
(\Delta\psi_1,\Delta\psi_2)_{L^2(\F)}+(\mathsf Hh_1,\mathsf Hh_2)_{L^2(\F)}+(h_1,h_2)_{L^2(\F)}+(\Delta\phi_1,\Delta\phi_2)_{L^2(\F)}.
\end{align}
\end{theorem}
\begin{proof}
 Let $u$ be given and  decompose the $L^2$ functions $\nabla^\perp\cdot u$ and $\nabla\cdot u$ respectively into the sums $\omega+\omega_h$ and $\delta+h$ 
 with $\omega,\delta\in V_0$ and $\omega_h,h\in\mathfrak H$. Since:
 $$\int_\F h\dx=\int_\F (\delta+h)\dx=\int_\Sigma u\cdot n\ds = 0,$$
 the harmonic function $h$ is actually in $ {\mathfrak H}_{\rm m}$. In the same way:
 $$\int_\F \omega_h\dx=\int_\F (\omega+\omega_h)\dx=-\int_\Sigma u\cdot\tau\ds=0,$$
 and $\omega_h$ is in $ {\mathfrak H}_{\rm m}$ as well. 
 Define now 
 $\phi$ and $\psi$ in $S_1$ such that $\Delta\phi=\delta$ and  $\Delta\psi=\omega$. One easily verifies that the vector field:
 $$v=u-\big[\nabla^\perp\psi+\nabla^\perp\varPsi_{\mathsf Hh}+\nabla\varTheta_h+\nabla\phi\big]
 \qquad\text{in }\F,$$
 is in $\mathbf H_1$ and that $\nabla\cdot v=0$. On the other hand $\nabla^\perp \cdot v =\omega_h-\mathsf H h$, which means in particular 
 that $\nabla^\perp \cdot v \in\mathfrak H_{\rm m}$. This entails that $v=0$. Indeed, according to 
 Helmholtz-Weyl  decomposition  (see \cite[Theorem 3.2]{Girault:1986aa}),
 there exists $\Phi\in H^1_{\rm m}$ and $\Psi\in S_0$ such that:
 $$v=\nabla^\perp\Psi+\nabla\Phi\qquad\text{in }\F.$$
 But $\Phi=0$ since $\nabla\cdot v=0$ and $\Psi$ belongs to $S_1$ according to the boundary conditions and the regularity of $v$. It follows that 
 $\Delta\Psi\in V_0$ but as observed earlier, 
 $\Delta\Psi=\nabla^\perp\cdot v\in  {\mathfrak H}_{\rm m}$, what implies that $\Psi=0$. This proves the existence and uniqueness of the decomposition \eqref{eq:decompH10}.
 \par
 Assume now that $u_1$ and $u_2$ are in ${\boldsymbol{\mathcal D}}(\F)=\mathcal D(\F;\mathbb R^2)$ and introduce their decompositions 
 as in \eqref{decomp:uk}.  %
Integrating by parts, we obtain:
\begin{align}
\nonumber
(u_1,u_2)_{\mathbf H_1}&=-(\Delta u_1,u_2)_{\mathbf L^2(\F)}\\
\label{eq:second_expand}
&= -\big\langle\nabla^\perp\Delta\psi_1+\nabla^\perp\mathsf Hh_1,u_2\big\rangle_{{\boldsymbol{\mathcal D}}'(\F),{\boldsymbol{\mathcal D}}(\F)}
 -\big\langle\nabla h_1+\nabla \Delta\phi_1,u_2
 \big\rangle_{{\boldsymbol{\mathcal D}}'(\F),{\boldsymbol{\mathcal D}}(\F)}.
 \end{align}
 We switch to the duality pairing in the second equality because although $u_1$ is smooth, this does not guaranty that every term in the decomposition \eqref{eq:decompH10}
 is also smooth (notice that invoking elliptic regularity results would require the boundary $\Sigma$ to be smoother than $\mathcal C^{1,1}$). 
 The former term in the right hand side of \eqref{eq:second_expand} yields:
\begin{align*}
\big\langle\nabla^\perp\Delta\psi_1+\nabla^\perp\mathsf Hh_1,u_2\big\rangle_{{\boldsymbol{\mathcal D}}'(\F),{\boldsymbol{\mathcal D}}(\F)}&=(\Delta\psi_1+\mathsf Hh_1,\nabla^\perp\cdot u_2)_{L^2(\F)}\\
& =-(\Delta\psi_1,\Delta\psi_2)_{L^2(\F)}-(\mathsf Hh_1,\mathsf H h_2)_{L^2(\F)},
 \end{align*}
while  the latter leads to:
$$\big\langle\nabla h_1+\nabla \Delta\phi_1,u_2
 \big\rangle_{{\boldsymbol{\mathcal D}}'(\F),{\boldsymbol{\mathcal D}}(\F)}=-( h_1+ \Delta\phi_1,\nabla\cdot u_2)_{\mathbf L^2(\F)}=-(h_1,h_2)_{L^2(\F)}-(\Delta\phi_1,\Delta\phi_2)_{L^2(\F)}.$$
 The equality \eqref{decomp_scal} follows by density of ${\boldsymbol{\mathcal D}}(\F)$ into $\mathbf H_1$ and the proof is complete.
\end{proof}
\begin{rem}
\begin{enumerate}
\item
The decomposition \eqref{eq:decompH10} differs from  the one in \cite[Theorem 3.3]{Girault:1986aa} where $u\in\mathbf H_1$ is decomposed into 
\begin{equation}
\label{decomp:nul}
u=\nabla^\perp\psi+(-\Delta_D)^{-1}\nabla p\qquad\text{in }\F,
\end{equation}
with $\psi\in S_1$ and a potential $p$ in $L^2_{\rm m}$. The operator $(-\Delta_D)^{-1}$ obviously stands for the inverse of the Laplacian operator 
with homogeneous boundary conditions.  The stream function $\psi$ is the same  in \eqref{eq:decompH10} and \eqref{decomp:nul}. 
\item In \cite[Theorem 3]{Auchmuty:2001aa} or \cite[Theorem 2.1]{Kozono:2009aa},  every vector field $u\in \mathbf L^2(\F)$ is shown to admit the decomposition:
\begin{equation}
u=\nabla^\perp\psi+\nabla^\perp h+\nabla p\qquad\text{in }\F,
\end{equation}
with $\psi\in H^1_0(\F)$, $p\in H^1_{\rm m}$ and $h\in \mathbb F_S$. This expression is used by Maekawa in \cite{Maekawa:2013aa} to derive necessary 
and sufficient conditions for $u$ to be in $\mathbf J_1$.
%
\item Identity \eqref{decomp_scal} is a trivial version of Friedrich's second inequalities; see \cite[Lemma 2.5 and Remark 2.7]{Girault:1986aa} and 
also for instance \cite{Krizek:1984aa}. However, it can also be readily deduced from \eqref{decomp_scal} that there exists a constant $\mathbf c_\F$ such 
that for every $u\in\mathbf H_1$:
$$\|u\|_{\mathbf H_1} \leqslant \mathbf c_\F\|\mathsf P^\perp \omega\|_{L^2(\F)} +\|\delta\|_{L^2(\F)}
\qquad\text{ or }\qquad
\|u\|_{\mathbf H_1}\leqslant\mathbf c_\F \| \omega\|^2_{L^2(\F)}+\|\mathsf P^\perp \delta\|_{L^2(\F)},$$
where $\omega= \nabla^\perp\cdot u$ and $\delta= \nabla\cdot u$. It means that the $\mathbf H_1$-norm of a vector field is controlled by:
\begin{enumerate}
\item Either the harmonic Bergman 
projection of the curl and the divergence of this vector field in $L^2$;
\item or by the curl and the harmonic Bergman projection of the divergence of this vector field in 
$L^2$.
\end{enumerate}
We were not able to find this result in the literature. 
\item The decomposition \eqref{eq:decompH10} of Theorem~\ref{theo:decompH10} allows to deduce a necessary and suffisant condition for 
the following overdetermined div-curl problem to be well-posed: There exists a unique $u$ in $\mathbf H_1$ such that:
$$\nabla\cdot u=\delta\qquad\text{and}\qquad \nabla^\perp\cdot u=\omega\qquad\text{ in }\F,$$
with $\delta$ and $\omega$ in $L^2(\F)$ if and only  if $\mathsf P^\perp\omega$ is the harmonic conjugate of $\mathsf P^\perp\delta$. This result 
seems to be new as well.
\end{enumerate}
\end{rem}
%
%
As shown in \cite{Simon:1999aa}, the definition of the pressure 
for the Navier-Stokes equations (in classical velocity-pressure formulation)
is not possible for a source term $f_{\mathbf J}$ in $L^2(0,T;\mathbf J_{-1})$, what means in nonprimitive variables, for $f_S\in L^2(0,T;S_{-1})$ (see Fig.~\ref{diag_3}). 
The definition of the pressure requires the source term  to be in $L^2(0,T;\mathbf H_{-1})$. To be more specific,  we need 
to elaborate on the structure of the 
dual space $\mathbf H_{-1}$.
%
\begin{prop}
\label{strct_H_1}
For every linear form $f_{\mathbf H}$ in $\mathbf H_{-1}$, there exists a unique 
pair $(\delta_{\mathbf H},\omega_{\mathbf H})\in L^2_{\rm m}\times V_0$ such 
that:
$$\langle f_{\mathbf H},u\rangle_{\mathbf H_{-1},\mathbf H_1}=(\delta_{\mathbf H},\nabla\cdot u)_{L^2(\F)}
+(\omega_{\mathbf H},\nabla^\perp\cdot u)_{L^2(\F)}\forallt u\in\mathbf H_1.$$
If $f_{\mathbf H}$ belongs to $\mathbf H_0$, then:
\begin{subequations}
\begin{equation}
\label{eq:hodge}
\delta_{\mathbf H}=-\phi_{\mathbf H}-\mathsf H^\ast \psi_{\mathbf H}
\quad\text{ and }\quad \omega_{\mathbf H}=-\mathsf P_1 \psi_{\mathbf H},
\end{equation}
where $f_{\mathbf H}=\nabla\phi_{\mathbf H}+\nabla^\perp \psi_{\mathbf H}$ is the Helmholtz-Weyl decomposition of $f_{\mathbf H}$. 
The identities \eqref{eq:hodge} can  easily be inverted:
\begin{equation}
\phi_{\mathbf H}=-\delta_{\mathbf H}+\mathsf H^\ast \omega_{\mathbf H}
\quad\text{ and }\quad
\psi_{\mathbf H}=-\mathsf Q_1\omega_{\mathbf H}.
\end{equation}
\end{subequations}
\end{prop}
\begin{proof}
Let $f_{\mathbf H}$ be in $\mathbf H_{-1}$. According to Riesz representation Theorem and Theorem~\ref{theo:decompH10}, there exists 
$\phi,\psi\in S_1$ and $h\in\mathfrak H_{\rm m}$ such that:
$$\langle f_{\mathbf H},u\rangle_{\mathbf H_{-1},\mathbf H_1}=(\Delta\phi+({\rm Id}+\mathsf H^\ast\mathsf H)h,\nabla\cdot u)_{L^2(\F)}
+(\Delta\psi,\nabla^\perp\cdot u)_{L^2(\F)}\forallt u\in\mathbf H_1.$$
It suffices to set $\delta_{\mathbf H}=\Delta\phi+({\rm Id}+\mathsf H^\ast\mathsf H)h$ and $\omega_{\mathbf H}=\Delta\psi$.
\par
Assume now that $f_{\mathbf H}$ lies in $\mathbf H_0$ and denote by $\phi_{\mathbf H}\in H^1_{\rm m}$ and $\psi_{\mathbf H}\in S_0$ the functions entering the
Helmholtz-Weyl decomposition of $f_{\mathbf H}$, i.e.
$$f_{\mathbf H}=\nabla\phi_{\mathbf H}+\nabla^\perp\psi_{\mathbf H}\qquad\text{in }\F.$$
By definition of $\mathbf H_0$ as   pivot space:
$$\langle f_{\mathbf H},u\rangle_{\mathbf H_{-1},\mathbf H_1}=(f_{\mathbf H},u)_{\mathbf L^2(\F)}=
-(\phi_{\mathbf H},\nabla\cdot u)_{L^2(\F)}-(\psi_{\mathbf H},\nabla^\perp\cdot u)_{L^2(\F)}\forallt u\in\mathbf H_1.$$
The orthogonal decomposition $L^2_{\rm m}(\F)=V_0\sumperp \mathfrak H_{\rm m}$ leads to  $\nabla^\perp\cdot u=\mathsf P_0^\perp
\nabla^\perp\cdot u+\mathsf P_0\nabla^\perp\cdot u$. But according to Theorem~\ref{theo:decompH10}, $\mathsf P_0^\perp
\nabla^\perp\cdot u=\mathsf H\nabla\cdot u$, whence we deduce that:
$$\langle f_{\mathbf H},u\rangle_{\mathbf H_{-1},\mathbf H_1}=-(\phi_{\mathbf H}+\mathsf H^\ast\psi_{\mathbf H},\nabla\cdot u)_{L^2(\F)}
-(\mathsf P_1\psi_{\mathbf H},\nabla^\perp\cdot u)_{L^2(\F)},$$
and the proof is complete.
\end{proof}
\begin{rem}
\label{link_fSfH}
It is now easy to verify that if $f_{\mathbf H}\in\mathbf H_{-1}$ and $f_S\in S_{-1}$ are two linear forms such that:
$$\langle f_{\mathbf H},\nabla^\perp\psi\rangle_{\mathbf H_{-1},\mathbf H_1}=\langle f_S,\psi\rangle_{S_{-1},S_1}
\forallt \psi\in S_1,$$
then $\mathsf P  f_S=\omega_{\mathbf H}$ where $\omega_{\mathbf H}\in V_0$ and $\delta_{\mathbf H}\in L^2_{\rm m}$ are defined from $f_{\mathbf H}$ in Proposition~\ref{strct_H_1}.
We shall prove that the pressure depends only upon $\delta_{\mathbf H}$ and therefore is actually independent of the source term $f_S$.
\end{rem}
\subsection{Weak solutions}

When the equation is nonlinear, the operator $\mathsf H$ is not suffisant   to define the pressure. Thus,
for every $u\in \mathbf L^4(\F)$, define $\pi_\varTheta[u],\pi_\varPsi[u]\in  L^2_{\rm m}$ by means of Riesz representation Theorem as:
\begin{subequations}
\label{def_non_lin_pi}
\begin{align}
(\pi_\varTheta[u],f)_{ L^2_{\rm m}}&=-(D^2\varTheta_fu,u)_{\mathbf L^2(\F)}\\
(\pi_\varPsi[u],f)_{ L^2_{\rm m}}&=(D^2\varPsi_fu,u^\perp)_{\mathbf L^2(\F)}\forallt f\in L^2_{\rm m}.
\end{align}
\end{subequations}
%
%
%
%
\begin{definition}
\label{def:pressure_weak}
Let $T$ be a positive real number, $\psi$ be a function in $\bar S_0(T)$, $\varphi$ be  in $\mathfrak H^1_K(T)$ (this space is defined in \eqref{def:HKT}) 
and $\delta_{\mathbf H}$ be in $L^2(0,T;L^2_{\rm m})$. Then 
introduce the velocity field $u=\nabla^\perp\psi+\nabla\varphi$ and  for a.e. $t\in (0,T)$ define $p_{\mathsf r}(t)$ by:
\begin{subequations}
\label{def:bernoulli}
\begin{equation}
p_{\mathsf r}(t)=-\partial_t\varphi(t)+\pi_\varTheta[u(t)]+\mathsf H^\ast\Big[\nu \omega(t)-\pi_\varPsi[u(t)]\Big]-\delta_{\mathbf H}(t),
\end{equation}
where $\omega(t)=\Delta\psi(t)$ (i.e. $\omega(t)$ is the regular part of $\bar\Delta_0\psi(t)$, see Remark~\ref{rem_singular_V0}). 
The pressure $p$ corresponding to these data is obtained by summing $p_{\mathsf r}$, called the regular part of 
the pressure, and a singular part $p_{\mathsf s}$:
\begin{equation}
\label{eq:express_p}
p= p_{\mathsf r} +p_{\mathsf s}\qquad\text{ with }\qquad p_{\mathsf s} =-\partial_t\mathsf H^\ast\psi.
\end{equation}
\end{subequations}
\end{definition}
%
The proof of the lemma below is obvious:
\begin{lemma}
\label{prop:reg_pressure}
The function $p_{\mathsf r}$ belongs to $L^2(0,T; L^2_{\rm m})$ and the mapping 
$$(\psi,\varphi,\delta_{\mathbf H})\in \bar S_0(T)\times \mathfrak H_k^1(T)\times L^2(0,T;L^2_{\rm m})
\mapsto p_{\mathsf r}\in L^2(0,T;L^2_{\rm m}),$$ 
is continuous. 
The function $p_{\mathsf s}$ lies in $W^{-1,\infty}(0,T; L^2_{\rm m})$ and the mapping $\psi\in \bar S_0(T)\mapsto p_{\mathsf s}\in W^{-1,\infty}(0,T; L^2_{\rm m})$ 
is continuous.
\end{lemma}
We can now state the main result of this subsection:
%
\begin{theorem}
\label{equiv_NS_prim_nonprim}
Let  $T$ be a positive real number and let $\psi\in \bar S_0(T)$ be a weak solution to the $\psi-$Navier-Stokes equations as defined in Definition~\ref{defi:Navier_Stokes}, and whose source term is recalled to be denoted by $f_S$. Let $\varphi\in\mathfrak H_K^1(T)$ be the   Kirchhoff potential  also introduced in   Definition~\ref{defi:Navier_Stokes}.  Finally, let $f_{\mathbf H}$ be in $L^2(0,T;\mathbf H_{-1})$ such that (see Remark~\ref{link_fSfH}):
$$\langle f_{\mathbf H},\nabla^\perp\theta\rangle_{\mathbf H_{-1},\mathbf H_1}=\langle f_S,\theta\rangle_{S_{-1},S_1}
\forallt \theta\in S_1.$$
According to Proposition~\ref{strct_H_1}, to  the linear form $f_{\mathbf H}$  can be associated a pair $(\delta_{\mathbf H},\omega_{\mathbf H})
\in L^2(0,T;L^2_{\rm m})\times L^2(0,T;V_0)$. 
\par
Denote now by $u$ the vector field $\nabla\varphi+\nabla^\perp\psi$ and by $p$ the pressure defined 
from $\psi$, $\varphi$ and $\delta_{\mathbf H}$ as explained in Definition~\ref{def:pressure_weak}.
Then the pair $(u,p)$ is a weak (Leray) solution to the  Navier-Stokes equations, namely, for every $w$ in $\mathbf H_1$:
\begin{equation}
\label{eq:NS_prim}
\frac{\rm d}{{\rm d}t}(u,w)_{\mathbf L^2(\F)}-(\nabla w u,u)_{\mathbf L^2(\F)}+\nu\int_\F \nabla u:\nabla w\dx-(p,\nabla\cdot w)_{L^2(\F)}
=\langle f_{\mathbf H},w\rangle_{\mathbf H_{-1},\mathbf H_1}\quad\text{on }(0,T).
\end{equation}
\end{theorem}
\begin{proof}
Remind that $\psi$ satisfies, for every $\theta\in S_1$:
\begin{equation}
\label{eq:grosseNS}
\frac{\rm d}{{\rm d}t}(\nabla\psi,\nabla\theta)_{\mathbf L^2(\F)}+\nu(\Delta\psi,\Delta\theta)_{L^2(\F)}
+(D^2\theta u,u^\perp)_{\mathbf L^2(\F)}=\langle f_{\mathbf H},\nabla^\perp\theta\rangle_{\mathbf H_{-1},\mathbf H_1}
\quad\text{on }(0,T).
\end{equation}
Let $u_b=\nabla^\perp\psi_b+\nabla\varphi$, $u_0=u-u_b=\nabla^\perp\psi_\ell+\nabla^\perp\psi_\varLambda$ (see Definition~\ref{defi:Navier_Stokes_strong}) and $\omega_b=\Delta\psi_b$.  
Then, for every $w\in\boldsymbol{\mathcal D}(\F)$:
\begin{subequations}
\label{divdiv}
\begin{equation}
\int_\F\nabla u_b:\nabla w\dx=-\langle \Delta u_b,w\rangle_{\boldsymbol{\mathcal D}'(\F),\boldsymbol{\mathcal D}(\F)}
=-\langle\nabla^\perp \omega_b,w\rangle_{\boldsymbol{\mathcal D}'(\F),\boldsymbol{\mathcal D}(\F)}=
(\omega_b,\nabla^\perp\cdot w)_{L^2(\F)},
\end{equation}
and this result extends by density to every $w\in \mathbf H_1$. According to Theorem~\ref{theo:decompH10}, we can decompose $w$ 
into
\begin{equation}
\label{decomp_w}
w=\nabla^\perp\theta+\nabla^\perp\varPsi_{\mathsf Hh}+\nabla\varTheta_h+\nabla\phi\qquad\text{in }\F,
\end{equation}
with $(\theta,\phi,h)\in S_1\times S_1\times  {\mathfrak H}_{\rm m}$ and it follows that $\nabla^\perp\cdot w=\Delta\theta+\mathsf Hh$. Since 
$\omega_b\in\mathfrak H_V$, we infer that:
\begin{equation}
(\omega_b,\nabla^\perp\cdot w)_{L^2(\F)}=(\omega_b,\mathsf Hh)_{L^2(\F)}=(\mathsf H^\ast\omega_b,h)_{L^2(\F)}=
(\mathsf H^\ast\omega,\nabla\cdot w)_{L^2(\F)},
\end{equation}
because $\mathsf H^\ast \omega=\mathsf H^\ast\omega_b$.
On the other hand, since $u_0$ belongs to $\mathbf H_1$, according to 
\eqref{decomp_scal}, it follows that:
\begin{equation}
\int_\F \nabla u_0:\nabla w\dx=(\Delta(\psi_\varLambda+\psi_\ell),\Delta\theta)_{L^2(\F)}=(\Delta\psi,\Delta \theta)_{L^2(\F)},
\end{equation}
\end{subequations}
the latter equality resulting from the orthogonality property $(\Delta\psi_b,\Delta\theta)_{L^2(\F)}=0$. Gathering now the identities \eqref{divdiv}, 
we obtain:
\begin{equation}
\label{nutruc}
\nu\int_\F\nabla u:\nabla w\dx=\nu(\Delta\psi,\Delta \theta)_{L^2(\F)}+\nu(\mathsf H^\ast\omega,\nabla\cdot w)_{L^2(\F)}.
\end{equation}
Invoking again the decomposition \eqref{decomp_w}, we get:
\begin{equation}
\label{time_deriv}
(u,w)_{\mathbf L^2(\F)}=(\nabla\psi,\nabla\theta)_{\mathbf L^2(\F)}-(\mathsf H^\ast\psi+\varphi,\nabla\cdot w)_{L^2(\F)},
\end{equation}
and also:
$$(\nabla wu,u)_{\mathbf L^2(\F)}=(D^2\theta u,u^\perp)_{\mathbf L^2(\F)}
+(D^2\varPsi_{\mathsf Hh}u,u^\perp)_{\mathbf L^2(\F)}+(D^2(\varTheta_h+\phi)u,u)_{\mathbf L^2(\F)}.$$
But notice that $\mathsf Hh=\mathsf H(\nabla\cdot w)$ and $\varTheta_h+\phi=\varTheta_{\nabla\cdot w}$ (both functions share the same 
boundary conditions and the same Laplacian). 
Using the notation \eqref{def_non_lin_pi}, we are then allowed to rewrite the above equality 
as:
\begin{equation}
\label{againNL}
(\nabla wu,u)_{\mathbf L^2(\F)}=(D^2\theta u,u^\perp)_{\mathbf L^2(\F)}+(\mathsf H^\ast\pi_\Psi[u],\nabla\cdot w)_{L^2(\F)}
-(\pi_\varTheta[u],\nabla\cdot w)_{L^2(\F)}.
\end{equation}
Finally, considering the source term:
\begin{equation}
\label{source_term}
\langle f_{\mathbf H},w\rangle_{\mathbf H_{-1},\mathbf H_1}=\langle f_{\mathbf H},\nabla^\perp\theta
\rangle_{\mathbf H_{-1},\mathbf H_1}+\langle f_{\mathbf H},w-\nabla^\perp\theta
\rangle_{\mathbf H_{-1},\mathbf H_1}=\langle f_S,\theta\rangle_{S_{-1},S_1}+(\delta_{\mathbf H},\nabla\cdot w)_{L^2(\F)}.
\end{equation}
Summing now the time derivative of \eqref{time_deriv} with \eqref{nutruc} and subtracting \eqref{againNL} and the term 
$(p,\nabla\cdot w)_{L^2(\F)}$, we obtain \eqref{source_term}, taking into account \eqref{eq:grosseNS}. The resulting equality 
is therefore \eqref{eq:NS_prim} and the proof is completed.
 \end{proof}
%
\subsection{Strong solutions}
For every $u\in \mathbf H^2(\F)$, define $\Phi[u]$ as the unique element in $H_{\rm m}^1$ such that:
\begin{equation}
\label{def_phiu}
(\Phi[u],\theta)_{H^1_{\rm m}}=-(\nabla uu,\nabla\theta)_{\mathbf L^2(\F)}\forallt \theta\in H^1_{\rm m}.
\end{equation}
%
\begin{definition}
\label{pressure_strong}
Let $T$ be a positive time, $\psi$ be a function in $\bar S_1(T)$, $\varphi$ be  in $\mathfrak H^2_K(T)$ and $\phi_{\mathbf H}$ (accounting 
for the source term) be in $L^2(0,T;H^1_{\rm m})$. Then 
introduce the velocity field $u=\nabla^\perp\psi+\nabla\varphi$. For a.e. $t\in (0,T)$ define $p(t)$ by:
\begin{equation}
\label{p_strong}
p(t)=-\partial_t\varphi(t)+\Phi[u(t)]+\nu\mathsf H^\ast \mathsf Q^\perp_1\omega(t)+\phi_{\mathbf H}(t),
\end{equation}
where $\omega(t)=\Delta\psi(t)$ (i.e. $\omega(t)$ is the regular part of $\bar\Delta_1\psi(t)$, see Remark~\ref{rem_singular_V1}).
\end{definition}
%
\begin{prop}
\label{p_weak_strong}
The function $p$ belongs to $L^2(0,T; H^1_{\rm m})$ and   the mapping 
$$(\psi,\varphi,\phi_{\mathbf H})\in \bar S_1(T)\times \mathfrak H_k^2(T)\times L^2(0,T;S_0)
\mapsto p\in L^2(0,T;H^1_{\rm m}),$$ 
is continuous. 
\par
Moreover, if $\psi$ is a solution to the $\psi-$NS equations as described in Definition~\ref{defi:Navier_Stokes_strong} with Kirchhoff potential $\varphi$ 
and source term $f_S\in L^2(0,T;S_0)$, then $p$ defined in\eqref{p_strong} from the triple $(\psi,\varphi,\phi_{\mathbf H})$ 
is equal to the pressure of Definition~\ref{def:pressure_weak} computed from the triple $(\psi,\varphi,\delta_{\mathbf H})$ with 
$\delta_{\mathbf H}=-\phi_{\mathbf H}-\mathsf H^\ast f_S$.
\end{prop}
%
\begin{proof}
The continuity of the mapping being obvious, let us verify the claim that $p$ in \eqref{p_strong} 
matches the expression given in \eqref{eq:express_p}. For a.e. $t\in(0,T)$, $\partial_t\psi$ is in $S_0$ and since 
$\psi$ is a strong solution to the $\psi-$NS equations it follows that:
$$\partial_t \mathsf H^\ast\psi=\mathsf H^\ast\partial_t\psi=\mathsf H^\ast \big(-\nu\mathsf A^S_2(\psi_\varLambda+\psi_{\ell})-\varLambda_0^S(\psi,\varphi)
+f_S\big).$$
On the one hand, according to the expression \eqref{expressAS2} of $\mathsf A^S_2$:
$$\mathsf H^\ast \mathsf A^S_2(\psi_\varLambda+\psi_{\ell})=-\mathsf H^\ast \mathsf Q_1(\omega_\varLambda+\omega_\ell)
=-\mathsf H^\ast\mathsf P_1^\perp \mathsf Q_1(\omega_\varLambda+\omega_\ell)=
\mathsf H^\ast \mathsf Q_1^\perp (\omega_\varLambda+\omega_\ell),$$
because $\mathsf H^\ast=\mathsf H^\ast\mathsf P^\perp$ and $\mathsf P_1^\perp\mathsf Q_1=({\rm Id}-\mathsf P_1)\mathsf Q_1=-\mathsf Q_1^\perp$.
On the other hand, for every $f$ in $L^2_{\rm m}$:
$$(\mathsf H^\ast\varLambda_0^S(\psi,\varphi),f)_{L^2(\F)}=(\varLambda_0^S(\psi,\varphi),\Delta\Psi_{\mathsf Hf})_{L^2(\F)}=
-(\varLambda_0^S(\psi,\varphi), \Psi_{\mathsf Hf})_{S_0}=-(\nabla u u, \nabla^\perp\Psi_{\mathsf Hf})_{\mathbf L^2(\F)},$$
the latter equality resulting from \eqref{alter_def_lambda0}. Summing up, we obtain that for every $f\in L^2_{\rm m}$ and a.e. $t\in (0,T)$:
\begin{multline}
\label{grosse_formule}
\big(\pi_\varTheta[u(t)]+\mathsf H^\ast\big[\nu \omega(t)-\pi_\varPsi[u(t)]
\big]-\partial_t\mathsf H^\ast\psi(t),f\big)_{L^2_{\rm m}}=
-(D^2\varTheta_f(t)u(t),u(t))_{\mathbf L^2(\F)}\\
-(D^2\varPsi_{\mathsf Hf}(t)u(t),u^\perp(t))_{\mathbf L^2(\F)}
-(\nabla u u, \nabla^\perp\Psi_{\mathsf Hf})_{\mathbf L^2(\F)}+\nu(\mathsf H^\ast\omega(t),f)_{L^2_{\rm m}}
+\nu(\mathsf H^\ast \mathsf Q_1^\perp (\omega_\varLambda+\omega_\ell)(t),f)_{L^2_{\rm m}}-(\mathsf H^\ast f_S,f)_{L^2_{\rm m}}.
\end{multline}
The two first terms in the right hand side  can be rewritten as:
$$(D^2\varTheta_f(t)u(t),u(t))_{\mathbf L^2(\F)}+
(D^2\varPsi_{\mathsf Hf}(t)u(t),u^\perp(t))_{\mathbf L^2(\F)}=(\nabla(\nabla\varTheta_f+\nabla^\perp\varPsi_{\mathsf Hf})(t)u(t),u(t))_{\mathbf L^2(\F)},$$
and the resulting quantity can now be integrated by parts:
$$(\nabla(\nabla\varTheta_f+\nabla^\perp\varPsi_{\mathsf Hf})(t)u(t),u(t))_{\mathbf L^2(\F)}=-(\nabla u(t) u(t),\nabla\varTheta_f(t)+\nabla^\perp\varPsi_{\mathsf Hf}(t))_{\mathbf L^2(\F)}.$$
Turning our attention to the two last terms in the right hand side of \eqref{grosse_formule}, we observe that 
$\mathsf H^\ast\omega(t)=\mathsf H^\ast\omega_b(t)=\mathsf H^\ast\mathsf Q_1^\perp \omega_b(t)$ according to the properties of $\mathsf H^\ast$ stated in Lemma~\ref{prop:H} 
and the fact that $\omega_b$ is the harmonic part of $\omega=\omega_\varLambda+\omega_\ell+\omega_b$. We have now proved that:
$$\big(\pi_\varTheta[u(t)]+\mathsf H^\ast\big[\nu \omega(t)-\pi_\varPsi[u(t)]
\big]-\partial_t\mathsf H^\ast\psi(t),f\big)_{L^2_{\rm m}}=
(\nabla u(t) u(t),\nabla\varTheta_f(t))_{\mathbf L^2(\F)}+\nu(\mathsf H^\ast \mathsf Q_1^\perp \omega(t),f)_{L^2_{\rm m}}
-(\mathsf H^\ast f_S,f)_{L^2_{\rm m}}.$$
Recalling the definition \eqref{def_phiu} of $\Phi[u(t)]$, we can integrate by parts the first term in the right hand side:
$$(\nabla u(t) u(t),\nabla\varTheta_f(t))_{\mathbf L^2(\F)}=-(\nabla\Phi[u(t)],\nabla \varTheta_f(t))_{\mathbf L^2(\F)}=
(\Phi[u(t)],f)_{L^2_{\rm m}}, $$
and thus complete the proof.
\end{proof}
\begin{theorem}
Let  $T$ be a positive real number and let $\psi\in \bar S_1(T)$ be a strong solution to the $\psi-$Navier-Stokes equations as defined in Definition~\ref{defi:Navier_Stokes_strong}. Let $\varphi\in\mathfrak H_K^2(T)$ be the   Kirchhoff potential  also introduced in   Definition~\ref{defi:Navier_Stokes_strong}.  
\par
Denote now by $u$ the vector field $\nabla\varphi+\nabla^\perp\psi$ and by $p$ the pressure defined 
from $\psi$, $\varphi$ and $\phi_{\mathbf H}=-\mathsf H^\ast f_S$ as explained in Definition~\ref{pressure_strong}.
Then the pair $(u,p)$ is a strong (Kato) solution to the  Navier-Stokes equations.
\end{theorem}
\begin{proof}
The stream function $\psi$ is also a weak solution to the $\psi-$NS equations in the sense of Definition~\ref{defi:Navier_Stokes}. According to 
Proposition~\ref{p_weak_strong}, the pressure $p$  is a ``weak'' pressure in the sense of Definition~\ref{def:pressure_weak} as well. 
It follows from Theorem~\ref{equiv_NS_prim_nonprim} that $(u,p)$ is a weak solution to the NS equations in primitive variables.
From the regularity of $u$ and $p$ we are allowed to deduce that $(u,p)$ is indeed a strong solution to the NS equations in primitive variables.
\end{proof}
\section{More regular vorticity solutions}
\label{SEC:more_regular}
So far and even for strong solutions as described in the preceding subsection, the regularity of the functions
does not allow writing the vorticity equation in the most common form \eqref{eq:main_vorti:1}, that is, loosely speaking, as an advection-diffusion 
 equation set in $L^2(\F)$ (see Remark~\ref{rem_advection}).
To achieve this level of regularity, a first guess would be to seek solutions in $V_1(T)$, in which case, the operator $(-\nu\Delta)$ 
should be 
$\nu\mathsf A^V_2$. However, since the nonlinear advection term $u\cdot\nabla\omega$ does not belong to $V_0$ in general, we are inclined 
to conclude that this approach leads to a dead end. 
We shall prove that the solution should rather be looked for in the space $\bar V_1(T)$. We recall that 
the spaces $\bar V_2, \bar V_1$ and $\bar V_0$ are all of them subspaces of $V_{-1}$.  They are defined in Subsection~ \ref{SEC:Nonhomogeneous_vorticity}.
\par
Before addressing the $\omega-$NS equations, we begin as usual with Stokes problems. All along this Section, we assume that 
the boundary $\Sigma$ is of class $\mathcal C^{3,1}$. 
\subsection{Regular Stokes vorticity solutions}
For every positive real number $T$, we aim to define solutions to Stokes problems belonging to:
$$\bar V_1(T)=H^1(0,T;\bar V_0)\cap \mathcal C([0,T];\bar V_1)\cap L^2(0,T;\bar V_2).$$
We recall that $\bar V_2\subset \bar V_1\subset\bar V_0\subset V_{-1}$ (see   Subsection~\ref{SEC:Nonhomogeneous_vorticity}). 
So, let a triple $(g_n,g_\tau,\varGamma)$ be given in $G_2(T)$ and define  $\omega_b$ in $ V^{\rm b}_1(T)$ by:
\begin{equation}
\label{def_omega_B}
\omega_b(t)=\mathsf L^V_{2}(g_n(t),g_\tau(t),\varGamma(t))=\omega_b^{\mathfrak H}(t)+\sum_{j=1}^N\varGamma_j(t)\zeta_j\qquad\text{for a.e. }t\in (0,T),
\end{equation}
where $\omega_b^{\mathfrak H}(t)$ belongs to $\mathfrak H_V^2$ (the operator $\mathsf L^V_{2}$ is defined in \eqref{def_LV} and maps continuously 
$G_2(T)$ into the space $ V^{\rm b}_1(T)$ defined in \eqref{def_mathfrakVT}). 
The source term $f_V$ is expected to belong to $L^2(0,T;\bar V_0)$ and  the decomposition \eqref{decomp_V0} of the space 
$\bar V_0$ leads to the splitting: 
\begin{equation}
\label{decomp_fV}
f_V(t)=f_V^0(t)+f_V^{\mathfrak H}(t)+\sum_{j=1}^N\alpha_V^j(t)\zeta_j\qquad\text{for a.e. }t\in (0,T),
\end{equation}
with $f_0(t)$ is in $V_0$, $f_V^{\mathfrak H}(t)$ in 
$\mathfrak H_V$ and $\alpha_V^j(t)$ in $\mathbb R$ for every $j=1,\ldots,N$.
Similarly, taking now into account the decomposition \eqref{decomp_barV2} of $\bar V_2$, 
we seek the total vorticity in the form:
\begin{equation}
\label{def_omega_1}
\omega(t)=\omega_1(t)+\omega_b(t)
\quad\text{with}\quad
\omega_1(t)=\omega_0(t)+\omega_{\mathfrak B}(t)+\sum_{j=1}^N\beta_j(t)\varOmega_j
\qquad\text{for a.e. }t\in (0,T),
\end{equation}
where the function $\omega_0(t)$ belongs to $V_2$, $\omega_{\mathfrak B}(t)$ is in $\mathfrak B_V^2$ and $\beta_j(t)$  in $\mathbb R$
for every $j=1,\ldots,N$. 
They are the unknowns of the problem.  
The function   $\omega_1$ 
is supposed to satisfy in particular Equation \eqref{eq:stokes_omega} for $k=0$, namely:
\begin{equation}
\label{firdt_guess_vorti}
\partial_t \omega_1+\nu\mathsf A^V_1\omega_1=f_V-\partial_t\omega_b\qquad\text{in }\F_T,
\end{equation}
this equality being set in $L^2(0,T;V_{-1})$. We want this equation to be satisfy in the slightly more regular space
$L^2(0,T;\bar V_0)$. Thus, the operator $\mathsf A^V_1$ turns into the operator $\bar{\mathsf A}^V_2$ (defined right above Proposition~\ref{prop_bardelta2}). 
Keeping in mind the decomposition \eqref{decomp_V0} of $\bar V_0$ we apply successively to Equation \eqref{firdt_guess_vorti} the orthogonal projections onto the spaces 
$V_0$, $\mathfrak H_V$ and $\mathbb F_V^\ast$ respectively to obtain the system:
\vspace{-1mm}
\begin{subequations}
\label{main_vorti_regul_sys}
\begin{alignat}{3}
\label{main_vorti_re}
\partial_t\omega_0+\nu\mathsf A^V_2\omega_0&=f_V^0-\partial_t\omega_{\mathfrak B}-\sum_{j=1}^N\ \beta_j'\varOmega_j&\qquad&\text{in }\F_T\\[-3mm]
\label{eq:second_regukl}
\nu\bar{\mathsf A}^V_2\omega_{\mathfrak B}&=f_V^{\mathfrak H}-\partial_t\omega_b^{\mathfrak H}&&\text{in }\F_T\\
\label{eq:third_regukl}
\nu\beta_j&=\varGamma'_j-\alpha_V^j&&\text{in }(0,T)\quad\text{ for every }j=1,\ldots,N.
\end{alignat}
\end{subequations}
A solution can be  worked out by taking the time derivatives of the equations \eqref{eq:second_regukl} and \eqref{eq:third_regukl}. 
Thus, one gets the expressions of $\partial_t\omega_{\mathfrak B}$ and $\partial_t\beta_j$ that can be used in \eqref{main_vorti_re}. This leads us 
to the following statement:
%
\begin{prop}
\label{stokes_regul}
Let $T$ be a positive real number, $(g_n,g_\tau,\varGamma)$ be a triple in 
\begin{subequations}
\label{regul_data}
\begin{align}
G_{\mathsf r}(T)&=G_2(T)\cap \mathcal C^1([0,T];G_0^n\times G_0^\tau\times\mathbb R^N)\cap H^2(0,T;G^n_{-1}\times G^\tau_{-1}\times \mathbb R^N)\\
&=\big\{(g_n,g_\tau,\varGamma)\in G_2(T)\,:\, (\partial_t g_n,\partial_t g_\tau,\varGamma')\in G_0(T)\big\},
\end{align}
 and $f_V$ 
be a source term in 
\vspace{-3mm}
\begin{multline}
\label{def:FT}
F_{\mathsf r}(T)=\Big\{f_V\in L^2(0,T;\bar V_0)\,:\,f_V=f_V^0+f_V^{\mathfrak H}+\sum_{j=1}^N\alpha^j_V\zeta_j\text{ with }f_V^0\in L^2(0,T;V_0),\\[-2mm]
f_V^{\mathfrak H}\in L^2(0,T;\mathfrak H_V)\cap\mathcal C([0,T];V_{-1})\cap H^1(0,T;V_{-2})
\text{ and }\alpha_V^j\in H^1(0,T)\text{ for every }j=1,\ldots,N\Big\}.
\end{multline}
\end{subequations}

Let $\omega_b$ be defined from the boundary data as in \eqref{def_omega_B}, let $\omega^{\rm i}$ be an initial data in $\bar V_1$ satisfying the compatibility condition $\omega^{\rm i}-\omega_b(0)\in V_1$ and let:
\vspace{-3mm}
\begin{equation}
\label{compa_initial}
\omega^{\rm i}_0=\omega^{\rm i}-\omega_b(0)-\frac{1}{\nu}\big({\mathsf A}^V_1\big)^{-1}\big(f_V^{\mathfrak H}-\partial_t\omega_b^{\mathfrak H}\big)(0)-\frac{1}{\nu}\sum_{j=1}^N
\big(\varGamma'_j(0)
-\alpha_V^j(0)\big)\varOmega_j.
\end{equation}
Then $\omega^{\rm i}_0$ belongs to $V_1$ and there exists a unique solution $\omega_0\in V_1(T)$ to the Cauchy problem:
%
\begin{subequations}
\begin{alignat}{3}
\label{main_sys_regul}
\partial_t\omega_0+\nu\mathsf A^V_2\omega_0&=f_V^0-\frac{1}{\nu}\big({\mathsf A}^V_0\big)^{-1}\big(\partial_t f_V^{\mathfrak H}-\partial^2_t\omega_b^{\mathfrak H}\big)
-\frac{1}{\nu}\sum_{j=1}^N\big(\varGamma_j''- (\alpha_V^j)'\big)\varOmega_j&\quad&\text{ in }\F_T,\\[-3mm]
\omega_0(0)&=\omega_0^{\rm i}&&\text{ in }\F.
\end{alignat}
\end{subequations}
The vorticity function:
\vspace{-3mm}
\begin{equation}
\label{def_omega_regul_sol}
\omega=\omega_0+\frac{1}{\nu}\big(\bar{\mathsf A}^V_2\big)^{-1}\big(f_V^{\mathfrak H}-\partial_t\omega_b^{\mathfrak H}\big)+\frac{1}{\nu}\sum_{j=1}^N(\varGamma'_j
-\alpha_V^j)\varOmega_j+\omega_b,
\end{equation}
belongs to $\bar V_1(T)$ and solves System~\eqref{main_vorti_regul_sys}.  It will be called a regular vorticity  solution   to the 
$\omega-$Stokes equations.
\end{prop}
%
It is worth noticing that:
\begin{enumerate}
\item System \eqref{main_vorti_regul_sys} is no longer a simple parabolic 
system but rather a coupled parabolic-elliptic 
system.
\item The regularity assumptions \eqref{regul_data} entail that the functions $f_V^{\mathfrak H}$ and $\partial_t\omega_b^{\mathfrak H}$ both belong to 
$\mathcal C([0,T];V_{-1})$ and therefore that the equality \eqref{compa_initial} at the initial time makes sense.
\item Under the hypotheses of the Proposition, the function $\omega_1$ can be defined as in \eqref{def_omega_1}. One easily verifies that $\omega_1$ solves \eqref{firdt_guess_vorti}.
\item In the definition \eqref{def:FT} of the space $F_{\mathsf r}(T)$, the regularities of the harmonic and nonharmonic parts of the source term are different.
\end{enumerate}
\begin{proof}
The proof is straightforward: The right-hand side of equation \eqref{main_sys_regul} clearly belongs to $L^2(0,T;V_0)$ and hence it suffices 
to apply Proposition~\ref{exist:sol:stokes_omega}.
\end{proof}
%
From Proposition~\ref{exist:sol:stokes_omega} and Equality \eqref{def_omega_regul_sol}, we deduce:
\begin{cor}
The spaces $G_{\mathsf r}(T)$ and $F_{\mathsf r}(T)$ being equipped with their natural topologies, there exists a positive constant $\mathbf c_{[\F,\nu]}$ such that, 
for every regular vorticity solution to  a $\omega-$Stokes problem as defined in Proposition~\ref{stokes_regul}, the estimate below holds true:
\begin{equation}
\|\omega\|_{\bar V_1(T)}\leqslant \mathbf c_{[\F,\nu]}\big[\|\omega^{\rm i}\|_{V_1}^2+\|(g_n,g_\tau,\varGamma)\|_{G_{\mathsf r}(T)}^2+
\|f_V\|_{F_{\mathsf r}(T)}^2\big]^{{\frac12}}.
\end{equation}
\end{cor}
\subsection{Regular Navier-Stokes vorticity solutions}
\subsubsection*{Setting up the system of equations}
To begin with, let us recall the expression \eqref{def_lambda_0V} of the advection term in the NS equation (strong vorticity version). 
For every $\omega\in \bar V_1$ and every Kirchhoff potential $\varphi\in\mathfrak H_K^2$:
$$\varLambda^V_0(\omega,\varphi)=-(\omega_{\mathsf r}(\nabla^\perp\psi+\nabla\varphi),\nabla\mathsf Q_1\cdot)_{\mathbf L^2(\F)},$$
where $\omega_{\mathsf r}$ is the orthogonal projection of $\omega$ on $H^1_V$ (i.e. $\omega_{\mathsf r}$ is the regular part of $\omega$; 
see Remark~\ref{rem_singular_V1}) and the 
stream function $\psi=(-\bar\Delta_1)^{-1}\omega$ belongs to $\bar S_2$ (see Fig.~\ref{figV13}).
Assuming 
now more regularity, namely that $\omega$ is in $\bar V_2$ and $\varphi$ in $\mathfrak H^3_k$, an integration by parts yields:
$$\langle\varLambda^V_0(\omega,\varphi),\theta\rangle_{V_{-1},V_1}=
-\int_\Sigma \omega_{\mathsf r}\frac{\partial\varphi}{\partial n}\mathsf Q_1\theta\ds+\big((\nabla^\perp\psi+\nabla\varphi)
\cdot \nabla\omega_{\mathsf r},\mathsf Q_1\theta)_{L^2(\F)}\forallt \theta\in V_1,$$
which leads us to define:
%
\begin{definition}
For every vorticity $\omega$ in $\bar V_2$, we denote by  $\psi=\bar\Delta^{-1}_2\omega$ the corresponding stream function and by 
$(\omega_{\mathsf r},\mathsf Q_1\cdot)_{L^2(\F)}$  the orthogonal projection of $\omega$ on $H^2_V$ (i.e. $\omega_{\mathsf r}$ is the regular part of $\omega$). For every Kirchhoff potential $\varphi$ in $\mathfrak H^2_K$, 
we define:
\begin{subequations}
\begin{equation}
\label{def_lambda_R}
\gamma_j(\omega,\varphi)=-\int_{\Sigma_j^-} \omega_{\mathsf r}\frac{\partial\varphi}{\partial n}\ds\quad(j=1,\ldots,N)
\quad\text{and}\quad
\varLambda^V_{\mathsf r}(\omega,\varphi)=(\nabla^\perp\psi+\nabla\varphi)
\cdot \nabla\omega_{\mathsf r}\in L^2(\F),
\end{equation}
and the linear form on $\bar V_0$:
\vspace{-3mm}
\begin{equation}
\bar\varLambda_0^V(\omega,\varphi)=\sum_{j=1}^N\gamma_j(\omega,\varphi)\zeta_j+(\varLambda^V_{\mathsf r}(\omega,\varphi),\mathsf Q_1\cdot)_{L^2(\F)}.
\end{equation}
\end{subequations}
\end{definition}
%
%
Let boundary data $(g_n,g_\tau,\varGamma)$ and a source term $f_V$ be given as in the preceding subsection. Taking now 
into account the nonlinear advection
term, System~\eqref{main_vorti_regul_sys} can be rewritten as follows:
\vspace{-1mm}
\begin{subequations}
\label{main_NS_vorticity3}
\begin{alignat}{3}
\label{main_vorti_re_NS}
\partial_t\omega_0+\nu\mathsf A^V_2\omega_0&=f_V^0-\mathsf P\varLambda^V_{\mathsf r}(\omega,\varphi)-\partial_t\omega_{\mathfrak B}-\sum_{j=1}^N \beta_j'\varOmega_j&\qquad&\text{in }\F_T\\[-2mm]
\label{eq:second_regukl_NS}
\nu\bar{\mathsf A}^V_2\omega_{\mathfrak B}&=f_V^{\mathfrak H}-\mathsf P^\perp\varLambda_{\mathsf r}^V(\omega,\varphi)-\partial_t\omega_b^{\mathfrak H}&&\text{in }\F_T\\
\label{eq:third_regukl_NS}
\nu\beta_j&=\varGamma'_j-\alpha_V^j+\gamma_j(\omega,\varphi)&&\text{ in }(0,T)\text{ for every }j=1,\ldots,N.
\end{alignat}
\end{subequations}
This formulation allows recovering the formulation  \eqref{eq:main_vorti:1} given at the beginning of the paper and that can be rewritten with the notation 
of this Section:
\begin{subequations}
\label{eq:main_vorti:suite}
\begin{alignat}{3}
\label{eq:main:vorti:suite}
\partial_t\omega_{\mathsf r}+u\cdot\nabla\omega_{\mathsf r}-\nu\Delta\omega_{\mathsf r}&=f_V&\quad&\text{in }\F_T\\ 
\label{eq:main:flux:suite}
-\varGamma'_j+\int_{\Sigma^-_k}\omega_{\mathsf r}\,g_n\ds-\nu\int_{\Sigma_k^-}\frac{\partial\omega_{\mathsf r}}{\partial n}\ds&=-\alpha^j_V&&\text{on }(0,T),\quad j=1,\ldots,N,
\end{alignat}
\end{subequations}
with $u=\nabla^\perp\psi+\nabla\varphi$ and $\psi=\bar\Delta_2^{-1}\omega$.
Thus, decomposing $\omega$ in $\bar V_2$ as in \eqref{def_omega_1}, the regular part of the vorticity $\omega_{\mathsf r}$ is given by:
\vspace{-3mm}
$$\omega_{\mathsf r}=\omega_0+\omega_{\mathfrak B}+\sum_{j=1}^N\beta_j\varOmega_j+\omega_b^{\mathfrak H}\quad\text{and}\quad
\omega=\omega_{\mathsf r}+\sum_{j= 1}^N\varGamma_j\zeta_j.$$
Summing \eqref{main_vorti_re_NS} and \eqref{eq:second_regukl_NS} gives \eqref{eq:main:vorti:suite} and \eqref{eq:main:flux:suite} is a rephrasing 
of \eqref{eq:third_regukl_NS}. Indeed, since $\omega_0$, $\omega_{\mathfrak B}$ and $\omega_b^{\mathfrak H}$ have zero mean flux through the inner boundaries, 
we have  for every $k=1,\ldots,N$:
$$\int_{\Sigma^-_k}\frac{\partial\omega_{\mathsf r}}{\partial n}\ds=\sum_{j= 1}^N\beta_j\int_{\Sigma^-_k}\frac{\partial\varOmega_j}{\partial n}\ds=-\beta_k,$$
according to \eqref{def_omegaj} and the second point of Remark~\ref{rem_fluxA}.
\par
It is worthwhile comparing  also the formulations \eqref{main_NS_vorticity3} (or  equivalently \eqref{eq:main_vorti:suite}) with the results of Maekawa in \cite{Maekawa:2013aa}. Therein, focusing on Section 2, only homogeneous boundary condition are considered and no  function space is specified 
for the vorticity. The author claims that the vorticity has to satisfy the {\it integral condition}:
\begin{equation}
\label{maekawa}
\frac{\partial}{\partial n}(-\Delta_D)^{-1}\omega+\sum_{j=1}^N\big((-\Delta_D)^{-1}\nabla^\perp\omega,\nabla^\perp\tilde q_j\big)_{L^2}
\frac{\partial \tilde q_j}{\partial n}=0\quad\text{ on }\Sigma,\end{equation}
where the functions $\tilde q_j$ are a free family in $\mathbb F_S$ chosen in such a way that $\tilde q_j=c_j\delta_i^j$ on $\Sigma^-_i$ with $c_j$ 
a normalizing real constant ensuring that $\|\nabla\tilde q_j\|_{L^2(\F)}=1$. In equality \eqref{maekawa}, $(-\Delta_D)^{-1}$ obviously stands for the inverse 
of the Laplacian operator 
with homogeneous Dirichlet boundary conditions on $\Sigma$. It seems  however that the family $\{\tilde q_j,\,j=1,\ldots,N\}$ should be replaced by an orthonormal family (such 
as the one  we 
have denoted by 
$\{\hat\xi_j,\,j=1,\ldots,N\}$). This remark holds earlier as well, in \cite[Theorem 2.1]{Maekawa:2013aa}, the proof of which amounts to quote 
\cite[Theorem 3.20]{Kozono:2009aa} where the family $\{\tilde q_j,\,j=1,\ldots,N\}$ (with different notation though) is indeed  an orthonormal family. Besides this observation, 
the condition \eqref{maekawa}  can be rephrased in a simpler way:
%
\begin{prop}
\label{prop:maekawa_A1}
Assuming that $\omega$ belongs to $H^1(\F)$, condition \eqref{maekawa} (replacing the functions $\tilde q_j$ by the functions $\hat\xi_j$) is 
equivalent to the condition:
\begin{equation}
\label{maekawa_2}
\omega\in V_0.
\end{equation}
\end{prop}
\begin{proof}
Let $\omega$ be smooth in $\F$ and define $\psi$ the stream function such that:
\begin{equation}
\label{def_Stream_maeka}
\psi=(-\Delta_D)^{-1}\omega+\sum_{j=1}^N\big((-\Delta_D)^{-1}\nabla^\perp\omega,\nabla^\perp\hat\xi_j\big)_{L^2}
\hat\xi_j\qquad\text{in }\F.
\end{equation}
Then $\psi$ belongs in particular to $\bar S_1$ and $-\Delta\psi=\omega$ in $\F$. Condition \eqref{maekawa} means that $\psi$ is in $S_1$ 
and hence that $\omega$ is in $V_0$. Reciprocally, let $\omega$ be in $V_0\cap H^1(\F)$ and denote by $\tilde\psi$ the stream function in $S_1$ 
such that $-\Delta\tilde\psi=\omega$. In that case, $\nabla^\perp\tilde\psi=(-\Delta_D)^{-1}\nabla^\perp\omega$ in $\F$. Decomposing 
$\tilde\psi$ according to the orthogonal decomposition of the space $S_0=H^1_0(\F)\sumperp \mathbb F_S$, we obtain exactly the right hand side of \eqref{def_Stream_maeka} (and indeed  the family $\{\hat\xi_j,\,j=1,\ldots,N\}$ has to 
be an orthonormal family at this stage). 
Therefore $\tilde\psi=\psi$ and \eqref{maekawa} holds.
\end{proof}
%
Further in \cite{Maekawa:2013aa}, the dynamics for the vorticity is claimed to be governed by the system of equations:
\begin{subequations}
\label{maekawa_sys}
\begin{equation}
\label{maekawa_3}
\partial_t\omega-\nu\Delta\omega+u\cdot\nabla\omega=0\quad\text{ in }\F_T,
\end{equation}
with $u=\nabla^\perp\psi$, the stream function  $\psi$ being given by the Biot-Savart law \eqref{def_Stream_maeka}. The classical evolution equation \eqref{maekawa_3} is supplemented with an 
initial condition:
\begin{equation}
\omega(0)=\omega_0\in V_0
\end{equation}
and a {\it boundary condition} on $\Sigma_T$ (once again it seems that the functions $\tilde q_j$ in Maekawa's paper have to be replaced 
by the functions $\hat\xi_j$):
\begin{equation}
\label{maekawa_4}
\nu\bigg\{\frac{\partial\omega}{\partial n}-\Lambda_{DN}\omega+\sum_{j=1}^N(\nabla\omega,\nabla\hat\xi_j)_{\mathbf L^2(\F)}\frac{\partial\hat\xi_j}{\partial n}
\bigg\}
=-\frac{\partial}{\partial n}\big(-\Delta_D\big)^{-1}(u\cdot\nabla\omega)+\sum_{j=1}^N(\omega u,\nabla\hat\xi_j)_{\mathbf L^2(\F)}\frac{\partial\hat\xi_j}{\partial n}.
\end{equation}
\end{subequations}
Notice that in the case of a simply connected domain, this condition was already mentioned by Weinan and Jian-Guo in \cite{E:1996aa}, borrowed
from an earlier article of Anderson \cite{Anderson:1989aa}.
 \begin{prop}
 \label{prop:maekawa_A2}
 Condition \eqref{maekawa_4} is equivalent  for every $\omega$ solving \eqref{maekawa_3} to:
 $$\partial_t\omega(t)\in V_0\qquad\text{ for a.e. }t\in(0,T).$$
 \end{prop}
 \begin{proof}
 As already mentioned in the proof of Lemma~\ref{first_iso}, for every function $h$ harmonic in $\F$, the function
$$h_0=h-\sum_{j=1}^N \left(\int_{\Sigma}\frac{\partial \hat\xi_j}{\partial n}h\ds\right) \hat\xi_j,$$
is in $\mathfrak H$.
 For a.e. $t\in (0,T)$, let us form the scalar product in $L^2(\F)$ of \eqref{maekawa_3} with $h_0$. Integrating by parts, we obtain on the one 
 hand:
 $$\nu(\Delta \omega,h_0)_{L^2(\F)}=\nu\int_\Sigma\bigg\{\frac{\partial\omega}{\partial n}-\Lambda_{DN}\omega+\sum_{j=1}^N(\nabla\omega,\nabla\hat\xi_j)_{\mathbf L^2(\F)}\frac{\partial\hat\xi_j}{\partial n}\bigg\} h\ds,$$
 and on the other hand, considering the advection term:
 $$(u\cdot\nabla\omega,h_0)_{L^2(\F)}=\int_\Sigma\bigg\{-\frac{\partial}{\partial n}\big(-\Delta_D\big)^{-1}(u\cdot\nabla\omega)+\sum_{j=1}^N(\omega u,\nabla\hat\xi_j)_{\mathbf L^2(\F)}\frac{\partial\hat\xi_j}{\partial n}\bigg\}h\ds.$$
 This shows that \eqref{maekawa_4} is indeed equivalent to $(\partial_t\omega(t),h_0)_{L^2(\F)}=0$ for a.e. $t\in (0,T)$ and completes the 
 proof.
 \end{proof}
%
Summarizing Propositions~\ref{prop:maekawa_A1} and \ref{prop:maekawa_A2}, Maekawa's System \eqref{maekawa_sys} turns out to be
 equivalent to:
 \begin{subequations}
 \label{sys_faux}
 \begin{alignat}{3}
 \partial_t\omega-\nu\Delta\omega+u\cdot\nabla\omega&=0&\quad&\text{ in }\F_T,\\
 \omega(0)&=\omega_0\in V_0&&\text{ in }\F,\\
 \partial_t\omega&\in V_0&&\text{ on }(0,T).
 \end{alignat}
 \end{subequations}
This seems to contradict the claim of \cite[Theorem 2.3]{Maekawa:2013aa} (namely, the equivalence of System \eqref{maekawa_sys} with the classical 
NS equations in primitive variables) in a multiply connected domain 
because Lamb's fluxes conditions 
\eqref{eq:third_regukl_NS} (see \cite[Art. 328a]{Lamb:1993aa}) on the inner boundaries:
 $$\int_{\Sigma^-_j}\frac{\partial\omega}{\partial n}\ds=0\qquad\text{ for every }j=1,\ldots,N,$$
 are missing and cannot be figured out from System \eqref{sys_faux} (this is explained in \cite[Remark 3.2]{Guermond:1997aa}). Notice however that the equivalence holds in the particular case 
 of a simply connected fluid domain. 
 %
%
\subsubsection*{Existence and uniqueness of a global solution}
We shall now study the existence of solutions to System \eqref{main_NS_vorticity3} (or  equivalently \eqref{eq:main_vorti:suite}).  
For simplicity purpose, we restrict our analysis to the case where there is no source term and  to homogeneous boundary conditions for 
the velocity field. 
 The system we consider reads therefore as follows:
 \begin{subequations}
\label{main_NS_vorticity4}
\begin{alignat}{3}
\label{main_vorti_re_NS1}
\partial_t\omega_0+\nu\mathsf A^V_2\omega_0&=-\mathsf P\varLambda^V_{\mathsf r}(\omega)-\partial_t\omega_{\mathfrak B}&\qquad&\text{in }\F_T\\
\label{eq:second_regukl_NS1}
\nu\bar{\mathsf A}^V_2\omega_{\mathfrak B}&=-\mathsf P^\perp \varLambda^V_{\mathsf r}(\omega)&&\text{in }\F_T,\\
\omega(0)&=\omega^{\rm i}&&\text{in }\F,
\end{alignat}
\end{subequations}
 with $\omega^{\rm i}\in V_1$, $\omega=\omega_0+\omega_{\mathfrak B}\in V_2\oplus \mathfrak B_V^2$, $\varLambda^V_{\mathsf r}(\omega)=\nabla^\perp\psi\cdot\nabla\omega$ and $\psi=\bar\Delta_2^{-1}\omega$. System~\eqref{main_NS_vorticity4} can be rephrased as a more standard Cauchy problem whose unknown is $\omega_0$ (the coupling condition \eqref{eq:second_regukl_NS1}
  cannot be got rid of though since $\omega$ still appears in the nonlinear advection term):
  \begin{subequations}
\label{main_NS_vorticity4_cauchy}
\begin{alignat}{3}
\label{main_vorti_re_NS1_cauchy}
\partial_t\omega_0+\nu\mathsf A^V_2\omega_0&=-\mathsf P\varLambda^V_{\mathsf r}(\omega)+\frac{1}{\nu}
\big(\mathsf A^V_0\big)^{-1}\partial_t\big(
\mathsf P^\perp \varLambda^V_{\mathsf r}(\omega)\big)&\qquad&\text{in }\F_T\\
%
%
\label{init_omegai0}
\omega_0(0)&=\omega^{\rm i}+\frac{1}{\nu}\big(\mathsf A_1^V\big)^{-1}\mathsf P^\perp \varLambda^V_{\mathsf r}(\omega^{\rm i})&&\text{in }\F.
\end{alignat}
\end{subequations}
 The solution $\omega$ to System~\eqref{main_NS_vorticity4} will be looked for in the 
 space:
$$
 \Omega(T)=\big[L^2(0,T;H^2_V)\cap H^1(0,T;V_0)\big]\cap\big[\mathcal C([0,T];V_1)\cap \mathcal C^1([0,T];V_{-1})\big].
$$
%
\begin{theorem}
\label{meka_F}
For every positive time  $T$ and every initial data $\omega^{\rm i}\in V_1$, 
System~\eqref{main_NS_vorticity4} admits a unique solution $\omega$  in $\Omega(T)$. Moreover
this solution satisfies the exponential decay estimate:
\begin{equation}
\label{palinstrophy}
\|\omega(t)\|_{V_1}\leqslant \mathbf c_{[\F,\nu,\omega^{\rm i}]}e^{-\frac12\nu\lambda_\F t}\qquad\forallt t\in (0,T),
\end{equation}
where we emphasis  that the constant $\mathbf c_{[\F,\nu,\omega^{\rm i}]}$ does not depend on $T$.
\end{theorem}
\begin{rem} It is worth noticing that:
\begin{enumerate}
\item In \cite{Maekawa:2013aa}, the author shows local in time existence for the same system, considering the particular 
case where the fluid domain $\F$ is a half-plane.
\item The quantity $\|\nabla\omega\|^2_{\mathbf L^2(\F)}$ is sometimes called the palinstrophy. The palinstrophy being 
controlled by $\|\omega\|_{V_1}^2$, estimates \eqref{palinstrophy} asserts that the palinstrophy is exponentially decreasing 
as time growths. This result was not known so far and may play an important role in turbulence theory.
\end{enumerate}
\end{rem}
%
We begin with establishing the {\it a priori} estimate \eqref{palinstrophy}.
\begin{lemma}
\label{a_rpiori_estim}
For every initial condition $\omega^{\rm i}\in V_1$, there exists a positive constant $\mathbf c_{[\F,\nu,\omega^{\rm i}]}$ such 
that for every positive time $T$, any solution $\omega$ to System~\eqref{main_NS_vorticity4} in $\Omega(T)$ satisfies 
estimate \eqref{palinstrophy}.
\end{lemma}
\begin{proof}
The proof is divided in several steps:
\par
\medskip
\noindent{\bf First step:}
As being a weak solution to the $\omega-$NS equations, we can apply Corollary~\ref{energy_decay_weak} which provides us 
with the following estimates, satisfied for every $t$ in $(0,T)$:
\begin{subequations}
\label{der_der}
\begin{equation}
\|\omega(t)\|_{V_{-1}}\leqslant \|\omega^{\rm i}\|_{V_{-1}}e^{-\nu\lambda_\F t}\qquad\text{and}\qquad
\int_0^t \|\omega(s)\|_{V_0}^2\ds\leqslant \frac{1}{2\nu}\|\omega^{\rm i}\|_{V_{-1}}^2.
\end{equation}
Arguing that $\omega$ is also a strong solution to the $\omega-$NS equation, we obtain that:
$$\frac{1}{2}\frac{\rm d}{{\rm d}t}\|\omega(t)\|_{V_0}^2+\nu \|\omega\|_{V_1}^2\leqslant \mathbf c_\F\|\omega\|_{V_{-1}}^{\frac12}
\|\omega\|_{V_0}\|\omega\|_{V_1}^{\frac32}\leqslant \frac{\nu}{2}\|\omega\|_{V_1}^2+\frac{\mathbf c_\F}{\nu^3}\|\omega\|_{V_{-1}}^2\|\omega\|_{V_0}^4,$$
that is
$$
\frac{\rm d}{{\rm d}t}\|\omega(t)\|_{V_0}^2+\nu \|\omega\|_{V_1}^2\leqslant \frac{\mathbf c_\F}{\nu^3}\|\omega\|_{V_{-1}}^2\|\omega\|_{V_0}^4,
$$
which leads us to the estimates:
\begin{equation}
\|\omega(t)\|_{V_0}\leqslant \|\omega^{\rm i}\|_{V_0}\mathbf E_{[\F,\nu,\omega^{\rm i}]}e^{-\frac12\nu\lambda t}\qquad
\text{with}\qquad \mathbf E_{[\F,\nu,\omega^{\rm i}]} =\exp\left(\frac{\mathbf c_\F}{\nu^4}\|\omega^{\rm i}\|_{V_{-1}}^4\right),
\end{equation}
and also:
\begin{equation}
\int_0^t\|\omega(s)\|_{V_1}^2\ds\leqslant \frac{1}{\nu}\|\omega^{\rm i}\|_{V_0}^2\left[1 +\frac{\mathbf c_\F}{\nu^4}\|\omega^{\rm i}\|_{V_{-1}}^4\mathbf E_{[\F,\nu,\omega^{\rm i}]} \right].
\end{equation}
\end{subequations}
\par
\medskip
%
\noindent{\bf Second step:}
We need now to estimate $\|\omega(t)\|_{V_1}$ in term of $\|\omega_0(t)\|_{V_1}$ and $\|\omega(t)\|_{\bar V_2}$ in term of 
$\|\omega_0(t)\|_{V_2}$ (and possibly some lower order terms). 
\par
Starting from the expression \eqref{eq:second_regukl_NS1} and forming for 
a.e. $t$ in $(0,T)$ the duality pairing   with $\omega_{\mathfrak B}(t)$, we obtain:
\begin{equation}
\label{ffgty}
\nu\| \omega_{\mathfrak B}(t)\|_{V_1}^2=- \langle\mathsf P^\perp \varLambda^V_{\mathsf r}(\omega(t)),\omega_{\mathfrak B}(t)\rangle
_{V_{-1},V_1}\leqslant \mathbf c_\F 
\|\mathsf P^\perp \varLambda^V_{\mathsf r}(\omega(t))\|_{V_{-1}}\|\omega_{\mathfrak B}(t)\|_{V_1}.
\end{equation}
Introducing the stream function $\psi=\Delta^{-1}_1\omega$, we have for every $\theta\in V_1$:
$$\langle \mathsf P^\perp \varLambda^V_{\mathsf r}(\omega),\theta\rangle_{V_{-1},V_1}=(\mathsf P^\perp\big(\nabla^\perp\psi\cdot\nabla\omega\big),\mathsf Q_1\theta)_{L^2(\F)}
 =-(\nabla^\perp\psi\cdot\nabla\omega,\mathsf Q_1^\perp\theta)_{L^2(\F)}
 =(\omega\nabla^\perp\psi,\nabla \mathsf Q_1^\perp\theta)_{\mathbf L^2(\F)},
 $$
the latter expression resting on the equalities $\mathsf P^\perp\mathsf Q_1=({\rm Id}-\mathsf P)\mathsf Q_1=\mathsf Q_1-{\rm Id}=-\mathsf Q_1^\perp$. We can deduce first that:
\begin{equation}
\label{pmlopm}
\|\mathsf P^\perp \varLambda^V_{\mathsf r}(\omega(t))\|_{V_{-1}}\leqslant \mathbf c_\F\|\omega(t)\|_{L^4(\F)}\|\nabla\psi(t)\|_{\mathbf L^4(\F)}
 \leqslant \mathbf c_\F \|\omega(t)\|_{V_{-1}}^{\frac12}\|\omega(t)\|_{V_0}\|\omega(t)\|_{V_1}^{\frac12},
 \end{equation}
and next, combining the inequality above with \eqref{ffgty}, that:
$$\| \omega_{\mathfrak B}(t)\|_{V_1}\leqslant \frac{\mathbf c_\F}{\nu} \|\omega(t)\|_{V_{-1}}^{\frac12}\|\omega(t)\|_{V_0}
\|\omega(t)\|_{V_1}^{\frac12}.$$
Since $\|\omega(t)\|_{V_1}\leqslant \|\omega_{\mathfrak B}(t)\|_{V_1}+\|\omega_0(t)\|_{V_1}$, we obtain, using 
Young's inequality:
\begin{subequations}
\begin{equation}
\label{gbhnjk}
\|\omega(t)\|_{V_1}\leqslant \mathbf c\|\omega_0(t)\|_{V_1}+\frac{\mathbf c_\F}{\nu^2}\|\omega(t)\|_{V_{-1}}\|\omega(t)\|_{V_0}^2.
\end{equation}
We are donne with the term $\|\omega(t)\|_{V_1}$ so let us turn our attention to $\|\omega(t)\|_{\bar V_2}$. Forming, for 
a.e. $t$ in $(0,T)$, the scalar product 
of \eqref{eq:second_regukl_NS1_bis} with $\bar{\mathsf A}^V_2\omega_{\mathfrak B}(t)$ in $V_0$, we obtain (using H\"older's inequality
followed by interpolation inequalities):
$$\nu\|\omega_{\mathfrak B}(t)\|_{\bar V_2}^2=\big|(\nabla^\perp\psi\cdot\nabla\omega(t),\Delta\omega_{\mathfrak B}(t))_{L^2(\F)}\big|\leqslant
\mathbf c_\F\|\omega(t)\|_{V_{-1}}^{\frac12}\|\omega(t)\|_{V_{0}}^{\frac12}\|\omega(t)\|_{V_{1}}^{\frac12}\|\omega(t)\|_{\bar V_{2}}^{\frac12}\|\omega_{\mathfrak B}(t)\|_{\bar V_2},$$
that is to say:
$$\|\omega_{\mathfrak B}(t)\|_{\bar V_2}\leqslant\frac{\mathbf c_\F}{\nu}
\|\omega(t)\|_{V_{-1}}^{\frac12}\|\omega(t)\|_{V_{0}}^{\frac12}\|\omega(t)\|_{V_{1}}^{\frac12}\|\omega(t)\|_{\bar V_{2}}^{\frac12}.$$
But $\|\omega(t)\|_{\bar V_2}^2=\|\omega_0(t)\|_{V_2}^2+\|\omega_{\mathfrak B}(t)\|_{\bar V_2}^2$ which, by Young's inequality yields:
\begin{equation}
\label{first_goal}
\|\omega(t)\|_{\bar V_2}\leqslant\mathbf c\|\omega_0(t)\|_{V_2}+
\frac{\mathbf c_\F}{\nu^2}
\|\omega(t)\|_{V_{-1}}\|\omega(t)\|_{V_{0}}\|\omega(t)\|_{V_{1}}.
\end{equation}
\end{subequations}
Our goal for this step is now achieved. 
\par
\medskip
\noindent{\bf Third step:}
We form for a.e. $t\in(0,T)$, the scalar product of \eqref{main_vorti_re_NS1_cauchy} 
with $\mathsf A^V_2\omega_0(t)$ in $V_0$ to obtain:
\begin{equation}
\label{almost_last}
\frac{1}{2}\frac{\rm d}{{\rm d}t}\|\omega_0(t)\|_{V_1}^2+\nu\|\omega_0(t)\|_{V_2}^2=-
\big(\mathsf P \varLambda^V_{\mathsf r}(\omega)(t),\mathsf A^V_2\omega_0(t)\big)_{V_0}+\frac{1}{\nu}\big(\big({\mathsf A}^V_0\big)^{-1}\big[\partial_t\mathsf P^\perp \varLambda^V_{\mathsf r}(\omega)(t)\big],\mathsf  A^V_2\omega_0(t)\big)_{V_0}.
\end{equation}
Both terms in the right hand side have to be estimated, this task being easier for the first one than for the second one.
Indeed, the first term can be rewritten as:
$$\big(\mathsf P \varLambda^V_{\mathsf r}(\omega)(t),\mathsf A^V_2\omega_0(t)\big)_{V_0}=
\big(\nabla^\perp\psi\cdot\nabla\omega(t),\Delta\omega_0(t))_{L^2(\F)},$$
whence we deduce that:
$$
\big|\big(\mathsf P \varLambda^V_{\mathsf r}(\omega)(t),\mathsf A^V_2\omega_0(t)\big)_{V_0}\big|\leqslant \mathbf c_\F\|\omega(t)\|_{V_{-1}}^{\frac12}\|\omega(t)\|_{V_{0}}^{\frac12}\|\omega(t)\|_{V_{1}}^{\frac12}\|\omega(t)\|_{\bar V_{2}}^{\frac12}\|\omega_0(t)\|_{V_2}.
$$
Once combined with \eqref{first_goal} to get rid of the term $\|\omega(t)\|_{\bar V_{2}}$, we end up with:
\begin{multline}
\label{first_pne:eq}
\big|\big(\mathsf P \varLambda^V_{\mathsf r}(\omega)(t),\mathsf A^V_2\omega_0(t)\big)_{V_0}\big|\leqslant 
\frac{\mathbf c_\F}{\nu}
\|\omega(t)\|_{V_{-1}}\|\omega(t)\|_{V_{0}}\|\omega(t)\|_{V_{1}}\|\omega_0(t)\|_{V_{2}}+\\
\mathbf c_\F\|\omega(t)\|_{V_{-1}}^{\frac12}\|\omega(t)\|_{V_{0}}^{\frac12}\|\omega(t)\|_{V_{1}}^{\frac12}\|\omega_0(t)\|_{V_2}^{\frac32},
\end{multline}
and we are now done with the first nonlinear term.
\par
\medskip
\noindent{\bf Fourth step:}
The second term in the right hand side of \eqref{almost_last} can be turned into:
$$\big(\big({\mathsf A}^V_0\big)^{-1}\big[\partial_t\mathsf P^\perp  \varLambda^V_{\mathsf r}(\omega)(t)\big],\mathsf  A^V_2\omega_0(t)\big)_{V_0}
=\big(\partial_t\mathsf P^\perp \varLambda^V_{\mathsf r}(\omega)(t),\mathsf A_0^V\mathsf A_2^V\omega_0(t)\big)_{V_{-2}}=
\big\langle\partial_t\mathsf P^\perp  \varLambda^V_{\mathsf r}(\omega)(t), \omega_0(t)\big\rangle_{V_{-2},V_2},$$
and therefore:
\begin{equation}
\label{eq:833}
\big|\big(\big({\mathsf A}^V_0\big)^{-1}\big[\partial_t\mathsf P^\perp  \varLambda^V_{\mathsf r}(\omega)(t)\big],\mathsf  A^V_2\omega_0(t)\big)_{V_0}\big|\leqslant \|\partial_t\mathsf P^\perp  \varLambda^V_{\mathsf r}(\omega)(t)\|_{V_{-2}}\|\omega_0(t)\|_{V_2}.
\end{equation}
From the expression \eqref{f_VpV2} we deduce that: 
\begin{equation}
\label{first_multi}
\|\partial_t\mathsf P^\perp  \varLambda^V_{\mathsf r}(\omega)(t)\|_{V_{-2}}\leqslant \mathbf c  \|\nabla\psi\|_{\mathbf L^4(\F)}\|\nabla \partial_t\psi\|_{\mathbf L^4(\F)}\leqslant \mathbf c_\F \|\omega(t)\|_{V_{-1}}^{\frac12}\|\omega(t)\|_{V_{0}}^{\frac12}\|\partial_t\omega(t)\|_{V_{-1}}^{\frac12}
\|\partial_t\omega(t)\|_{V_0}^{\frac12},
\end{equation}
and we need now to estimate both terms involving a time derivative.
As being a strong solution to the $\omega-$NS equation, $\omega(t)$ satisfies for a.e. $t$ in $(0,T)$ the identity below, set in $V_{-1}$:
$$\partial_t\omega(t)=-\nu\mathsf A_1^V\omega(t)-\varLambda_0^V(\omega(t),0),$$
where we recall that the definition of $\varLambda_0^V$ is given in \eqref{def_lambda_0V}. This equality provides us with the
 inequality:
$$
\|\partial_t\omega(t)\|_{V_{-1}}\leqslant \nu\|\omega(t)\|_{V_1}+\|\varLambda_0^V(\omega(t),0)\|_{V_{-1}}.
$$
Resting on the definition \eqref{def_lambda_0V}, we next easily obtain that:
$$
\|\varLambda_0^V(\omega(t),0)\|_{V_{-1}}\leqslant \|\nabla\psi(t)\|_{\mathbf L^4(\F)}\|\omega(t)\|_{L^4(\F)} 
\leqslant \mathbf c_\F\|\omega(t)\|_{V_{-1}}^{\frac12}\|\omega(t)\|_{V_{0}} \|\omega(t)\|_{V_{1}}^{\frac12},
$$
and therefore:
\begin{subequations}
\label{eq:83555}
\begin{equation}
\label{eq:835}
\|\partial_t\omega(t)\|_{V_{-1}}\leqslant \nu\|\omega(t)\|_{V_1}+\mathbf c_\F\|\omega(t)\|_{V_{-1}}^{\frac12}\|\omega(t)\|_{V_{0}} \|\omega(t)\|_{V_{1}}^{\frac12}.
\end{equation}
On the other hand, since by hypothesis $\omega$ is a solution on $(0,T)$ to System~\eqref{main_NS_vorticity4}, 
it satisfies for a.e. $t$ in $(0,T)$:
$$
\partial_t\omega(t)=\nu\Delta\omega(t)-\nabla^\perp\psi(t)\cdot\nabla\omega(t)\qquad\text{in }V_0,
$$
whence we deduce that:
$$
\|\partial_t\omega(t)\|_{V_0}\leqslant \nu\|\omega\|_{\bar V_2}+\mathbf c_\F\|\omega(t)\|_{V_{-1}}^{\frac12}\|\omega(t)\|_{V_{0}}^{\frac12}\|\omega(t)\|_{V_{1}}^{\frac12}\|\omega(t)\|_{\bar V_{2}}^{\frac12}.
$$
Once combined with \eqref{first_goal}, this estimate becomes:
\begin{equation}
\label{eq:836}
\|\partial_t\omega(t)\|_{V_0}\leqslant \mathbf c\nu\|\omega_0\|_{V_2}+
\frac{\mathbf c_\F}{\nu}\|\omega(t)\|_{V_{-1}}\|\omega(t)\|_{V_{0}}\|\omega(t)\|_{V_{1}}
+\mathbf c_\F\|\omega(t)\|_{V_{-1}}^{\frac12}\|\omega(t)\|_{V_{0}}^{\frac12}\|\omega(t)\|_{V_{1}}^{\frac12}
\|\omega_0(t)\|_{V_{2}}^{\frac12}.
\end{equation}
\end{subequations}
Gathering now \eqref{eq:833} and identities  \eqref{eq:83555} we finally obtain the following estimate for 
the second nonlinear term in \eqref{almost_last}:
\begin{multline}
\label{second_pne:eq}
\big|\big(\big({\mathsf A}^V_0\big)^{-1}\big[\partial_t\mathsf P^\perp  \varLambda^V_{\mathsf r}(\omega)(t)\big],\mathsf  A^V_2\omega_0(t)\big)_{V_0}\big|\leqslant \mathbf c_\F\|\omega(t)\|_{V_{-1}}^{\frac12}\|\omega(t)\|_{V_{0}}^{\frac12}\bigg(
\nu\|\omega(t)\|_{V_{1}}^{\frac12}\|\omega_0(t)\|_{V_2}^{\frac32}\\
+\|\omega(t)\|_{V_{-1}}^{\frac12}\|\omega(t)\|_{V_{0}}^{\frac12}\|\omega(t)\|_{V_1}\|\omega_0(t)\|_{V_2}
+\sqrt{\nu}\|\omega(t)\|_{V_{-1}}^{\frac14}\|\omega(t)\|_{V_{0}}^{\frac14}\|\omega(t)\|_{V_1}^{\frac34}\|\omega_0(t)\|_{V_2}^{\frac54}\\
+\sqrt{\nu}\|\omega(t)\|_{V_{-1}}^{\frac14}\|\omega(t)\|_{V_{0}}^{\frac12}\|\omega(t)\|_{V_1}^{\frac14}\|\omega_0(t)\|_{V_2}^{\frac32}
+\frac{1}{\sqrt{\nu}}\|\omega(t)\|_{V_{-1}}^{\frac34}\|\omega(t)\|_{V_{0}}\|\omega(t)\|_{V_1}^{\frac34}\|\omega_0(t)\|_{V_2}\\
+\|\omega(t)\|_{V_{-1}}^{\frac12}\|\omega(t)\|_{V_{0}}^{\frac34}\|\omega(t)\|_{V_1}^{\frac12}\|\omega_0(t)\|_{V_2}^{\frac54}\bigg).
\end{multline}
Both terms in the right hand side of \eqref{almost_last} have now be estimated. Let us collect all the estimates obtained so far 
and move on to the next step consisting in applying Gr\"onwall's inequality.
\par
\medskip
\noindent{\bf Fifth step:} Combining \eqref{first_pne:eq} and \eqref{second_pne:eq} with \eqref{almost_last} and using 
craftily Young's inequality several times, we deduce that for a.e. $t$ in $(0,T)$:
\begin{equation}
\label{antepenul}
\frac{1}{2}\frac{\rm d}{{\rm d}t}\|\omega_0(t)\|_{V_1}^2+\nu\|\omega_0(t)\|_{V_2}^2\leqslant \frac{\nu}{2}\|\omega_0(t)\|_{V_2}^2
+ \frac{\mathbf c_\F}{\nu^3} \|\omega(t)\|_{V_{-1}}^2\|\omega(t)\|_{V_{0}}^2\|\omega(t)\|_{V_1}^2+
\frac{\mathbf c_\F}{\nu^7}\|\omega(t)\|_{V_{-1}}^4\|\omega(t)\|_{V_{0}}^6.
\end{equation}
Applying 
Gr\"onwall's inequality to \eqref{antepenul}, we obtain:
\begin{equation}
\label{ggftp}
e^{\nu \lambda_\F t}\|\omega_0(t)\|_{V_1}^2(t) \leq \|\omega_0^{\rm i}\|_{V_1}^2 +  \int_{0}^t \left[\frac{\mathbf c_\F}{\nu^3} \|\omega(s)\|_{V_{-1}}^2\|\omega(s)\|_{V_{0}}^2\|\omega(s)\|_{V_1}^2+
\frac{\mathbf c_\F}{\nu^7}\|\omega(s)\|_{V_{-1}}^4\|\omega(s)\|_{V_{0}}^6\right]e^{\nu \lambda_\F s}\ds,
\end{equation}
where, according to \eqref{init_omegai0}, the initial data $\omega^{\rm i}_0$ is defined by:
$$\omega^{\rm i}_0=\omega^{\rm i}+
\frac{1}{\nu}\big(\mathsf A_1^V\big)^{-1}\mathsf P^\perp \varLambda^V_{\mathsf r}(\omega^{\rm i}).$$
%
 %
 With \eqref{pmlopm}, we deduce that:
 \begin{subequations}
\label{gtfdrtt}
 \begin{equation}
 \|\omega^{\rm i}_0\|_{V_1}\leqslant \|\omega^{\rm i}\|_{V_1}+\frac{\mathbf c_\F}{\nu}
  \|\omega^{\rm i}\|_{V_{-1}}^{\frac12}\|\omega^{\rm i}\|_{V_0}\|\omega^{\rm i}\|_{V_1}^{\frac12}
  \leqslant \mathbf c\|\omega^{\rm i}\|_{V_1}+ \frac{\mathbf c_\F}{\nu^2}
  \|\omega^{\rm i}\|_{V_{-1}} \|\omega^{\rm i}\|_{V_0}^2.
  \end{equation}
The second term in the right-hand side of \eqref{ggftp} can be estimated using the estimates \eqref{der_der} as follows:
\begin{align}
\nonumber
\int_0^t \frac{\mathbf c_\F}{\nu^3} \|\omega(s)\|_{V_{-1}}^2\|\omega(s)\|_{V_{0}}^2\|\omega(s)\|_{V_1}^2 e^{\nu\lambda_\F s} \ds   \leq & \frac{\mathbf c_\F}{\nu^3}\|\omega^{\rm i}\|_{V_{-1}}^2\|\omega^{\rm i}\|_{V_{0}}^2 \mathbf E_{[\F,\nu,\omega^{\rm i}]}   \int_0^t e^{-2\nu \lambda_\F s}\|\omega(s)\|_{V_1}^2\ds \\
  \leq & \frac{\mathbf c_\F}{\nu^4}\|\omega^{\rm i}\|_{V_{-1}}^2\|\omega^{\rm i}\|_{V_{0}}^4 \mathbf E_{[\F,\nu,\omega^{\rm i}]}   \left[1 +\frac{\mathbf c_\F}{\nu^4}\|\omega^{\rm i}\|_{V_{-1}}^4\mathbf E_{[\F,\nu,\omega^{\rm i}]} \right],
\end{align}
and 
\begin{align}
\nonumber
	\int_0^t \frac{\mathbf c_\F}{\nu^7}\|\omega(s)\|_{V_{-1}}^4\|\omega(s)\|_{V_{0}}^6  e^{\nu\lambda_\F s} \ds & \leq  \frac{\mathbf c_\F}{\nu^7}\|\omega^{\rm i}\|_{V_{-1}}^4\|\omega^{\rm i}\|_{V_{0}}^4\mathbf E_{[\F,\nu,\omega^{\rm i}]}   \int_0^t e^{-5\nu \lambda_\F s}\|\omega(s)\|_{V_0}^2\ds \\
	& \leq \frac{\mathbf c_\F}{\nu^8}\|\omega^{\rm i}\|_{V_{-1}}^6\|\omega^{\rm i}\|_{V_{0}}^4\mathbf E_{[\F,\nu,\omega^{\rm i}]}.
\end{align}
\end{subequations}
We finally obtain, gathering \eqref{ggftp}  and inequalities \eqref{gtfdrtt} (using again Young's inequality):
\begin{equation}
\|\omega_0(t)\|_{V_1}^2\leqslant \mathbf c_\F\left[\|\omega^{\rm i}\|_{V_1}^2+\frac{1}{\nu^4}\|\omega^{\rm i}\|_{V_{-1}}^2
\|\omega^{\rm i}\|_{V_0}^4\mathbf E_{[\F,\nu,\omega^{\rm i}]}
+\frac{1}{\nu^8}\|\omega^{\rm i}\|_{V_{-1}}^6
\|\omega^{\rm i}\|_{V_0}^4\mathbf E_{[\F,\nu,\omega^{\rm i}]} \right] e^{-\nu \lambda_\F t},
\end{equation}
what, with \eqref{gbhnjk} and estimates \eqref{der_der}, completes the proof of the lemma. 
\end{proof}
%
The rest of the section is devoted to the proof of the theorem, which is classical and  based again on a fixed point argument. 
%
Let us fix $T>0$ and $\omega^{\rm i}\in V_1$ and
 introduce the spaces:
\begin{align*}
\Omega(T,\omega^{\rm i})&=\big\{\omega\in \Omega(T)\,:\,\omega(0)=\omega^{\rm i}\big\},\\
\Psi(T)&=\big[L^2(0,T;\bar S_3)\cap H^1(0,T;S_1)\big]\cap\big[\mathcal C([0,T];S_2)\cap \mathcal C^1([0,T];S_0)\big],
\end{align*}
and 
$$F(T)=\{f_V\in L^2(0,T; L^2_V)\cap\mathcal C([0,T];V_{-1})\,:\, \mathsf P^\perp f_V\in H^1(0,T;V_{-2})\}.$$
Then define the mapping $\mathsf X_T:f_V\in F(T)\mapsto \omega\in \Omega(T,\omega^{\rm i})$ where $\omega=\omega_0+\omega_{\mathfrak B}$ is the solution to the $\omega-$Stokes problem:
\begin{subequations}
\label{main_NS_vorticity4_bis}
\begin{alignat}{3}
\label{main_vorti_re_NS1_bis}
\partial_t\omega_0+\nu\mathsf A^V_2\omega_0&=\mathsf Pf_V-\partial_t\omega_{\mathfrak B}&\qquad&\text{in }\F_T\\
\label{eq:second_regukl_NS1_bis}
\nu\bar{\mathsf A}^V_2\omega_{\mathfrak B}&=\mathsf P^\perp f_V&&\text{in }\F_T,\\
\omega(0)&=\omega^{\rm i}&&\text{in }\F,
%
%
\end{alignat}
\end{subequations}
and $\mathsf Y_T:\omega\in\Omega(T,\omega^{\rm i})\mapsto \varLambda^V_{\mathsf r}(\omega)\in F(T)$ where $\varLambda^V_{\mathsf r}(\omega)$ 
is defined in \eqref{def_lambda_R} (with $\varphi=0$ since, as already mentioned, we consider only homogeneous boundary conditions).
%
\begin{lemma}
\label{mekawa_esim}
The mapping $\mathsf X_T$ is well-defined and there exists a positive constant $\mathbf c_{[\F,\nu]}$ such that:
\begin{subequations}
\begin{equation}
\label{estim_XT}
\|\mathsf X_T(f_V)\|_{\Omega(T)}\leqslant \mathbf c_{[\F,\nu]}\big[\|\omega^{\rm i}\|^2_{V_1}+\|f_V\|_{F(T)}^2\big]^{{\frac12}}
\forallt f_V\in F(T).
\end{equation}
%
%
%
The mapping $\mathsf Y_T$ is also well-defined and there exists a positive constant $\mathbf c_\F$ such that, for all $\omega_1$ and 
$\omega_2$ in $\Omega(T,\omega^{\rm i})$:
\begin{equation}
\label{eval_YT}
\|\mathsf Y_T(\omega_2)-\mathsf Y_T(\omega_1)\|_{F(T)}\leqslant\mathbf c_{\F}T^{\frac{1}{10}}\big[\|\omega_1\|_{\Omega(T)}^2+\|\omega_2\|_{\Omega(T)}^2\big]^{{\frac12}}
\|\omega_2-\omega_1\|_{\Omega(T)},
\end{equation}
\end{subequations}
providing that $T<1$.
\end{lemma}
\begin{proof}
The mapping $\mathsf X_T$ is well-defined from the space $F_{\mathsf r}(T)$ (defined in \eqref{def:FT}) into $\bar V_1(T)$ according to Proposition~\ref{stokes_regul} and following Proposition~\ref{stokes_regul}, there exists a positive constant $\mathbf c_{[\F,\nu]}$ such that:
$$\|\omega\|_{\bar V_1(T)}\leqslant \mathbf c_{[\F,\nu]}\big[\|\omega^{\rm i}\|_{V_1}^2+\|f_V\|_{L^2(0,T; L^2_V)}^2+\|\partial _t (\mathsf P^\perp f_V)\|_{L^2(0,T;V_{-2})}^2\big]^{{\frac12}}.$$
However, comparing with Proposition~\ref{stokes_regul}, the source term $f_V$ is assumed herein to satisfy   the extra hypothesis $f_V\in \mathcal C([0,T];V_{-1})$ (and not only $\mathsf P^\perp f_V
\in  \mathcal C([0,T];V_{-1})$). We recall that every solution to the $\omega-$Stokes problem \eqref{main_NS_vorticity4_bis} satisfies also:
$$\partial_t \omega=-\mathsf A^V_1\omega+f_V\qquad\text{in }\F_T.$$
Since $\omega$ belongs in particular to $\mathcal C([0,T],V_1)$, we infer that $\partial_t\omega$ is in $\mathcal C([0,T];V_{-1})$ and finally that there exists a constant 
$\mathbf c_{[\F,\nu]}$ such that \eqref{estim_XT} holds.
\par
%
 %
For every $\theta\in V_1$, we have by definition:
\begin{subequations}
\begin{align}
\label{fV_V1}
\langle \varLambda^V_{\mathsf r}(\omega),\theta\rangle_{V_{-1},V_1}&=(\nabla^\perp\psi\cdot\nabla\omega,\mathsf Q_1\theta)_{L^2(\F)}
=-(\omega\nabla^\perp\psi,\nabla(\mathsf Q_1\theta))_{\mathbf L^2(\F)},\\
\label{fV_V1_b}
\langle \mathsf P^\perp \varLambda^V_{\mathsf r}(\omega),\theta\rangle_{V_{-1},V_1}&=(\mathsf P^\perp\big(\nabla^\perp\psi\cdot\nabla\omega\big),\mathsf Q_1\theta)_{L^2(\F)}
 =-(\nabla^\perp\psi\cdot\nabla\omega,\mathsf Q_1^\perp\theta)_{L^2(\F)},
 \end{align}
the latter expression resting on the equalities $\mathsf P^\perp\mathsf Q_1=({\rm Id}-\mathsf P)\mathsf Q_1=\mathsf Q_1-{\rm Id}=-\mathsf Q_1^\perp$. Assuming now that $\theta$ belongs to $V_2$, the right hand side in 
 \eqref{fV_V1_b} can be integrated by parts twice to obtain:
$$
(\nabla^\perp\psi\cdot\nabla\omega,\mathsf Q_2^\perp\theta)_{L^2(\F)}
 =(D^2(\mathsf Q_2^\perp\theta)\nabla\psi,\nabla^\perp\psi)_{\mathbf L^2(\F)},$$
whence it can be deduced   in particular that:
\begin{equation}
\label{f_VpV2}
\langle \partial_t\big(\mathsf P^\perp \varLambda^V_{\mathsf r}(\omega)\big),\theta\rangle_{V_{-2},V_2}=(D^2(\mathsf Q_2^\perp\theta)\nabla\partial_t\psi,\nabla^\perp\psi)_{\mathbf L^2(\F)}+
(D^2(\mathsf Q_2^\perp\theta)\nabla\psi,\nabla^\perp\partial_t\psi)_{\mathbf L^2(\F)}.
\end{equation}
\end{subequations}
%
The same arguments as those used in the proof of Equality \eqref{estim_varlambda_strong_2} yield:
$$\|\nabla^\perp\psi\cdot\nabla\omega \|_{L^2(0,T;L^2(\F))}^2\leqslant \mathbf c_{\F}T^{\frac15}\|\omega\|_{\mathcal C([0,T];V_1)}^{\frac{2}{5}}\|\psi\|^2_{\mathcal C([0,T];S_1)}
\|\omega\|_{L^2(0,T;H^2_V)}^{\frac{8}{5}},$$
which entails that:
%
\begin{equation}
\label{NLT_1}
\|\varLambda_{\mathsf r}^V(\omega_2)-\varLambda_{\mathsf r}^V(\omega_1)\|_{L^2(0,T;L^2_V)}\leqslant  \mathbf c_{\F}T^{\frac{1}{10}}\big[\|\omega_1\|_{\Omega(T)}^2+\|\omega_2\|_{\Omega(T)}^2\big]^{{\frac12}}
\|\omega_2-\omega_1\|_{\Omega(T)}.
\end{equation}
Considering now the expression \eqref{fV_V1}, we first easily obtain:
\begin{subequations}
\label{sub:tout}
\begin{multline}
\label{long_formula_o}
\|\omega_2\nabla^\perp\psi_2-\omega_1\nabla^\perp\psi_1\|_{\mathbf L^2(\F)}\leqslant \mathbf c_\F\|\omega_2-\omega_1\|_{V_0}^{\frac15}
\|\omega_2-\omega_1\|_{L^4(\F)}^{\frac45}\|\nabla\psi_2\|_{\mathbf L^5(\F)}\\
+ 
\mathbf c_\F\|\omega_1\|_{L^4(\F)}\|\nabla(\psi_2-\psi_1)\|_{\mathbf L^4(\F)}.
\end{multline}
On the one hand, since $\omega_1$ and $\omega_2$ share the same initial value, we are allowed to write that:
\begin{equation}
\label{estim_dotomega}
\|\omega_2-\omega_1\|_{\mathcal C([0,T];V_0)}\leqslant  T^{{\frac12}}\|\partial_t\omega_2-\partial_t\omega_1\|_{L^2(0,T;V_0)}.
\end{equation}
On the other hand, Sobolev embedding theorem ensures that:
\begin{align}
\|\omega_2-\omega_1\|_{\mathcal C([0,T];L^4(\F))}&\leqslant \mathbf c_\F \|\omega_2-\omega_1\|_{\mathcal C([0,T];V_1)}\\
\|\nabla\psi_2\|_{\mathcal C([0,T];\mathbf L^5(\F))}&\leqslant \mathbf c_\F\|\omega_2\|_{\mathcal C([0,T];V_0)}.
\end{align}
The second term in the right hand side of \eqref{long_formula_o} is estimated in a similar manner, thus:
\begin{equation}
\|\nabla(\psi_2-\psi_1)\|_{\mathcal C([0,T];\mathbf L^4(\F))}\leqslant \mathbf c_\F\|\omega_2-\omega_1\|_{\mathcal C([0,T];V_0)}
\leqslant \mathbf c_\F T^{{\frac12}}\|\partial_t\omega_2-\partial_t\omega_1\|_{L^2(0,T;V_0)}.
\end{equation}
\end{subequations}
Assuming that $T<1$,  the estimates \eqref{sub:tout} give rise to:
\begin{equation}
\label{NLT_2}
\|\varLambda_{\mathsf r}^V(\omega_2)-\varLambda_{\mathsf r}^V(\omega_1)\|_{\mathcal C([0,T];V_{-1})}\leqslant  \mathbf c_{\F}T^{\frac{1}{10}}\big[\|\omega_1\|_{\Omega(T)}^2+\|\omega_2\|_{\Omega(T)}^2\big]^{{\frac12}}
\|\omega_2-\omega_1\|_{\Omega(T)}.
\end{equation}
We turn now our attention to the right hand side of \eqref{f_VpV2}. On the one hand, we obtain that:
\begin{subequations}
\label{both_estim}
\begin{multline}
\||\nabla(\partial_t\psi_2-\partial_t\psi_1)||\nabla\psi_2|\|^2_{L^2(0,T;L^2(\F))}\\
\leqslant \mathbf c_\F T^{\frac15}\|\partial_t\omega_2-\partial_t\omega_1\|_{\mathcal C([0,T];V_{-1})}^{\frac{2}{5}}
\|\omega_2\|_{\mathcal C([0,T];V_0)}^2\|\partial_t\omega_2-\partial_t\omega_1\|_{L^2(0,T;V_0)}^{\frac{8}{5}}.
\end{multline}
On the other hand, using again \eqref{estim_dotomega}:
%
\begin{equation}
\||\nabla \partial_t\psi_1||\nabla(\psi_2-\psi_1)|\|_{L^2(0,T;L^2(\F))}\leqslant \mathbf c_\F T^{{\frac12}}\|\partial_t\omega_1\|_{L^2(0,T;V_0)}\|\partial_t\omega_2-\partial_t\omega_1\|_{L^2(0,T;V_0)}.
\end{equation}
\end{subequations}
Providing again that $T<1$, both estimates \eqref{both_estim} yield:
\begin{equation}
\label{NLT_3}
\|\partial_t(\varLambda_{\mathsf r}^V(\omega_2))-\partial_t(\varLambda_{\mathsf r}^V(\omega_1))\|_{L^2(0,T;V_{-2})}\leqslant \mathbf c_{\F}T^{\frac{1}{10}}\big[\|\omega_1\|_{\Omega(T)}^2+\|\omega_2\|_{\Omega(T)}^2\big]^{{\frac12}}
\|\omega_2-\omega_1\|_{\Omega(T)}.
\end{equation}
Estimate \eqref{eval_YT} derives now straightforwardly from \eqref{NLT_1}, \eqref{NLT_2} and \eqref{NLT_3}. This completes the proof.
\end{proof}
\begin{proof}[Proof of Theorem~\ref{meka_F}]
Define the mapping $\mathsf Z_T:f_V\in F(T)\mapsto \mathsf Y_T\circ\mathsf X_T(f_V)\in F(T)$ et let $f_V^{\rm i}=\mathsf Z_T(0)$. Then, according 
to the estimates of Lemma~\ref{mekawa_esim}:
\begin{align*}
\|\mathsf Z_T(f_V)-f_V^{\rm i}\|_{F(T)}&\leqslant \mathbf c_{[\F,\nu]}T^{\frac{1}{10}}\big(\|f_V\|_{F(T)}^2+\|\omega^{\rm i}_0\|_{V_1}^2\big)
\forallt f_V\in F(T),\\
\|\mathsf Z_T(f_V^1)-\mathsf Z_T(f_V^2)\|_{F(T)}&\leqslant \mathbf c_{[\F,\nu]}
T^{\frac{1}{10}}\big(\|f^1_V\|_{F(T)}^2+\|f^2_V\|_{F(T)}^2+\|\omega^{\rm i}\|_{V_1}^2\big)^{{\frac12}}\| f_V^1-f_V^2\|_{F(T)},
\end{align*}
for every $f_V^1,f_V^2$ in $F(T)$ and $T<1$. For every $R>0$, there exists a time $T^\ast<1$ (depending only on $\F$, $\nu$, $\|\omega^{\rm i}_0\|_{V_1}$ 
and $R$) such that $\mathsf Z_{T^\ast}$ is 
a contraction from $B(f_V^{\rm i},R)\subset F(T^\ast)$ into $B(f_V^{\rm i},R)$. From Banach fixed point theorem, the mapping $\mathsf Z_{T^\ast}$ admits a unique 
fixed point in $B(f_V^{\rm i},R)$, the image of which by the mapping $\mathsf X_{T^\ast}$ is a solution to System~\eqref{main_NS_vorticity4}
on $[0,T^\ast)$. We conclude that $T^\ast$ can be chosen arbitrarily large following the lines of the proof of Theorem~\ref{THEO:EXIST_NS_STRONG}, using the estimate of Lemma~\ref{a_rpiori_estim}.
Finally, every solution is also a strong solution in the sense of Definition~ \ref{defi:Navier_Stokes_strong_vorticity}, which 
was proved to be unique.
\end{proof}
\section{Concluding remarks}
\label{SEC:conclude}
By introducing a suitable functional framework, the 2D vorticity equation has been shown to be not a classical parabolic equation but 
rather a parabolic-elliptic coupling. Indeed, applying the harmonic Bergman
projection to the equation $\partial_t\omega-\nu\Delta\omega+u\cdot\nabla\omega=0$
leads to its splitting   into, on the one hand, an evolution diffusion-advection equation for the non-harmonic 
part of $\omega$ (equation \eqref{main_vorti_re_NS1_bis}) and on the other hand 
a (steady) elliptic equation for the remaining harmonic part (equation \eqref{eq:second_regukl_NS1_bis}). 
By exploiting this structure of the equation, we were able to prove the exponential decay of the palinstrophy for large time, 
a result which was 
not known so far. In this work, it is worth noticing the surprising role played by 
the circulation in this context, circulation being well known for entering the analysis of perfect fluids but usually less came across  in the context 
of viscous fluids. The other point that deserves to be highlighted is the simple form taken by the Biot-Savart operator, described in Theorem~\ref{biot-savart_simple}.
\par
In a forthcoming work, we shall apply our method  to fluid-structure problems by considering a set of disks, pinned
at their centers but free to rotate, immersed in 
a viscous fluid. The equations governing the coupled fluid-rotating disks system can be stated in terms of the vorticity of the fluid and the angular velocities of the 
disks only. The analysis of these equations   will obviously be carried out in nonprimitive variables. 
%
\appendix
\section{Gelfand triple}
\label{gelf_triple}
\subsection{General settings}
Let $H_1$ and ${H_0}$ be two Hilbert spaces. Their scalar products are denoted respectively by $(\cdot,\cdot)_1$ and $(\cdot,\cdot)_0$  and their norms 
by $\|\cdot\|_{1}$ and $\|\cdot\|_0$. We assume that:
\begin{equation}
\label{first_Gel}
{H_1}\subset {H_0},
\end{equation}
where the inclusion is continuous and dense. 
Applying Riesz representation theorem, the space $H_0$ is identified with its dual $H'_0$. It means that, for every $u\in H_0$, the linear form 
$(\cdot,u)_0$ is identified with $u$.  The space $H_0$ is usually referred to as the pivot space. 
Therefore, 
the space ${H_1}$ cannot be identified with its dual $H_{-1}$ but with a subspace of $H_0$. Thus, the configuration
\begin{equation}
\label{first_Gel_1}
{H_1}\subset {H_0}\subset {H_{-1}},
\end{equation}
is called (with   a slight abuse of  terminology) Gelfand triple. The inclusions are both continuous and dense.
\par
%
We define the operator 
\begin{equation}
\label{eq:defA1}
{\mathsf A}_{1}:{H_1}\longrightarrow {H_{-1}},\qquad {\mathsf A}_{1}u=(u,\cdot)_1,\quad\text{for all }u\in H_1,
\end{equation}
and it can be readily verified that ${\mathsf A}_1$ is an isometry.
%
%
%
%
%
Then, we define  the space ${H_2}={\mathsf A}_{1}^{-1}H_0$ and the operator ${\mathsf A}_2:{H_2}\to H_0$ by setting, for every $u\in {H_2}$:
\begin{equation}
\label{defA2}
({\mathsf A}_2u,\cdot)_0={\mathsf A}_1u\quad\text{in }H_{-1}.
\end{equation}
We equip the space $H_2$ with the scalar product:
$$(u,v)_2=({\mathsf A}_2 u,{\mathsf A}_2 v)_0,\qquad \text{for all }u,v\in {H_2},$$
and the corresponding norm $\|\cdot\|_2$.
%
%
%
\begin{lemma}
The space ${H_2}$ is a Hilbert space, the operator ${\mathsf A}_2$ is an isometry  and the inclusion ${H_2}\subset {H_1}$ is continuous and dense. It entails that 
the inclusion $H_{-1}\subset H_{-2}$, where $H_{-2}$ stands for the dual space of $H_2$, is continuous and dense as well.
\end{lemma}
%
\begin{proof}
The estimate below is satisfied by every $u\in {H_2}$:
\begin{equation}
\label{Cauchy}
\|u\|_0 \leqslant \mathbf c  \|{\mathsf A}_2u\|_0.
\end{equation}
Indeed, the continuity of the inclusion \eqref{first_Gel} yields $\|u\|_0^2\leqslant \mathbf c \|u\|_1^2$. Then, the definitions of both the operator 
${\mathsf A}_2$ and the space $H_2$ lead to the identity $\|u\|_1^2=({\mathsf A}_2u,u)_0$. Applying Cauchy-Schwarz inequality, we obtain \eqref{Cauchy}.
\par
Let assume that $(u_n)_{n\geqslant 0}$  is Cauchy sequence in ${H_2}$, or equivalently that $({\mathsf A}_2u_n)_{n\geqslant 0}$ is a Cauchy sequence in $H_0$, and
denote by $v^\ast$ the limit of $({\mathsf A}_2u_n)_{n\geqslant 0}$ in $H_0$. According to \eqref{Cauchy}, we deduce that  
$(u_n)_{n\geqslant 0}$ is a Cauchy sequence in $H_0$ as well. The equality:
$$ ({\mathsf A}_2u_n-{\mathsf A}_2u_m,u_n-u_m)_0=\|u_n-u_m\|_{1}^2,$$
available for every pair of indices $n$ and $m$, entails that the sequence $(u_n)_{n\geqslant 0}$ is also a Cauchy sequence in ${H_1}$. We denote by $u^\ast$ its limit in this space. Letting $n$ goes to $\infty$ in the identity:
$$({\mathsf A}_2 u_n,\cdot)_0={\mathsf A}_1 u_n\quad\text{in }H_{-1},$$
%
%
we obtain:
$$(v^\ast,\cdot)_0={\mathsf A}_1 u^\ast,$$
and therefore $u^\ast$ belongs to ${H_2}$. This proves that ${H_2}$ is complete and hence is a Hilbert space.
\par
Let now $v$ be in $H_2^{\perp}$ in ${H_1}$. There exists $u\in H_2$ such that ${\mathsf A}_{2}u=v$ and:
$$\|v\|_0^2=({\mathsf A}_2u,v)_0=(u,v)_1=0.$$
It follows that $H_2^{\perp}=0$ in $H_1$ and therefore   ${H_2}$ is dense in ${H_1}$.
\par
The continuity of the inclusion $H_2\subset H_1$ results from the identity:
$$\|u\|_1^2=({\mathsf A}_2 u,u)_0,\quad\text{for all }u\in H_2,$$
combined with Cauchy-Schwarz inequality:
$$({\mathsf A}_2 u,u)_0\leqslant \|{\mathsf A}_2u\|_0\|u\|_0,,\quad\text{for all }u\in H_2,$$
and estimate \eqref{Cauchy}.
\par
Finally, the operator ${\mathsf A}_2$ is onto by definition and it is also injective because the identity ${\mathsf A}_2u=0$ for some $u\in {H_2}$ leads to $({\mathsf A}_2u,u)_0=\|u\|^2_1=0$. 
The proof of the lemma is now completed. 
\end{proof}
%
Let us define the operator ${\mathsf A}_0:H_0\to {H_{-2}}$ by:
\begin{equation}
\label{def_A0}
{\mathsf A}_0:u\in H_0\mapsto ({\mathsf A}_2\cdot,u)_0\in {H_{-2}}.
\end{equation}
%
\begin{lemma}
The operator ${\mathsf A}_0$ is an isometry.
\end{lemma}
\begin{proof}
The operator ${\mathsf A}_0$ is  injective. Indeed, the identity ${\mathsf A}_0u=0$ for some $u\in H_0$ entails that $({\mathsf A}_2{\mathsf A}_2^{-1}u,u)_0=\|u\|_0^2=0$. The operator 
${\mathsf A}_0$ is also onto: Any element of $H_{-2}$ can be written, according to Riesz theorem, as $(\cdot,v)_2=({\mathsf A}_2\cdot,{\mathsf A}_2v)_0$ for some $v\in {H_2}$, and 
hence it is equal to ${\mathsf A}_0 u$ with $u={\mathsf A}_1v\in H_0$.  
\par
Finally, the operator 
${\mathsf A}_0$ is also an isometry since we have, for every $u\in H_0$:
$$\|{\mathsf A}_0u\|_{-2}=\sup_{v\in {H_2}\atop v\neq 0}\frac{|({\mathsf A}_2v,u)_0|}{\|v\|_2}=\sup_{v\in {H_2}\atop v\neq 0}\frac{|({\mathsf A}_2v,u)_0|}{\|{\mathsf A}_2v\|_0}
=\sup_{w\in H_0\atop w\neq 0}\frac{|(w,u)_0|}{\|w\|_0}=\|u\|_0,$$
and the proof is completed.
\end{proof}
So far, we have proved that in the chain of inclusions:
$$H_2\subset H_1\subset H_0\subset H_{-1}\subset H_{-2},$$
every inclusion is continuous and dense and that the operators ${\mathsf A}_k:H_k\to H_{k-2}$ for $k=0,1,2$ are isometries.
\par
By induction, we can next define ${H_{k+2}}={\mathsf A}^{-1}_{k+1}{H_k}$ for every positive integer $k$. The operator 
$${\mathsf A}_{k+2}:{H_{k+2}}\longrightarrow {H_{k}}$$ 
is defined from the operator ${\mathsf A}_{k+1}$ by setting ${\mathsf A}_{k+2}u={\mathsf A}_{k+1}u$ for every $u\in {H_{k+2}}$.
The spaces ${H_{k+2}}$  are Hilbert spaces once 
equipped with the scalar products:
$$(u,v)_{k+2}=({\mathsf A}_{k+2}u,{\mathsf A}_{k+2}v)_{k},\qquad \text{for all }u,v\in {H_{k+2}}.$$
For every $k\geqslant 1$, the dual space of ${H_{k}}$ is denoted by ${H_{-k}}$ and we introduce the operator 
$${\mathsf A}_{-k}:{H_{-k}}\longrightarrow {H_{-k-2}},$$
defined by duality as follows:
\begin{equation}
\label{eq:def_Akdual}
{\mathsf A}_{-k} u = \langle u,{\mathsf A}_{k+2}\cdot\rangle_{-k,k}\in {H_{-k-2}},\qquad \text{for all }u\in {H_{-k}}.
\end{equation}
It can be readily verified that the  Hilbert spaces $H_k$ ($k\in\mathbb Z$) satisfy:
$$\ldots \subset H_{k+1}\subset H_{k}\subset {H_{k-1}}\subset \ldots \subset H_{1}
\subset H_0\subset H_{-1}\subset \ldots \subset {H_{-k+1}}\subset H_{-k}\subset H_{-k-1}\subset \ldots$$
each inclusion being continuous and dense. Furthermore, for every integer $k$, the operator:
$${\mathsf A}_k:H_k\longrightarrow H_{k-2},$$
is an isometry.
%
\begin{lemma}
For every integers $n,n'$ such that $n'\leqslant n$ and for every $u\in H_n$, the following equality holds:
\begin{equation}
\label{Akexpand}
{\mathsf A}_n u={\mathsf A}_{n'}u.
\end{equation}
\end{lemma}
\begin{proof}
For $n'\geqslant 0$, the property \eqref{Akexpand} is obvious. 
\par
%
%
On the other hand, let $0\leqslant k'\leqslant k$ be given and assume that $u\in H_{-k'}\subset H_{-k}$. The definition of ${\mathsf A}_{-k}u$ leads to:
$${\mathsf A}_{-k}u=\langle u,{\mathsf A}_{k+2}\cdot\rangle_{-k,k}\in H_{-k-2}.$$
But ${\mathsf A}_{k+2}={\mathsf A}_{k'+2}$ in $H_{k+2}$ and therefore:
$$\langle u,{\mathsf A}_{k+2}\cdot\rangle_{-k,k}=\langle u,{\mathsf A}_{k'+2}\cdot\rangle_{-k',k'}={\mathsf A}_{-k'}u.$$
The proof is now complete.
\end{proof}
\begin{lemma}
For every integer $k$, the following identity hold:
\begin{equation}
\label{main_def_AK}
({\mathsf A}_{k+1}u,v)_{k-1}=(u,v)_k,\quad\text{for all }u\in H_{k+1},\quad\text{for all }v\in H_k.
\end{equation}
\end{lemma}
\begin{proof}
The proof is by induction on $k$. The equality \eqref{main_def_AK} is true for $k=1$ according to the definition \eqref{defA2} of ${\mathsf A}_2$.
Let us assume that \eqref{main_def_AK} is true for some integer $k$. By definition, if $z$ belongs to $H_{k+2}$, then 
${\mathsf A}_{k+1}z$ belongs to $H_k$. Replacing $v$ by ${\mathsf A}_{k+1}z$ in \eqref{main_def_AK}, we obtain:
$$({\mathsf A}_{k+1}u,{\mathsf A}_{k+1}z)_{k-1}=(u,{\mathsf A}_{k+1}z)_k,\quad\text{for all }u\in H_{k+1},\quad\text{for all }z\in H_{k+2},$$
that is to say, reorganizing the terms:
$$({\mathsf A}_{k+2}z,u)_k=(z,u)_{k+1},\quad\text{for all }u\in H_{k+1},\quad\text{for all }z\in H_{k+2},$$
and therefore, formula \eqref{main_def_AK} is true replacing $k$ by $k+1$. Let us verify that it is also true for $k-1$. Thus, we have:
$$({\mathsf A}_{k+1}u,v)_{k-1}=(u,v)_k=({\mathsf A}_k u,{\mathsf A}_kv)_{k-2},\quad\text{for all }u\in H_{k+1},\quad\text{for all }v\in H_k.$$
But ${\mathsf A}_{k+1}$ is an isometry from $H_{k+1}$ onto $H_{k-1}$ and ${\mathsf A}_k={\mathsf A}_{k+1}$ in $H_{k+1}$, then:
$$(w,v)_{k-1}=(u,v)_k=(w,{\mathsf A}_kv)_{k-2},\quad\text{for all }w\in H_{k-1},\quad\text{for all }v\in H_k.$$
The proof is now complete.
\end{proof}
%
\begin{lemma}
\label{def_DA}
Let $k$ be an integer, $w$ be in $H_{k-1}$ and $u$ be in $H_k$ such that:
$$(w,v)_{k-1}=(u,v)_k,\quad\text{for all }v\in H_k.$$
Then $u\in H_{k+1}$ and $w={\mathsf A}_{k+1}u$. 
\end{lemma}
\begin{proof}
Let $\tilde u = {\mathsf A}_{k+1}^{-1}w$. Then $(\tilde u-u,v)_k=0$ for every $v\in H_k$ and therefore $u=\tilde u$.
\end{proof}
%
%
%
%
\subsection{Isometric chain of embedded Hilbert spaces}
\label{isometric_chain}
Let $\{H_k,\,k\in\mathbb Z\}$ and $\{\hat H_k,\,k\in\mathbb Z\}$ be two families of embedded Hilbert spaces build from Gelfand triples. 
We assume that $H_0$ and $\hat H_0$ are not necessary the pivot spaces. 
As usual, for every integer $k$, there exist 
isometries ${\mathsf A}_k:H_k\to H_{k-2}$ and $\hat {\mathsf A}_k:\hat H_k\to \hat H_{k-2}$ such that ${\mathsf A}_k={\mathsf A}_{k-1}$ in $H_k$ and $\hat {\mathsf A}_k=\hat {\mathsf A}_{k-1}$ in $\hat H_k$. 
\par
We assume 
furthermore that there exist  isometries $p_0:H_0\to \hat H_0$ and $p_1:H_1\to \hat H_1$ such that $p_1=p_0$ in $H_1$. For every integer $k\geqslant 2$, 
we define by induction $p_k=\hat {\mathsf A}_{k}^{-1} p_{k-2} {\mathsf A}_k$ and for every $k\geqslant 1$, we set $p_{-k-2}=\hat {\mathsf A}_{-k} p_{-k} {\mathsf A}_{-k}^{-1}$.
\begin{lemma}
\label{p_kp_k}
For every pair of integers $k$ and $k'$ such that $k'\leqslant k$:
\begin{equation}
\label{pkpk}
p_{k'}=p_k\quad\text{in }H_k.
\end{equation}
Moreover, for every integer $k$, the operator $p_k:H_k\to \hat H_k$ is an isometry.
\end{lemma}
\begin{proof}
Since ${\mathsf A}_k$ and $\hat {\mathsf A}_k$ are isometries for every integer $k$, we can draw the same conclusion for $p_k$ providing that $p_{k-2}$ is an isometry as well. The conclusion 
follows by induction for every $k\geqslant 0$. The same reasoning allows proving that $p_{-k}$ is also an isometry for every $k\geqslant 1$.
\par
It remains to verify that the equalities \eqref{pkpk} are true. 
Assume that for some index $k\geqslant 0$, $p_k=p_{k+1}$ in $H_{k+1}$. So, from the identity:
$$({\mathsf A}_{k+2}u,v)_{H_k}=(u,v)_{H_{k+1}},\forallt u\in H_{k+2},\quad\text{for all }v\in H_{k+1},$$
we deduce that:
$$(p_k {\mathsf A}_{k+2}u,p_k v)_{\hat H_k}=(p_{k+1} u,p_{k+1} v)_{\hat H_{k+1}},\forallt u\in H_{k+2},\quad\text{for all }v\in H_{k+1}.$$
From the definition of $p_{k+2}$, we deduce that $p_k {\mathsf A}_{k+2}=\hat {\mathsf A}_{k+2}p_{k+2}$, whence, denoting ${\mathbf v}=p_{k}v=p_{k+1}v$:
$$(\hat {\mathsf A}_{k+2}p_{k+2}u,{\mathbf v})_{\hat H_k}=(p_{k+1} u,{\mathbf v})_{\hat H_{k+1}},\forallt u\in H_{k+2},\quad\text{for all }{\mathbf v}\in \hat H_{k+1}.$$
This equality entails first that $\hat {\mathsf A}_{k+2}p_{k+2}u=\hat {\mathsf A}_{k+2}p_{k+1}u$ and next, since $\hat {\mathsf A}_{k+2}$ is invertible, that $p_{k+2}u=p_{k+1}u$ for every $u\in H_{k+2}$. The conclusion follows by induction
and the cases $k\leqslant 0$ are treated similarly.
\end{proof}
%
\subsection{Semigroup}
\label{SUB:Semigroup}
We assume that the inclusion \eqref{first_Gel} 
is in addition compact. In that case, we claim:
%
\begin{lemma}
For every integer $k$, the inclusion $H_{k+1}\subset H_k$ is compact.
\end{lemma}
\begin{proof}
We address the case $k=1$. 
Assume that the sequence $(u_n)_{n\geqslant 0}$ is weakly convergent toward $0$ in ${H_2}$. On the one hand, it means that:
$$(u_n,v)_2=({\mathsf A}_2 u_n,{\mathsf A}_2v)_0\longrightarrow 0\text{ as }n\to+\infty\forallt v\in {H_2},$$
and therefore, that:
$$({\mathsf A}_2 u_n,w)_0\longrightarrow 0\text{ as }n\to+\infty\forallt w\in H_0.$$
That is, $({\mathsf A}_2u_n)_{n\geqslant 0}$ is weakly convergent toward $0$ in $H_0$. On the other hand,  since ${H_2}$ is 
continuously included into ${H_1}$, the sequence $(u_n)_{n\geqslant 0}$ is also weakly convergent toward $0$ in ${H_1}$ and hence strongly convergent in $H_0$. It follows that:
$$\|u_n\|_{1}^2=({\mathsf A}_2u_n,u_n)_0\longrightarrow 0\text{ as }n\to+\infty.$$
Since the operators $\mathsf A_k$ were proved to be isometries for every $k$, the other cases follows and the proof is completed.
\end{proof}
%
For every integer $k$, we define the unbounded operators $\mathcal A_k$ of domain $D(\mathcal A_k)=H_{k+2}$ in $H_k$ by:
\begin{equation}
\label{def:unboundAk}
\mathcal A_k x= \mathsf A_{k+2}x\forallt x\in D(\mathcal A_k).
\end{equation}
\begin{prop}
For every integer $k$, the operator $\mathcal A_k$ is self-adjoint with compact inverse. All the operators $\mathcal A_k$ share the same 
spectrum that consists in a sequence $(\lambda_n)_{n\geqslant 1}$ of positive eigenvalues that tends to $+\infty$. All the operators $\mathcal A_k$ 
share also the same eigenfunctions, denoted by $e_n$ ($n\geqslant 1$) and for every nonnegative integer $n$:
$$e_n\in H_\infty\qquad\text{with}\qquad H_\infty= \bigcap_{p\geqslant 0} H_p.$$
The eigenfunctions are chosen to form an orthogonal  Riesz basis in every $H_k$ and they are scaled to be of unit norm in $H_0$.
\end{prop}
%
The spaces $H_k$ are isometric to the spaces:
$$\ell_k=\Big\{u=(u_n)_{n\geqslant 1}\in\mathbb R^{\mathbb N}\,:\,\sum_{n\geqslant 1} \lambda^k_n u_n^2<+\infty\Big\},$$
provided with the scalar product:
$$(u,v)_{\ell_k}=\sum_{n\geqslant 1} \lambda_n^k u_n v_n\forallt u=(u_n)_{n\geqslant 1}\quad\text{ and }\quad v=(v_n)_{n\geqslant 1}\text{ in } \ell_k,$$
the isometries being obviously: 
$$\mathcal I_k:u\in H_k\mapsto ((u,e_n)_k)_{n\geqslant 1}\in\ell_k\quad\text{with inverse}\quad \mathcal I^{-1}_k:u=(u_n)_{n\geqslant 1}
\in\ell_k\longmapsto 
\sum_{n=1}^{+\infty}u_n e_n\in H_k.$$
In $\ell_k$ we define the strongly continuous semigroup of contraction $(\mathbb T_k(t))_{t\geqslant 0}$ by:
$$\mathbb T_k(t)u=\Big(e^{-\lambda_n t}u_n\Big)_{n\geqslant 0}\forallt t>0\text{ and }u=(u_n)_{n\geqslant 0}\in\ell_k.$$
This semigroup admits the operator $\mathcal B_k=\mathcal I_k\mathcal A_k\mathcal I^{-1}_k$ as infinitesimal generator. 
We deduce that the semigroup $(\mathbb S_k(t))_{t\geqslant 0}$ defined by:
$$\mathbb S_k(t)=\mathcal I^{-1}_k\mathbb T(t)\mathcal I_k,$$
is a strongly continuous semigroup of contraction in $H_k$ with infinitesimal generator $\mathcal A_k$. It is a simple exercice to verify that:
\begin{lemma}
\begin{enumerate}
\item 
For every $u\in \ell_k$ and for every positive real number $T$:
$$\mathbb T_k(\cdot)u\in H^1(0,T;\ell_{k-1})\cap \mathcal C([0,T];\ell_k)\cap L^2(0,T;\ell_{k+1}).$$
\item
Let $v$ be in $L^2(0,T;\ell_{k-1})$ and define $w(s)=\int_0^t \mathbb T_k(t-s)v(s)\ds$ for every $t\in(0,T)$. Then:
$$w\in H^1(0,T;\ell_{k-1})\cap \mathcal C([0,T];\ell_k)\cap L^2(0,T;\ell_{k+1}).$$
\end{enumerate}
\end{lemma}
%
Considering, for any integer $k$, any time $T>0$ and any initial data $u^{\rm i}\in H_k$ the Cauchy problem:
\begin{subequations}
\label{abstrc_cauchy}
\begin{alignat}{3}
\partial_t u+\mathsf A_{k+1}u&=f&\quad&\text{on }(0,T)\\
u(0)&=u^{\rm i}&&\text{in }H_k,
\end{alignat}
\end{subequations}
where the source term $f$ is given in $L^2(0,T;H_{k-1})$, we deduce:
\begin{prop}
\label{prop:cauchy_abstract}
The Cauchy problem \eqref{abstrc_cauchy} admits a unique solution in the space:
$$H^1(0,T;H_{k-1})\cap \mathcal C([0,T];H_k)\cap L^2(0,T;H_{k+1}),$$
and this solution is given by:
$$u(t)=\mathbb S_k(t)u^{\rm i}+\int_0^t \mathbb S_k(t-s)f(s)\ds\forallt t\in[0,T].$$
\end{prop}
%
\begin{rem}
\begin{enumerate}
\item
The chain of embedded spaces $H_k$ and semigroup $(\mathbb S_k(t))_{t\geqslant 0}$ fit with the 
notion of {\it Sobolev towers} as described in \cite[\S II.2.C]{Engel:2006aa}.
\item The semigroups $(\mathbb S_k(t))_{t\geqslant 0}$ are called diagonalizable semigroups; see  \cite[\S 2.6]{Tucsnak:2009aa}.
\end{enumerate}
\end{rem}
\end{document}